\documentclass[12pt]{amsart}
 \usepackage{amsfonts,amssymb}

 \usepackage{color}
 \usepackage{array}
\usepackage{hyperref}
\usepackage[all,ps,cmtip,rotate]{xy} 
\usepackage{stmaryrd,accents}

\usepackage{mathrsfs}

\usepackage{bm}

\def\Z{{\bf Z}}

\def\C{{\bf C}}

\def\Q{{\bf Q}}
\def\P{{\bf P}}

\def\ppavs{principally polarized abelian varieties}

 \def\phi{{\varphi}}

\def\cI{\mathscr{I}}
\def\cJ{\mathscr{J}}
\def\cD{\mathscr{D}}

\def\cA{\mathscr{A}}
\def\cAb{\mathbf{A}}

\def\cF{\mathscr{F}}
\def\cL{\mathscr{L}}
\def\cO{\mathscr{O}}

\def\cH{\mathscr{H}}
\def\cE{\mathscr{E}}

\def\cC{\mathscr{C}}
\def\cR{\mathscr{R}}
\def\cM{\mathscr{M}}

\def\cK{\mathscr{K}}

\def\cU{\mathscr{U}}
\def\cV{\mathscr{V}}
\def\cW{\mathscr{W}}
\def\cX{\mathscr{X}}

\DeclareMathOperator{\SEPW}{\Sigma^{\mathrm{EPW}}}
\DeclareMathOperator{\bEPW}{\mathbf{\overline{M}}{}^{\mathrm{EPW}}}
\DeclareMathOperator{\EPW}{\mathbf{M}^{\mathrm{EPW}}}

\def\bmu{{\boldsymbol\mu}}

\def\ns{S}

\def\k{\mathbf k}
\def\K{\mathbf K}

\def\bm{\mathbf m}
\def\bq{\mathbf q}

\def\fa{\mathfrak a}

\def\lra{\longrightarrow}
\def\llra{\hbox to 10mm{\rightarrowfill}}
\def\lllra{\hbox to 15mm{\rightarrowfill}}

\def\llla{\hbox to 10mm{\leftarrowfill}}
\def\lllla{\hbox to 15mm{\leftarrowfill}}

\def\dra{\dashrightarrow}

\def\hra{\hookrightarrow}

\def\lhra{\ensuremath{\lhook\joinrel\relbar\joinrel\rightarrow}}

\newcommand{\cWn}{\widehat\cW}
\newcommand{\mun}{{\widehat\mu}}
\newcommand{\bqn}{{\widehat\bq}}

\newcommand{\gmspe}{{\rm GM, spe}}
\newcommand{\lagspe}{{\rm Lag, spe}}
\newcommand{\gmord}{{\rm GM, ord}}
\newcommand{\lagord}{{\rm Lag, ord}}
 
\def\wpepw{\wp^{\mathrm{EPW}}}
\def\wpgm{\wp^{\mathrm{GM}}}

\def\isom{\simeq}

 \def\vide{\varnothing}
  \def\emptyset{\varnothing}
\def\k{\mathbf k}

\DeclareMathOperator{\isomlra}{\stackrel{{}_{\scriptstyle\sim}}{\lra}}

\DeclareMathOperator{\isomto}{\isomlra}

\DeclareMathOperator{\Aut}{Aut}

\DeclareMathOperator{\codim}{codim}
\DeclareMathOperator{\Cone}{Cone}
\DeclareMathOperator{\Coker}{Coker}

\DeclareMathOperator{\Tor}{Tor}
\DeclareMathOperator{\GL}{GL}
\DeclareMathOperator{\Gr}{\mathsf{Gr}}
\DeclareMathOperator{\LGr}{\mathsf{LGr}}
\DeclareMathOperator{\LGradv}{\mathsf{LGr}_{\mathrm{adv}}}

\DeclareMathOperator{\CGr}{\mathsf{CGr}}

\DeclareMathOperator{\Hom}{Hom}

\def\Im{\mathop{\rm Im}\nolimits}

\DeclareMathOperator{\Ker}{Ker}

\DeclareMathOperator{\PGL}{PGL}

\DeclareMathOperator{\Pic}{Pic}

\DeclareMathOperator{\Spec}{Spec}

\DeclareMathOperator{\Sing}{Sing}

\DeclareMathOperator{\Sym}{\mathrm{Sym}}
\DeclareMathOperator{\SL}{SL}
\newcommand{\Sch}{{\mathrm{Sch}}}

\def\bw#1#2{\textstyle{\bigwedge\hskip-0.9mm^{#1}}\hskip0.2mm{#2}}
\def\sbw#1#2{\small{\bigwedge\hskip-0.9mm^{#1}}\hskip0.2mm{#2}}

\def\Gm{\mathbb{G}_{\mathrm{m}}}

\def\hhf{f}

\newtheorem{lemm}{Lemma}[section]
\newtheorem{theo}[lemm]{Theorem}
\newtheorem{coro}[lemm]{Corollary}
\newtheorem{prop}[lemm]{Proposition}

\theoremstyle{remark}
\newtheorem{defi}[lemm]{Definition}
\newtheorem{rema}[lemm]{Remark}

\newtheorem{exam}[lemm]{Example}

\def\id{\mathsf{id}}

\DeclareMathOperator{\Dis}{
{Disc}}

\newcommand{\tcE}{{\widetilde{\cE}}}
\newcommand{\tcK}{{\widetilde{\cK}}}

\newcommand{\ttcE}{{\widehat{\cE}}}

\newcommand{\diag}{\mathop{\mathrm{diag}}}

\newcommand{\gquot}{{/\!\!/}}

\DeclareMathOperator{\Bl}{{\mathrm{Bl}}}

\newcommand{\ucM}{{{\mathfrak{M}}}^{\mathrm{GM}}}
\newcommand{\ucMtwo}{{{\mathfrak{M}}}^{\mathrm{GM}}_{\mathrm{2,ord,ss}}}
\newcommand{\bcM}{{{\mathbf{M}}}^{\mathrm{GM}}}
\newcommand{\bcMtwo}{{{\mathbf{M}}}^{\mathrm{GM}}_{\mathrm{2,ord,ss}}}

\newcommand{\ucMd}{{{\mathfrak{M}}}^{\mathrm{GM\text{-}data}}}
\newcommand{\tucMd}{{{\widetilde{\mathfrak{M}}}}{}^{\mathrm{GM\text{-}data}}}

\newcommand{\cmlag}{{{\mathfrak{M}}}^{\mathrm{Lag}}}
\newcommand{\bcmlag}{{{\mathbf{M}}}^{\mathrm{Lag}}}
\newcommand{\tcmlag}{{{\widetilde{\mathfrak{M}}}}{}^{\mathrm{Lag}}}

\def\setminus{\smallsetminus}
\def\cong{\isom}

\newcommand{\rG}{\mathrm{G}}
\newcommand{\bS}{\mathbf{S}}
\newcommand{\rS}{\mathrm{S}}
\newcommand{\wwrS}{\widehat{\vphantom{\rule{1pt}{10pt}}\smash{\widehat{\rS}}}} 
\newcommand{\hhrS}{{\wwrS}}

\newcommand{\brS}{{\overline{\mathrm{S}}}}
\newcommand{\bbrS}{{\overline{\overline{\mathrm{S}}}}}
\newcommand{\hrS}{{\widehat{\mathrm{S}}}}


\begin{document}
\title[Gushel--Mukai varieties: moduli]
{Gushel--Mukai varieties: moduli}

  \author[O. Debarre]{Olivier Debarre}
\address{Universit\'e Paris-Diderot, PSL Research University, CNRS, 
\'Ecole normale sup\'erieure, D\'epartement de  Math\'ematiques et Applications,
45 rue d'Ulm, 75230 Paris cedex 05, France}
\email{{\tt olivier.debarre@ens.fr}}
 \author[A. Kuznetsov]{Alexander Kuznetsov}
 \address{Algebra Section, Steklov Mathematical Institute,
  8 Gubkin str., Moscow 119991 Russia
 \\ 
 The Poncelet Laboratory, Independent University of Moscow
\\Laboratory of Algebraic Geometry, SU-HSE, 7 Vavilova Str., Moscow, Russia, 117312 }
 \email{{\tt  akuznet@mi.ras.ru}}
 
\thanks{A.K. was supported by the Russian Academic Excellence Project~\mbox{``5--100''}.}


\maketitle

\begin{abstract}
 {We describe the moduli stack of Gushel--Mukai varieties as a global quotient stack 
and its coarse moduli space as the corresponding GIT quotient.\ The construction is based on a comprehensive study of the relation between this stack and the stack of Lagrangian data (as defined in~\cite[Section~3]{DK}); 
roughly speaking, we show that the former is a generalized root stack of the latter.\ As an application, we define the period map for Gushel--Mukai varieties and construct
some complete nonisotrivial families of smooth Gushel--Mukai varieties.\ In an appendix, we describe a generalization of the root stack construction used in our approach to the moduli space.}
\end{abstract}

{\renewcommand{\baselinestretch}{0.5}\normalsize
\tableofcontents}

\section{Introduction}

This article is the third in the series started in~\cite{DK,DK:periods} and devoted to the investigation of 
{\sf Gushel--Mukai (GM) varieties} defined over a  field~$\k$
 of  characteristic zero.\ 
These varieties are positive-dimensional, dimensionally  transverse intersections
\begin{equation*}
X = \CGr(2,V_5) \cap \P(W) \cap Q,
\end{equation*}
where $\CGr(2,V_5)  \subset  \P(\k \oplus \bw2V_5) $ is the cone over the Grassmannian of two-dimensional vector subspaces 
in a $\k$-vector space $V_5$ of dimension~5,
$\P(W) \subset \P(\k \oplus \bw2V_5)$ is a projective space of dimension~\mbox{$n + 4$}, 
and $Q$ is a quadric hypersurface in $\P(W)$.\ 
We have  
\begin{equation*}
\dim (X)=n\in\{1,\dots,6\}.
\end{equation*}
Various geometric characterizations of GM varieties can be found in~\cite[Section~2.3]{DK};
for instance,~\cite[Theorem~2.16]{DK}   shows that   smooth GM varieties of dimension~\mbox{$n \ge 3$}
are exactly the  Fano varieties of Picard rank~1, coindex~3, and degree~10.
 
In~\cite{DK}, we  described the set of isomorphism classes of all GM varieties.\ 
In particular, we associated with each GM variety what we called a {\sf GM data set}.\ 
Roughly speaking,  
it is a collection $(W,V_6,V_5,\mu,\bq)$, where  $W$, $V_6$, and $V_5$ are  {$\k$}-vector spaces 
of respective dimensions~$n + 5$, 6, and~5, with $V_5 \subset V_6$, and
\begin{equation*}
\mu \colon W \to \bw2V_5
\qquad\text{and}\qquad
\bq \colon V_6 \to \Sym^2\!W^\vee 
\end{equation*}
are   $\k$-linear maps.\ The   map $\mu$ is the composition the embedding $W \hookrightarrow \k \oplus \bw2V_5$ coming from the definition of $X$ with the projection onto the second summand,
whereas the   map $\bq$ is obtained by identifying $V_6$ with the space of quadratic equations of $X$ in $\P(W)$.\ Under this identification, the hyperplane $V_5 \subset V_6$ corresponds to the space of Pl\"ucker quadrics
defining~$\CGr(2,V_5)$ in~$\P(\k \oplus \bw2V_5)$.

There are two types of smooth GM varieties: ordinary and special, 
distinguished by the injectivity or  the noninjectivity of the map $\mu$.\ 
 {When the   field $\k$ is quadratically closed,} 
there is a natural bijection between the set  of isomorphisms classes of special GM varieties of dimension~$n$
and the set of ordinary GM varieties of dimension~$n-1$ (\cite[Lemma~2.33]{DK}).\ 
On the other hand, special GM varieties can be obtained from ordinary GM varieties  {of the same dimension} by a specialization 
(except in the case $n = 6$, when there are no ordinary GM varieties).\ 
However, they behave in a slightly different way and provide various complications to the theory.

The first main result of~\cite{DK}, Theorem~2.9, provides a bijection between the set of isomorphism classes
of GM varieties of dimension $n$ and an appropriate subset of the set of isomorphism classes of GM data sets.\ After introducing in Sections~\ref{subsection:stack-gm} and~\ref{subsection:stack-gm-data} the stacks
of GM varieties and GM data, we present in Section~\ref{subsection:gm=gm-data} a version of this bijection
that works for families and promotes the bijection of sets of isomorphism classes to an isomorphism of moduli stacks (see~Theorem~\ref{theorem:gm-substack-gm-data}).

The second main result of~\cite{DK}, Theorem~3.6, relates GM data sets of ordinary GM varieties to so-called {\sf Lagrangian data sets}.\ These consist of triples $(V_6,V_5,A)$, where  $V_6$ is a vector space of dimension~6,
$V_5 \subset V_6$ is a hyperplane, and $A \subset \bw3V_6$ is a  subspace which is Lagrangian for   the  
$\det(V_6)$-valued symplectic form on $\bw3V_6$ given by   wedge product.\ Theorem~3.6  {of~\cite{DK}} establishes a bijection between ordinary GM data sets of dimension~$n$ and 
Lagrangian data sets such that $\dim(A \cap \bw3V_5) = 5 - n$, as well as, 
(if $\k$ is quadratically closed)
 {using the bijection} between ordinary and special GM data sets,
a bijection between special GM data sets of dimension~$n$ and  Lagrangian data sets such that \mbox{$\dim(A \cap \bw3V_5) = 6 - n$}.

A nice feature of this bijection, proved in~\cite[Theorem~3.16]{DK}, is that the smoothness of the GM variety
associated with a Lagrangian data set $(V_6,V_5,A)$ only depends on $A$: when~\mbox{$n \ge 3$}, the corresponding
GM variety is smooth if and only if $A$ {\sf has no decomposable vectors}, that is, 
\begin{equation*}
\P(A) \cap \Gr(3,V_6) = \varnothing,
\end{equation*}
the intersection being taken inside $\P(\bw3V_6)$.

The main goal of the present article is to combine all these constructions into a single construction
of the moduli stack of smooth GM varieties.\ In other words, we   find analogs of the above constructions that   work for ``mixed'' families (with both ordinary and special varieties as members).\ The main difficulty   is  that the GM/Lagrangian data sets bijection of~\cite{DK} does not work with these
  ``mixed'' families.\
Nevertheless, we define in Section~\ref{subsection:lagrangian-data} the stack of Lagrangian data
that classifies all Lagrangian data sets $(V_6,V_5,A)$ such that
\begin{equation*}
\dim(A \cap \bw3V_5)  \in \{ 5 - n, 6 - n \} 
\end{equation*}
and such that $A$ has no decomposable vectors.\ We observe in Section~\ref{subsection-from-gm-to-lag} (see Proposition~\ref{proposition-gm-to-lag}) 
that the natural family version of the construction from~\cite[Theorem~3.6]{DK}
that associates with a family of GM data $(\ns,\cW,\cV_6,\cV_5,\mu,\bq)$ over a scheme $S$ a family of Lagrangian data~$(\ns,\cV_6,\cV_5,\cA)$ 
is still well defined and  {gives} a morphism of stacks.\ However, this morphism cannot be an isomorphism for the following simple reason.

For any family of GM data $(\ns,\cW,\cV_6,\cV_5,\mu,\bq)$ over a scheme $S$, we define in Section~\ref{subsection:ord-spe-substacks}
a closed subset $S_\gmspe \subset S$ corresponding to points of $S$ that parameterize GM data sets of special varieties
and endow it with a natural scheme structure (Definition~\ref{definition:gm-special-locus}).\ 
Similarly, given a family of Lagrangian data $(\ns,\cV_6,\cV_5,\cA)$, we consider the closed subset $S_\lagspe \subset S$ 
corresponding to points of  $S$ such that $\dim(A \cap \bw3V_5) = 6 - n$  and
endow it with a natural scheme structure (Definition~\ref{definition:lagrangian-special-locus}).\ An important consequence of Proposition~\ref{proposition-gm-to-lag} is that  although  
the special loci of an $S$-family of GM data and of the associated $S$-family of Lagrangian data are the same set-theoretically,
they have different scheme  structures: the ideal of $S_\lagspe$ is the square of the ideal of $S_\gmspe$.\  {Consequently}, if we start with a family of Lagrangian data  such that the ideal of its special locus
is not a square, there is no  corresponding family of GM data!

However, we prove in Section~\ref{subsection-from-lag-to-gm} that the inverse construction can be made when 
   the ideal of the subscheme $S_\lagspe \subset S$ for an $S$-family of Lagrangian data 
     is a square and $S_\lagspe$ is a Cartier divisor in $S$  (this second condition seems to be of technical nature but we do not know how to make the inverse construction  without  it).\ This   is the central construction of the article.\ It is based on two vector bundle constructions which we develop in Section~\ref{section:preliminaries}
and which are interesting by themselves.

The first   is the \emph{canonical factorization construction} of Proposition~\ref{proposition-factorization}: given a morphism of vector bundles $\varphi \colon \cE \to \cF$ of generic rank $r$ over a scheme $S$, 
such that the rank of~$\varphi$ is everywhere at least $r - 1$ 
and   the degeneration scheme of $\varphi$ is a Cartier divisor~$D \subset S$,
we construct a canonical factorization 
\begin{equation*}
\cE \twoheadrightarrow \cE_1 \xrightarrow{\ \varphi_1\ } \cF_1 \hookrightarrow \cF,
\end{equation*}
where $\cE_1$ and $\cF_1$ are vector bundles of rank $r$, 
the first map is an epimorphism, the last map is a fiberwise monomorphism, 
and the map $\varphi_1$ is an embedding of coherent sheaves whose cokernel is a line bundle on $D$.\

Given a family of Lagrangian data $(\ns,\cV_6,\cV_5,\cA)$ such that $S_\lagspe$ is a Cartier divisor in $S$, 
we apply in Section~\ref{subsection-from-lag-to-gm} this construction to the composition 
\begin{equation*}
\varphi \colon \cA \hookrightarrow \bw3\cV_6 \xrightarrow{\ \lambda_3\ } \bw2\cV_5 \otimes (\cV_6/\cV_5)
\end{equation*}
(where the second map is induced by the natural projection $\lambda \colon \cV_6 \to \cV_6/\cV_5$)  
and obtain a factorization
\begin{equation}\label{phi1}
\cA \twoheadrightarrow \cW' \xrightarrow{\ \varphi_1\ } \cW'' \hookrightarrow \bw2\cV_5 \otimes (\cV_6/\cV_5).
\end{equation}
The cokernel of the morphism $\varphi_1$ is supported on the Cartier divisor $D = S_\lagspe$.\
Assuming that this divisor can be written  {as}
\begin{equation*}
D = 2E,
 \end{equation*}
where $E$ is also a Cartier divisor, we find a unique vector bundle $\cW$  
such that the morphism~$\varphi_1$ factors as 
\begin{equation*}
\cW' \xrightarrow{\ \varphi'\ } \cW \xrightarrow{\ \varphi''\ } \cW'',
\end{equation*}
where both $\varphi'$ and $\varphi''$ are embeddings of sheaves whose cokernels are line bundles on $E$.\ 
We prove in Proposition~\ref{proposition-gm-from-lag-naive} that $(\ns,\cW,\cV_6,\cV_5,\mu,\bq)$, where $\mu$ is   the composition
\begin{equation*}
\cW \xrightarrow{\ \varphi''\ } \cW'' \hookrightarrow \bw2\cV_5 \otimes (\cV_6/\cV_5)
\end{equation*}
and the map $\bq$ will be defined below, is a family of GM data corresponding to a smooth family of GM varieties,
whose associated family of Lagrangian data is equivalent to $(\ns,\cV_6,\cV_5,\cA)$.

The construction of the map $\bq \colon \cV_6 \to \Sym^2\!\cW^\vee$ is carried out in three steps 
(there is actually an extra line bundle twist on the target of $\bq$, but we will ignore it here for simplicity).\ First, we define a map 
\begin{equation*}
\bq_\cA \colon \cV_6 \lra \Sym^2\!\cA^\vee
\end{equation*}
by an explicit formula~\eqref{equation-qa}, which is just a family version of a formula used in the proof of~\cite[Theorem~3.6]{DK}.\ We then observe that the kernel of the epimorphism $\cA \twoheadrightarrow \cW'$ is contained in the kernel of $\bq_\cA$,
hence there is a morphism
\begin{equation}\label{q'}
\bq' \colon \cV_6 \lra \Sym^2\!\cW^{\prime\vee}
\end{equation}
induced by $\bq_\cA$.\ The last step is the  construction of a map 
\begin{equation}\label{q}
\bq \colon \cV_6 \lra \Sym^2\!\cW^\vee
\end{equation}
such that $(\Sym^2\!\varphi^{\prime T}) \circ \bq = \bq'$; it uses the second vector bundle construction from Section~\ref{section:preliminaries}.

This second construction is explained in Section~\ref{subsection:hecke}; we call it the \emph{Hecke transform} of a family of quadratic forms.\ 
It starts from a morphism $\cV \to \Sym^2\!\cE^\vee$ of vector bundles  over a scheme $S$
 {(viewed as a family of quadrics in $\P_S(\cE)$ parameterized by~$\P_S(\cV)$)},
a double Cartier divisor $D = 2E$ on $S$, and a line subbundle $\cK \subset \cE\vert_D$ 
contained in the kernel of all quadratic forms (restricted to $D$).\ 
We define a new vector bundle~$\tcE$ on $S$ by the exact sequence
\begin{equation*}
0 \to \tcE^\vee \xrightarrow{\ \epsilon\ } \cE^\vee \to \cK^\vee\vert_E \to 0
\end{equation*}
 and check in Proposition~\ref{proposition-hecke-transform} 
that there is a unique family of quadratic forms \mbox{$\cV \to \Sym^2\!\tcE^\vee$} such that the original family of quadratic forms is 
the composition of this family with~$\Sym^2\!\epsilon$.

We apply this construction to the family of quadratic forms $\bq'$ from \eqref{q'}, with~$\cV = \cV_6$ and $\cE = \cW'$,
taking $D = S_\lagspe$ and $\cK = \Ker(\varphi_1\vert_D \colon \cW'\vert_D \to \cW''\vert_D)$.\  
The corresponding  Hecke transform  $\widetilde{\cW'}$ is just $\cW$, so Proposition~\ref{proposition-hecke-transform} 
provides the required family $\bq$ of quadratic forms on $\cW$ as in \eqref{q}.\ 
We also prove in Proposition~\ref{proposition-hecke-transform} that 
there is a canonical direct sum decomposition
\begin{equation*}
\cW\vert_E \cong (\cW'\vert_E) / (\cK\vert_E) \oplus (\cK\vert_E)(E)
\end{equation*}
which is   orthogonal for all quadrics in the family $\bq\vert_E$
and which recovers the canonical direct sum decomposition
of~\cite[Proposition~2.30]{DK} for special GM data sets.\ This observation is essential for proving that the constructed family of GM varieties is smooth.

In Section~\ref{section:global}, we use the constructions of Section~\ref{section-relation}
to provide a description of the stack of smooth GM varieties as a global quotient stack.\ We first fix a vector space $V_6$ of dimension~6 and consider the scheme 
\begin{equation*}
\brS_n = \left\{ (A,V_5) \in \LGr(\bw3V_6) \times \P(V_6^\vee) \ \left\vert\ 
\parbox{.41\textwidth}{$   \dim(A \cap \bw3V_5) \in\{ 5 - n,6 - n\}$ and\\[.5ex]$A$ has no decomposable vectors} \right.\right\}.
\end{equation*}
The condition $\dim(A \cap \bw3V_5) \ge 5 - n$ is closed, while the conditions $\dim(A \cap \bw3V_5) \le 6 - n$ 
and ``$A$ has no decomposable vectors'' are open, so $\brS_n \subset \LGr(\bw3V_6) \times \P(V_6^\vee)$ is a locally closed subscheme.\ 
When $n\in\{3,4,5\}$, this scheme contains a closed subscheme $\rS_{n-1}$ and its open complement~$\rS_n$, 
defined by the conditions that $\dim(A \cap \bw3V_5)$ equals $6-n$ and~$5-n$ respectively,
while $\brS_6 = \rS_5$.\
 {The fibers of the projection $\rS_n \to \LGr(\bw3V_6)$ over a point~$[A]$ are just the strata 
of the Eisenbud--Popescu--Walter stratification of $\P(V_6^\vee)$ associated with $A$ (see~\cite[Section 2]{og1} or Section~\ref{subsection:epw})
and the fibers of $\brS_n$ are unions of these strata.\
In particular, the schemes}
$\brS_6$ and $\brS_5$ are smooth, and for $n \in \{3,4\}$, one has~$\rS_{n-1} = \Sing(\brS_n)$
and both strata $\rS_{n-1}$ and $\rS_n$ are \emph{Lagrangian intersection loci}.

 The construction of the moduli stack of smooth GM varieties of dimension $n$ goes as follows.\
Assume $n \in \{3,4,5\}$ (the case $n = 6$ is slightly different and we skip it in this introduction).\
It was proved in~\cite{DK:coverings} that, if a certain divisibility condition holds in the group~$\Pic(\brS_n)$
(in fact it does not, but we will go back to this point later), 
there is a double covering 
\begin{equation*}
\widetilde\rS_n \lra \brS_n
\end{equation*}
branched over $\rS_{n-1}$  such that $\widetilde\rS_n$ is smooth.\
Note that $\codim_{\brS_n}(\rS_{n-1}) = 6 - n$, so for $n \le 4$, this is not a classical double covering branched over a hypersurface.\
We consider the quotient stack 
\begin{equation*}
\hrS_n := \widetilde\rS_n / \bmu_2
\end{equation*}
with respect to the involution of the double covering
 {(this is the \emph{canonical stack} of $\brS_n$ in the terminology of~\cite{Vistoli})}.\ 
This is a smooth Deligne--Mumford stack  and the natural $\PGL(V_6)$-action on the scheme $\brS_n$ lifts to a $\PGL(V_6)$-action on $\hrS_n$.\ 
Our main theorem, Theorem~\ref{theorem:gm-global-quotient}, states that  there is an isomorphism
\begin{equation*}
\ucM_n \cong  \hrS_n / \PGL(V_6)
\end{equation*}
between the  moduli stack  $\ucM_n$ of smooth GM varieties
of dimension $n\in\{3,4,5\}$ and the quotient stack $\hrS_n / \PGL(V_6)$.

The main step in the proof of this theorem is the construction of a family of smooth GM varieties over the stack $\hrS_n$
or, more precisely, over a certain scheme $\hhrS_n$  that provides a covering of $\hrS_n$ in the smooth topology
(the morphism $\hhrS_n \to \hrS_n$ is actually a~$\Gm$-torsor).\ There is also a double covering map from $\hhrS_n$ to a certain $\Gm$-torsor over $\brS_n$
which is branched over the preimage of $\rS_{n-1}$.\ 
Consequently, pulling back from $\brS_n$, we construct on~$\hhrS_n$ a family of Lagrangian data $(\ns,\cV_6,\cV_5,\cA)$ 
with   trivial $\cV_6  = V_6 \otimes \cO$  and (pullbacks of) tautological bundles~$\cV_5$ and~$\cA$.\ The Lagrangian special locus of this family is the scheme-theoretic preimage of $\rS_{n-1}$,
that is, the preimage of the branch locus of the double covering, hence 
its ideal is the square of the ideal of a certain smooth subscheme in $\hhrS_n$.\ 
Considering the blow up $\beta \colon \bS \to \hhrS_n$  of this subscheme, 
we  arrive at the situation of Section~\ref{subsection-from-lag-to-gm}.

Applying Proposition~\ref{proposition-gm-from-lag-naive}, we obtain a family $(\bS,\cW,\cV_6,\cV_5,\mu,\bq)$ of GM data.\ We check that this family is   the pullback by $\beta$ of a family of GM data on the scheme $\hhrS_n$.\ Moreover, 
 this family is equivariant with respect to the natural action of the algebraic group 
\begin{equation*}
\rG_n = \GL(V_6)/\bmu_{3(5-n)}
\end{equation*}
and thus descends to a family of GM data over 
\begin{equation}
\label{eq:intro-hhrs-hrs-iso}
\hhrS_n / \rG_n \cong \hrS_n / \PGL(V_6).
\end{equation}
This construction provides a morphism from the quotient stack $\hrS_n / \PGL(V_6)$ to the moduli stack of GM data.\ For the construction in the opposite direction, we use the much simpler procedure of Proposition~\ref{proposition-gm-to-lag}
and a   universal property of the stack $\hrS_n$ proved in Proposition~\ref{prop:factorization-through-hats}.\ Combining these two constructions, we obtain an isomorphism between the moduli stack of smooth GM varieties 
and the global quotient stack~\eqref{eq:intro-hhrs-hrs-iso}.

To deal with the fact that the double covering $\widetilde\rS_n \to \brS_n$ does not exist 
(since the required divisibility condition does not hold in $\Pic(\brS_n)$),  we note that  
the divisibility condition holds locally over $\brS_n$, so the   double covering exists locally.\  
One can obtain  the stack~$\hrS_n$ by  {gluing} the quotients stacks of the local coverings, 
as in the standard construction of the root stack.\  
Alternatively, one can  construct  the stack $\hrS_n$ directly (see Appendix~\ref{section:generalized-root-stack}).\  
After that, the construction goes as explained above.

To conclude this introduction, we mention that the global quotient stack description of Theorem~\ref{theorem:gm-global-quotient}
gives, via GIT, a construction of the coarse moduli space for smooth GM varieties (Theorem~\ref{theorem:gm-coarse}).\ 
This provides a foundation for the period map of GM varieties that was discussed in~\cite{DK:periods} ({see} Proposition~\ref{pm}).\ 
We also use our results to construct in Section~\ref{subsection:complete-families} 
several examples of complete nonisotrivial families of smooth GM varieties.

{\bf Acknowledgements.}
We are grateful to Ariyan Javanpeykar and to Alex Perry for interesting discussions 
and {extremely} useful comments on preliminary drafts of this article.\
{A.K.\ is also grateful to Sergey Gorchinskiy for sharing his understanding of stacks.}

\section{Preliminaries {on vector bundles}}
\label{section:preliminaries}

All schemes are over a fixed field $\k$.

We first discuss some aspects of the theory of vector bundles on possibly nonreduced schemes.\ Most of the material in Sections~\ref{subsection:epi-fiberwise-mono} and \ref{subsection:degeneration-schemes}
is   well known but we collect it for the reader's convenience.\ The results of Sections~\ref{subsection:canonical-factorization}, \ref{subsection:residual-quadrics}, and~\ref{subsection:hecke} 
seem to be new and are essential for our treatment of the stack of GM varieties.

\subsection{Epimorphisms and fiberwise monomorphisms}
\label{subsection:epi-fiberwise-mono}

Let $S$ be   a scheme.\ By a {\em point of $S$,} we mean a $\K$-point $s \colon \Spec(\K) \to S$ for some field $\K$.\
A \emph{geometric} point of $S$ is a $\K$-point with  $\K$ algebraically closed.\ 

A vector bundle $\cE$ on $S$ is a locally free sheaf of $\cO_S$-modules of constant finite rank.\ 
  Given  a $\K$-point $s $ of $S$, we let $\cE_s $ be the $\K$-vector space $ \cE \otimes_{\cO_S}  \K$, the {\em  fiber} of $\cE$ at $s$.\


\begin{lemm}\label{lemma-epi-points}
A morphism $\varphi\colon\cE \to \cF$ between vector bundles on the scheme $S$ is surjective if and only if, 
for every  {geometric} point $s $ of $ S$, the induced linear map $\varphi_s\colon\cE_s \to \cF_s$
between fibers is surjective.
\end{lemm}

\begin{proof}
Let $\cC$ be the cokernel of $\varphi$.\ Since the tensor product functor is right exact, we have, for each point $s $ of $ S$,  an exact sequence
\begin{equation*}
\cE_s \xrightarrow{\ \varphi_s\ } \cF_s \to \cC_s \to 0.
\end{equation*}
By Nakayama's lemma,   $\cC = 0$ if and only if  $\cC_s = 0$ for every  {geometric} point $s$ of $ S$.
\end{proof}

We say that $\varphi$ is a {\sf fiberwise monomorphism} if, for every  {geometric} point $s $ of $ S$, 
the morphism $\varphi_s\colon\cE_s \to \cF_s$ is a monomorphism.

\begin{lemm}
\label{lemma-fmono-epi}
A morphism $\varphi\colon\cE \to \cF$ between vector bundles on a scheme $S$ is a fiberwise monomorphism if and only if the dual map $\varphi^\vee\colon\cF^\vee \to \cE^\vee$ is surjective.
\end{lemm}

\begin{proof}
 Since $(\cE^\vee)_s = (\cE_s)^\vee$, $(\cF^\vee)_s = (\cF_s)^\vee$, and $(\varphi^\vee)_s = (\varphi_s)^\vee$,
the result follows from Lemma~\ref{lemma-epi-points}.
\end{proof}

Epimorphisms and fiberwise monomorphisms enjoy the following nice properties.

\begin{lemm}
\label{lemma-kernel-cokernel-lf}
Let $\varphi \colon \cE \to \cF$ be a morphism between vector bundles on a scheme $S$.

If $\varphi  $ is surjective,   $\Ker(\varphi)$ is a vector bundle and the natural map $\Ker(\varphi) \to \cE$ is a fiberwise monomorphism.

If $\varphi$ is a fiberwise monomorphism,   $\Coker(\varphi)$ is a vector bundle and the natural map $\cF \to \Coker(\varphi)$ is surjective.
\end{lemm}

\begin{proof}
The kernel is locally free since this is a local property  and the kernel of an epimorphism of projective modules over a ring is projective.\
Furthermore, since $\cF$ is locally free, we have $\Tor_1(\cF, {\K}) = 0$ for any  {$\K$}-point $s$ of $ S$, hence the sequence
\begin{equation*}
0 \to \Ker(\varphi)_s \to \cE_s \to \cF_s \to 0
\end{equation*}
is exact.\ 
 By definition, the map $\Ker(\varphi) \to \cE$ is therefore a fiberwise monomorphism.\ The second part of the lemma follows by duality.
\end{proof}

An effective Cartier divisor on a   scheme
is a subscheme locally defined   by a regular function which is \emph{not a zero divisor}.

\begin{lemm}
\label{lemma-kernel-cartier}
Let $S$ be a scheme, let $i\colon D \hookrightarrow S$ be an effective Cartier divisor, let $\cE$ be a vector bundle on $S$, and let $\cF$ be a vector bundle on $D$.\ If $\varphi\colon\cE \to i_*\cF$ is surjective,
  $\Ker(\varphi)$ is a vector bundle on $S$.
\end{lemm}

\begin{proof}
We may assume that $D$ is nonempty.\ Since $D$ is a   Cartier divisor, locally, the projective dimension of~$\cO_D$,   
hence also of ${i_*}\cF$,  as an $\cO_S$-module  {is~1}.\  
Therefore, the projective dimension of $\Ker(\varphi)$ is 0, so $\Ker(\varphi)$ is locally free.
\end{proof}

\subsection{Degeneration schemes}
\label{subsection:degeneration-schemes}

Let  $\varphi\colon\cE \to \cF$ be a morphism between vector bundles on a 
 scheme $S$.\ For every nonnegative integer $k  $, it induces a morphism
\begin{equation*}
 \bw{k}{}\varphi\colon  \bw{k}\cE \xrightarrow{\ \ \ } \bw{k}\cF
\end{equation*}
  locally given by the   $k \times k$-minors of a matrix of regular functions defining $\varphi$.\
 The {\sf   rank}
of $\varphi$ is the smallest integer $r$ such that
 $\bw{r+1}\varphi = 0$ (identically on $S$).\ In particular,  $\varphi=0$ if and only if its rank is 0.\

Given any nonnegative integer $k$, we define the {\sf rank-$k$ degeneration scheme} of $\varphi$ as the zero locus on $S$ of the morphism $\bw{k+1}\varphi$.\
If the rank of $\varphi$ is $r$, we abbreviate its rank-$(r-1)$ degeneration scheme to just degeneration scheme (the degeneration scheme of the zero morphism is empty).

The morphism $\varphi\colon\cE \to \cF$  is {\sf generically surjective} if its rank 
on every irreducible component of $S$
equals the rank of $\cF$.\ The cokernel of a generically surjective morphism is a torsion sheaf  supported on the degeneration scheme of $\varphi$.

If $\varphi$ has rank $r$, for any $\K$-point~$s $ of $ S$, the rank of the  {$\K$}-linear map $\varphi_s$ is at most $r$.\ 
The converse    may however not be true: if $S = \Spec (\k[x]/x^2)$, $\cE = \cF = \cO_S$, and $\varphi = x$,  
the rank of $\varphi$ is 1 and $\varphi$ is generically surjective, but $\varphi_s = 0$ at all points $s $ of $ S$.\ 
Its degeneration scheme is  $S_{\textnormal{red}}$.

\begin{lemm}
\label{lemma-support-cokernel}
Let $\varphi\colon\cE \to \cF$ be a 
morphism of positive rank $r$ between vector bundles on a scheme $S$.\ 
Assume that $\bw{r-1}\varphi_s$ does not vanish for any  {geometric} point~\mbox{$s$ of $ S$}.

The degeneration scheme of $\varphi$ then equals the scheme-theoretic support of the sheaf $\Coker(\varphi)$,
that is, the subscheme corresponding to the annihilator ideal of~$\Coker(\varphi)$.\  Moreover, the sheaf $\Coker(\varphi)$ is isomorphic to the pushforward of a line bundle on this subscheme.
\end{lemm}

\begin{proof}
For any  {$\K$}-point   $s$ of $ S$, one of the $(r-1)\times (r-1)$-minors of $\varphi$ does not vanish in~{$\K$}, 
hence it spans $\cO_{S,s}$; this means that the first Fitting ideal of $\Coker(\varphi)$ 
(generated by the $(r-1)\times (r-1)$-minors of   $\varphi$) is trivial (\cite[Corollary-Definition~20.4]{Eisenbud}).\ By ~\cite[Proposition~20.7]{Eisenbud}, the zeroth Fitting ideal (which defines the degeneration scheme of $\varphi$) 
is then  equal to the annihilator of $\Coker(\varphi)$.

 To prove the second part, we base change to the support of $\Coker(\varphi)$.\
By~\cite[Corollary~20.5]{Eisenbud}, the first Fitting ideal of $\Coker(\varphi)$ is still trivial,
while the zeroth Fitting ideal is equal to zero;~\cite[Proposition~20.8]{Eisenbud} then implies that $\Coker(\varphi)$ is a line bundle.
\end{proof}

Lemma~\ref{lemma-support-cokernel} can also   be proved by the argument of Proposition~\ref{proposition-factorization} below.

\begin{lemm}
\label{lemma:degeneration-cartier}
Let $\varphi \colon \cE \to  \cF$ be a morphism between vector bundles of rank $r$ on a scheme~$S$.\ Assume that
the degeneration scheme of $\varphi$ is a Cartier divisor $D $ on $ S$.
 
We have $\Ker(\varphi) = 0$  
and $\Coker(\varphi)$ is supported scheme-theoretically on $D$.
 \end{lemm}

\begin{proof}
Let $\cL := \det(\cE^\vee) \otimes \det(\cF)$.\
 By definition of the degeneration scheme, the ideal of~$D$ is generated by $\det(\varphi)$, 
so the assumption that $D$ be a Cartier divisor means that $\det(\varphi)$, viewed as a section of $\cL$,  is not a zero divisor.\
Consider the diagram
\begin{equation*}
\xymatrix{
0 \ar[r] &
\Ker(\varphi) \ar[r] &
\cE \ar[r]^-{\varphi} \ar[d]_{\det(\varphi)} &
\cF \ar[r] \ar[d]^{\det(\varphi)} \ar[dl]_{\widehat\varphi} & 
\Coker(\varphi) \ar[r] &
0
\\
0 \ar[r] &
\Ker(\varphi) \otimes \cL \ar[r] &
\cE \otimes \cL \ar[r]^-{\varphi} &
\cF \otimes \cL \ar[r] & 
\Coker(\varphi) \otimes \cL \ar[r] &
0 
}
\end{equation*}
where $\widehat\varphi$ is the composition
\begin{equation*}
\cF \cong 
\bw{r-1}\cF^\vee \otimes \det(\cF) \xrightarrow{\  \sbw{r-1}\varphi^\vee\ }
\bw{r-1}\cE^\vee \otimes \det(\cF) \cong
 \cE \otimes \cL,
\end{equation*}
that is, $\widehat\varphi$ is the adjoint morphism of $\varphi$.\ In particular, the diagram commutes.\ 
It follows that the morphism induced by $\det(\varphi)$ on $\Ker(\varphi)$ and $\Coker(\varphi)$ is zero,
hence both   sheaves are supported on $D$.\ Since $\cE$ is torsion free, it follows that $\Ker(\varphi) = 0$.
 \end{proof}

\begin{lemm}
\label{lemma-factorization}
Let $\varphi\colon\cE \to \cE'$ be a morphism between vector bundles on a  
 scheme~$S$, 
which is surjective on the complement of an effective Cartier divisor, and let $\cF$ be a vector bundle on~$S$.\
Then the map $\Hom(\cE',\cF) \xrightarrow{\ \circ\varphi\ } \Hom(\cE,\cF)$ is injective.\ 
In other words, if a morphism~$\psi\colon\cE \to \cF$ factors through $\cE'$, such a factorization is unique.
\end{lemm}

\begin{proof}
Let $\cC$ be the cokernel of $\varphi$.\ By the left exactness of the $\Hom$ functor, it is enough to show  that $\Hom(\cC,\cF) = 0$.\
The question is local, so we may assume $\cF = \cO_S$.\ Moreover, since the degeneration scheme is contained in an effective Cartier divisor, 
we may assume that~$\cC$ is annihilated by a regular function $f$ on $S$ which is not a zero divisor.\ But  the image of any morphism~$\cC \to \cO_S$ is then annihilated
by $f$, hence is zero.
\end{proof}

If we do not assume that the degeneration scheme is contained in an effective Cartier divisor, 
the conclusion of Lemma~\ref{lemma-factorization} may not hold.\ 
For example, let $S = \Spec(\k[x,y]/(xy,y^2))$, $\cE = \cO_S \oplus \cO_S$, $\cE' = \cF = \cO_S$, and~$\varphi = (x, y)$.\ 
Consider the nonzero map $\cE' \to \cF$ given by~$y$;   its composition with $\varphi$ is zero.\ 
The degeneration scheme of $\varphi$ is defined by the maximal ideal~$(x,y)$;
it is a Weil divisor, but not a Cartier divisor.

\subsection{Canonical factorization}
\label{subsection:canonical-factorization}

The following canonical factorization of a morphism between vector bundles  seems   to be little known,
but it will be crucial for our construction.

\begin{prop}
\label{proposition-factorization}
Let $\varphi\colon\cE \to \cF$ be a morphism of positive rank $r$ between vector bundles on a scheme $S$.\ 
Assume that its degeneration scheme is a Cartier divisor $D$ on $S$ and that~$\bw{r-1}\varphi_s$  does not vanish for any  {geometric} point $s$ of $ S$.\ 
There is a unique factorization  
\begin{equation*}
\varphi\colon \cE \twoheadrightarrow \cE_1 \xrightarrow{\ \varphi_1\ } \cF_1 \hookrightarrow \cF 
\end{equation*}
such that
\begin{itemize}
\item  $\cE_1$ and $\cF_1$ are vector bundles of rank $r$,
\item  the   map $\cE \twoheadrightarrow \cE_1$ is an epimorphism,
\item the   map $\cF_1 \hookrightarrow \cF$ is a fiberwise monomorphism,
\item the map $\varphi_1$ is a monomorphism and its cokernel is a line bundle on $D$.
\end{itemize}
\end{prop}

\begin{proof}
For any such factorization,  $\varphi_1$ is injective by Lemma~\ref{lemma:degeneration-cartier}, 
hence $\cE_1$ is the image of~$\varphi$ and $\cF_1$ is the dual of the image of~$\varphi^\vee$, so the uniqueness is clear.\
It is therefore enough to prove the proposition locally
 so we may assume that $\cE$ and $\cF$ are trivial vector bundles 
 and~$\varphi$ 
is given by a matrix of regular functions on $S$.

Let $s  $ be a  {geometric} point on $S$.\ 
The rank of $\varphi_s$ is either $r$ or $r-1$.\ If it is $r$,   one of the $r\times r$-minors   of $\varphi$ is nonzero at $s$, hence
is invertible in a neighborhood of $s$.\ Restricting to such a neighborhood and considering appropriate bases for the fibers $\cE_s$ and $\cF_s$,
we may assume that the minor corresponds to the first $r$ basis vectors in each basis.\ In other words, the matrix has the form
\begin{equation*}
 \varphi = \begin{pmatrix} \varphi_{1,1} & \varphi_{1,2} \\ \varphi_{2,1} & \varphi_{2,2} \end{pmatrix}  ,
\end{equation*}
where $\varphi_{1,1}$ is a square matrix of size $r$ with   invertible  determinant.\ 
The matrix $\varphi_{1,1}$ is therefore invertible and, upon multiplying it by its inverse,
we may assume that it is the identity matrix~$I_r$.\ 
Applying elementary transformations to rows and columns, we may assume that~$\varphi_{1,2} = \varphi_{2,1} = 0$.\
The entries of the matrix $\varphi_{2,2}$ are then $(r+1)\times(r+1)$-minors of the matrix~$\varphi$, hence they all  vanish.\
In a neighborhood of $s$, the map $\varphi$ can therefore be written as a composition 
of the epimorphism of $\cE$ onto the trivial vector bundle of rank~$r$ (corresponding to the first $r$ basis vectors) 
and its fiberwise monomorphism into~$\cF$.\ 
In particular,   $\varphi_1$ is an isomorphism.

If the rank of $\varphi_s$ is $r-1$, restricting to a   neighborhood of $s$ and choosing bases of $\cE_s$ and~$\cF_s$ appropriately, we may assume
that $\varphi$ is in the form as above, but where now~\mbox{$\varphi_{1,1} = I_{r-1}$} and $\varphi_{1,2} = \varphi_{2,1} = 0$.\ 
Again, the entries of $\varphi_{2,2}$ are the $r\times r$-minors of~$\varphi$, hence they    generate the   ideal 
generated by the equation $f$ of the Cartier divisor $D$.\ 
Therefore, we can write $\varphi_{2,2} = f\varphi'_{2,2}$; the ideal generated by the entries of $\varphi'_{2,2}$ is trivial, 
hence the matrix~$\varphi'_{2,2}$  vanishes nowhere.

On the other hand, the $2\times2$-minors  of the matrix~$\varphi_{2,2}$ are equal to (some) $(r+1)\times(r+1)$-minors of $\varphi$, 
hence they all  vanish identically; therefore (recall that $f$ is not a zero divisor), the same is true for the matrix $\varphi'_{2,2}$.\ 
Applying the same arguments as above, we can assume the matrix $\varphi_{2,2}$ has   top left
entry   $f$ and all   other entries   $0$.\ 
In a neighborhood of $s$, the map~$\varphi$ can thus be written as the composition of the epimorphism of $\cE$ 
onto the trivial vector bundle~$\cO^{\oplus r}$ (corresponding to the first $r$ basis vectors), 
a map $\varphi_1$ given by the diagonal matrix   $\diag(1,\dots,1,f)$, and   a 
 fiberwise monomorphism from~$\cO^{\oplus r}$ into $\cF$.\ 
In particular,  $\varphi_1$ is a monomorphism and its cokernel is (locally) the structure sheaf of~$D$.
 \end{proof}

\subsection{Families of  quadratic forms and residual families}
\label{subsection:residual-quadrics}

Let $\cE$ and $\cV$ be   vector bundles of respective ranks~$r$ and~$k$  on a   scheme~$S$ and let 
\begin{equation}\label{defq}
\bq\colon\cV \lra \Sym^2 \!\cE^\vee
\end{equation}
be a morphism of sheaves.\ 
We may think of $\bq$ as a family $\cE \otimes \cE \to \cV^\vee $  of quadratic forms on~$\cE$ with values in  $\cV^\vee$.\ 

When $\cV$ has rank 1, we define the discriminant subscheme~$\Dis(\bq)\subset S$ of the family~$\bq$ as 
the degeneration scheme of the associated morphism 
\begin{equation*}
\cE \xrightarrow{\ \bq\ } \cE^\vee \otimes \cV^\vee 
\end{equation*}
of rank-$r$ vector bundles, that is, the zero locus of the induced section $\det(\bq)$ of the line bundle $\det(\cE^\vee)^{\otimes 2}\otimes (\cV^\vee)^{\otimes r}$.

Let $i\colon D \hra S$ be an effective Cartier divisor.\ 
Let $\cK \subset i^*\cE$ be a line subbundle which is contained in the kernel of the quadratic forms $i^*\bq$.\ 
In other words, the composition $i^*\cV \xrightarrow{\ i^*\bq\ } \Sym^2 (i^*\cE^\vee) \to i^*\cE^\vee \otimes \cK^\vee$
vanishes (when $\cV$ has rank 1, this  {implies} that $D$ is contained in the discriminant  $\Dis(\bq)$).
In this situation, we construct a family of quadratic forms on the line bundle~$\cK$ over $D$ as follows.

Let $\tcK \subset \cE$ be a local extension over an open subset of $S$ of the line subbundle \mbox{$\cK \subset i^*\cE$}.
The family of   quadratic forms $\bq$ induces a map 
$\cV \to \Sym^2 \!\cE^\vee \to \Sym^2 \!{\tcK}^\vee$
 which by our assumption vanishes on the divisor $D$, hence factors through a map 
\begin{equation*}
\cV(D) \lra \Sym^2\!\tcK^\vee.
\end{equation*}
Restricting it to $D$, we get a map 
\begin{equation*}
\cV(D)\vert_D \lra \Sym^2\!\cK^\vee
\end{equation*}
which we call   the {\sf residual family of quadratic forms}.

\begin{lemm}\label{lemma-residual-quadric}
The residual family of quadratic forms   on $D$
is independent of  the choices made.

If the rank of $\cV$ is $1$ and $\Dis(\bq) = D$ as   subschemes of $S$,
  the rank of $\bq$ on $D$ is equal to $r-1$  and 
the residual family of quadratic forms  vanishes nowhere on $D$.
\end{lemm}

\begin{proof}
Let $s_0$ be a section of $\cE$ extending locally a section generating $\cK$.\ Any other extension
can be written as  $s_0 + t s$ for some section $s$ of $\cE$,
where $t$ is a local equation of the divisor~$D$.\ The evaluation of $\bq$ on this section is equal to
\begin{equation*}
\bq(s_0 + t s, s_0 + t s) = \bq(s_0,s_0) + 2t\bq(s_0,s) + t^2\bq(s,s).
\end{equation*}
The factorization through $\cV(D)$ is then given by 
\begin{equation*}
 \frac1t \Bigl(\bq(s_0,s_0) + 2t\bq(s_0,s) + t^2\bq(s,s) \Bigr) \Big\vert_{t = 0} = 
 \Bigl( \frac1t\bq(s_0,s_0)  + 2\bq(s_0,s) \Bigr) \Big\vert_{t = 0}.
\end{equation*}
It remains to note that $\bq(s_0,s)$ vanishes on $D$ since $s_0$ is    in the kernel of ${i^*}\bq$.

Let us choose   local trivializations of $\cE$ and $\cV$ such that $\bq$ corresponds
to a $k$-tuple of symmetric matrices $(q^\alpha_{ij})_{1 \le i,j \le r}$ of regular functions on $S$ (where    $ \alpha \in\{1,\dots,k\}$) 
and  the section $s_0$ (defining a local extension $\tcK$ of $\cK$) corresponds to the first basis vector of $\cE$.\ 
The condition that $\cK$ be in the kernel of $i^*\bq$ then means   
\begin{equation}
\label{equation-divisibility}
q^\alpha_{1i} = t\bar{q}^\alpha_{1i}
\qquad\qquad
\text{for all $1 \le i \le r$ and all $1 \le \alpha \le k$},
\end{equation} 
and the residual quadric is just $\bar{q}^\alpha_{11}\vert_D$.

When the rank of $\cV$ is one (so we  have just one symmetric matrix $(q_{ij})$), we have
\begin{equation*}
\det(q_{ij}) \equiv t \bar{q}_{11} \det(q_{i,j})_{2\le i,j \le r} \pmod {t^2}.
\end{equation*}
The condition $\Dis(\bq) = D$ implies that the last factor  {$\det(q_{i,j})_{2\le i,j \le r}$} is invertible 
(hence the rank of~$\bq$ on~$D$ is~$r-1$)
and $\bar{q}_{11}$ is invertible on $D$ too, so that the residual family of quadratic forms vanishes nowhere.
\end{proof}

\subsection{Hecke transform of  a family of quadratic forms}
\label{subsection:hecke}

We continue working in the setup of the previous section.\
The construction presented here (which we call  Hecke transform) is a generalization 
of the construction from~\cite[Lemma~1.14]{Sarkisov}.

\begin{defi}\label{defi:double}
An effective Cartier divisor $D$ is a {\sf double} 
if there is an effective Cartier divisor $E$ on $S$ such that $D = 2E$,
that is,  {the ideal of $D$ is the square of the ideal of $E$.}
 \end{defi}

Assume  that the effective Cartier divisor $D$ considered in the previous section is a double and write   $D = 2E$.\ For any line subbundle $\cK \subset \cE\vert_D  $, we set 
 \begin{equation*}
\cK_E := \cK\vert_E,
\qquad
\cK_E^\vee := \cK^\vee\vert_E,
\qquad
\cE_D:=\cE\vert_D,
\qquad 
\cE_E:=\cE\vert_E.
\end{equation*}
 Since $\cK$ is a line bundle on $D$ and $E$ is a Cartier divisor on $ {S}$, 
the kernel of the  {natural} epimorphism~$\cE^\vee \to  \cK_E^\vee $ is a vector bundle (Lemma~\ref{lemma-kernel-cartier}).\ 
We denote by $\tcE$ its dual, so that there is an exact sequence 
\begin{equation}\label{equation-tcw-dual}
0 \to \tcE^\vee \to \cE^\vee \to  \cK_E^\vee  \to 0 
\end{equation} 
of sheaves on $S$, whose dual sequence can be written as
\begin{equation}
\label{equation-tcw}
0 \to \cE \to \tcE \to  \cK_E(E) \to 0.
\end{equation} 

The following lemma will be very useful later.

\begin{lemm}\label{lemma-bq-factors}
Assume  $D = 2E$.\ Let $\cK \subset  \cE_D$ be a line subbundle contained in the kernel of $ \bq\vert_D$.\ Define the vector bundle $\tcE$  by the exact sequence~\eqref{equation-tcw-dual}.\ The   family  of quadratic forms $\bq \colon \cV \to \Sym^2\!\cE^\vee$ then factors through $\Sym^2\!\tcE^\vee$ in a unique way.
\end{lemm}

\begin{proof}
We can use a representation of $\bq$ by a matrix $(q^\alpha_{ij})$ as in the proof of Lemma~\ref{lemma-residual-quadric}
with the same conventions on the coordinates, assuming in particular that~\eqref{equation-divisibility} holds.\ Let $u$ be an equation of $E$ in ${S}$, so that the equation of $D$ can be written as~$t = u^2$.\  The sequence~\eqref{equation-tcw-dual}  can then  be written in local coordinates on ${S}$ as
 \begin{equation}\label{equation-local-sequence}
0 \to \cO_{S}^{\oplus r} \xrightarrow{\ \diag(u,1,\dots,1)\ } \cO_{S}^{\oplus r} \lra \cO_E \to 0  .
\end{equation}
 The factorization condition that we want to prove just means that $q^\alpha_{11}$ is divisible by $u^2$ and~$q^\alpha_{12},\dots,q^\alpha_{1r}$  
are divisible by~$u$; since $t = u^2$, this follows   from~\eqref{equation-divisibility}.\  
The uniqueness of the factorization follows from Lemma~\ref{lemma-factorization} applied to  {the symmetric square of}~\eqref{equation-tcw}.
\end{proof}

We denote by $\tilde\bq \colon \cV \to \Sym^2\!\tcE^\vee$ the induced map and call it the {\sf Hecke transform} of~$\bq$
with respect to the line subbundle $\cK \subset  \cE_D$.

\begin{prop}
\label{proposition-hecke-transform}
Let $\cE$ and $\cV$ be   vector bundles of respective ranks $r$ and $k$ on a   scheme~$S$.\ 
Let  $\bq\colon\cV \to \Sym^2\!\cE^\vee$ be a family  of quadratic forms
on $\cE$ with values in~$\cV^\vee$.\
Assume finally that there exist  a double Cartier divisor $D = 2E$ and a line subbundle $\cK \subset \cE_D$ on $D$
which is contained in the kernel of~$\bq\vert_D$.\
Let $\tilde\bq\colon\cV \to \Sym^2\!\tcE^\vee$ be the Hecke transform of $\bq$ with respect to $\cK$.

{\rm(a)} 
The restriction of the sequence~\eqref{equation-tcw} to $E$ splits and gives
 a canonical direct sum decomposition
\begin{equation}
\label{equation-tcw-splitting}
\tcE_E \cong ( \cE_E/\cK_E) \oplus \cK_E(E) 
\end{equation}
of the restriction $\tcE_E$ of $\tcE$ to $E$.

{\rm(b)} 
The summands of~\eqref{equation-tcw-splitting} are mutually orthogonal with respect to the  {quadratic form}~$\tilde\bq\vert_E$.\ Moreover, the restriction of $ \tilde\bq\vert_E$ to the first summand   of~\eqref{equation-tcw-splitting}   is induced by $  \bq\vert_E$  
and the restriction of $ \tilde\bq\vert_E$ to the second summand is the residual family of  quadratic forms for $\bq$.
 \end{prop}

\begin{proof}
Consider the vector bundle $\ttcE$ on $S$ defined as the dual of the kernel of the natural map $\cE^\vee \to  \cK^\vee$.\ We have
an exact sequence 
 \begin{equation}\label{equation-tttcw}
0 \to \cE \to \ttcE \to  \cK(D)  \to 0.
\end{equation}
 By construction, the embedding $\cE \to \ttcE$ factors as   $\cE \to \tcE \to \ttcE$
and the  map $  \tcE \to \ttcE$ fits into the exact sequence
\begin{equation}\label{equation-ttcw}
0 \to \tcE \to \ttcE \to  \cK_E(2E)  \to 0.
\end{equation}
 Restricting~\eqref{equation-tcw},~\eqref{equation-ttcw}, and~\eqref{equation-tttcw} to $E$, we obtain  exact sequences 
 \begin{gather} 
 0 \to \cK_E \to  \cE_E 		 \to  \tcE_E \to \cK_E(E)\hphantom{2} \to 0, \label{r1}\\
0 \to \cK_E(E) \to  \tcE_E 	 \to  \ttcE_E \to \cK_E(2E) \to 0, \label{r2}\\
0 \to \cK_E \to  \cE_E 		 \to  \ttcE_E \to \cK_E(2E) \to 0.\label{r3}
\end{gather}
Since the composition of the middle arrows of \eqref{r1} and  \eqref{r2} is the middle arrow of  \eqref{r3},
the composition $\cK_E(E) \to  \tcE_E \to \cK_E(E)$ is an isomorphism, hence the sheaf $\cK_E(E)$ is a direct summand of $ \tcE_E$.\  The    sequence \eqref{r1} identifies the other summand with $ \cE_E/\cK_E$.\ This proves~(a).

Let us prove (b).\  The first map in~\eqref{equation-tcw} restricted to $E$ can be written as a composition
\begin{equation*}
 \cE_E \twoheadrightarrow   \cE_E /\cK_E \hookrightarrow  \tcE_E.
\end{equation*}
Taking the symmetric square and dualizing, we see that the map $\Sym^2\!\tcE^\vee_E \to  \Sym^2 \!\cE^\vee_E$ can be written as the composition
\begin{equation*}
 \Sym^2\!\tcE^\vee_E\twoheadrightarrow \Sym^2( \cE_E/\cK_E)^\vee \hookrightarrow  \Sym^2\!\cE^\vee_E.
\end{equation*}
By Lemma~\ref{lemma-bq-factors}, the quadric $ \bq\vert_E$ (considered as a map from $\cV$ to $\Sym^2\!\cE^\vee_E$) 
factors through $ \Sym^2\!\tcE^\vee_E $ as $ \tilde\bq\vert_E$.\ Hence it a fortiori factors through the middle term.\ Such a factorization is nothing  but the induced family  of quadratic forms 
on $ \cE_E/\cK_E$  and the image of $ \tilde\bq\vert_E$ is   the restriction of $ \bq\vert_E$
to $ \cE_E/\cK_E$.\ These two families of   quadratic forms therefore coincide.\ 

We now show that the summands in~\eqref{equation-tcw-splitting} are mutually orthogonal and 
that the restriction of~$ \tilde\bq\vert_E$ to the summand $\cK_E(E)$ is given by the residual family of  quadratic forms.\
The question is local, so we can assume that $\cE$, $\cK$, and $\cV$ are trivialized as in the proof of Lemma~\ref{lemma-residual-quadric}.\
Under these assumptions, the sequence~\eqref{equation-tcw-dual} can be rewritten
as in~\eqref{equation-local-sequence}.\ This means that the matrix of $\tilde\bq$ is 
 \begin{equation*}
\left(
\begin{smallmatrix}
q^\alpha_{11}/u^2 & q^\alpha_{12}/u & \dots & q^\alpha_{1r}/u \\
q^\alpha_{21}/u & q^\alpha_{22} & \dots & q^\alpha_{2r} \\
\vdots & \vdots & \ddots & \vdots \\
q^\alpha_{r1}/u & q^\alpha_{r2} & \dots & q^\alpha_{rr} 
\end{smallmatrix}
\right).
\end{equation*}
In particular, its restriction to the summand $\cK_E(E)$ is given by $(q^\alpha_{11}/u^2)\vert_E = (q^\alpha_{11}/t)\vert_E$,
which is  the residual family of quadratic forms.\  
Moreover, if we set $\bar{q}^\alpha_{1i} := q^\alpha_{1i}/t$ for all $i > 1$ as in~\eqref{equation-divisibility},  
we have $(q^\alpha_{1i}/u)\vert_E = (u \bar{q}^\alpha_{1i})\vert_E = 0$,
hence the summands in~\eqref{equation-tcw-splitting} are mutually orthogonal.
 \end{proof}

\section{The moduli stack of smooth GM varieties}

In this section, we introduce the stack of GM varieties and the closely related   stacks of GM and Lagrangian data.\
We mostly  work in the \'etale topology, but one can also work with the  {fppf topology}.\
 {From now on, we assume that the characteristic of the base field $\k$ is zero.}

\subsection{The stack of GM varieties}
\label{subsection:stack-gm}

We start with a definition of the stack of GM varieties.

\begin{defi}
A {\sf family of smooth polarized GM varieties of dimension~$n$ over a scheme $S$} is    a {pair} $(\cX \to S,\cH)$,
where 
\begin{itemize}
\item 
$\pi_\cX\colon\cX \to S$ is a smooth and proper morphism, 
\item 
 $\cH \in \Pic_{\cX/S}(S)$ is a relative $\pi_\cX$-ample divisor class,
\end{itemize}
such that for every  {geometric} point $s $ of 
$ S$,
\begin{itemize}
\item 
the pair $(\cX_s,\cH\vert_{\cX_s})$ is a smooth polarized GM variety of dimension~$n$
in the sense of~\cite[Definition~2.1]{DK}.
\end{itemize}

A morphism of families of GM varieties from $(\cX \to S,\cH)$ to $(\cX' \to S',\cH')$ is 
a pair~$(f,\varphi)$ giving a Cartesian square
\begin{equation}
\label{eq:morphism-gm}
\xymatrix{
\cX \ar[r]^-{\varphi} \ar[d]_{\pi_{\cX}} &
\cX' \ar[d]^{\pi_{\cX'}} 
\\
S \ar[r]^-f &
S'
}
\end{equation}
such that $\cH = \varphi^*\cH'$ in the relative Picard group $\Pic_{\cX/S}(S)$.

Families of smooth polarized GM varieties of dimension $n$ form a category fibered in groupoids over the category $\Sch/\k$ of schemes over $\k$; 
we denote it by $\ucM_n$.
\end{defi}

Smooth GM varieties exist in each dimension  $n\in\{1,\dots,6\}$.\ A GM variety of dimension~1 is   a Clifford general curve of genus~6 (\cite[Proposition~2.12]{DK})
and a GM variety of dimension~2 is   a Brill--Noether general polarized K3 surface of genus~6 (\cite[Proposition~2.13]{DK})  and their moduli stacks are well studied.\ Accordingly,  we will concentrate in this article   on GM varieties of dimension $n \in\{3,4,5,6\}$.\ 
There is then an isomorphism 
\begin{equation*}
\Pic_{\cX/S} \cong \underline{\Z}
\end{equation*}
between \'etale sheaves  {(see~\cite[Lemma~2.29]{DK})};  
over a connected scheme $S$, there is therefore a unique choice of a relative divisor class $\cH$.

\begin{prop}[{\cite[Proposition~A.2]{KP}}]
\label{proposition:mgm-dm}
For $n \in \{2,\dots,6\}$, the fibered category $\ucM_n$ is a smooth and irreducible Deligne--Mumford stack of dimension $25 - (5-n)(6-n)/2$.\ {It is of finite type over $\Q$ with affine diagonal of finite type}.
 \end{prop}

We call $\ucM_n$ {\sf the moduli stack of smooth GM varieties of dimension~$n$}.\ 
We will see later {(Theorem~\ref{theorem:gm-global-quotient})} that $\ucM_n$ is separated.

\subsection{The stack of GM data}
\label{subsection:stack-gm-data}

GM data sets over a field were defined in~\cite[Definition~2.5]{DK}.\ 
It will be  convenient to change the definition slightly as follows.

\begin{defi}
\label{definition-gm-data-normalized}
A {\sf normalized family of GM data of dimension $n$ over a scheme $S$} is a collection~$(\ns,\cW,\cV_6,\cV_5,\mu,\bq)$, 
where 
\begin{itemize}
\item $\cW$, $\cV_6$, and $\cV_5$ are vector bundles on $S$ of respective ranks $n+5$, $6$, and $5$,
\item $\cV_5 \hookrightarrow \cV_6$ is a fiberwise monomorphism,
 \item $\mu \colon \cW \to \bw2\cV_5 \otimes (\cV_6/\cV_5)$ and
\item $\bq \colon \cV_6 \to \Sym^2\!\cW^\vee \otimes \det(\cV_5) \otimes (\cV_6/\cV_5)^{\otimes 2}$ are morphisms between vector bundles,
 \end{itemize}
such that the   diagram 
\begin{equation}
\label{definition-gm-data}
\vcenter{\xymatrix@C=3em@M=6pt{
\cV_5 \otimes \Sym^2\!\cW \ar[d]_{\Sym^2\!\mu} \ar@{^{(}->}[r] & 
\cV_6 \otimes \Sym^2\!\cW \ar[d]^\bq \\
\cV_5 \otimes \Sym^2(\bw2\cV_5) \otimes (\cV_6/\cV_5)^{\otimes 2}
 \ar[r]^-{\wedge} & 
\det(\cV_5) \otimes (\cV_6/\cV_5)^{\otimes 2}
 }}
\end{equation}
(the bottom arrow is given by wedge product) commutes.

A morphism {of} normalized families $(\ns,\cW,\cV_6,\cV_5,\mu,\bq)$ and $(\ns',\cW',\cV'_6,\cV'_5,\mu',\bq')$ 
of GM data over schemes $S$ and $S'$   is a morphism $f \colon S \to S'$ and isomorphisms
 \begin{equation*}
\varphi_V \colon \P_S(\cV_6) \isomto \P_S(f^*\cV'_6) \quad\text{and}\quad 
\varphi_W \colon \P_S(\cW) \isomto \P_S(f^*\cW') 
\end{equation*}
such that $\varphi_V(\P_S(\cV_5)) = \P_S(f^*\cV'_5)$  and 
 the following diagrams commute
\begin{equation*}
\vcenter{\xymatrix@C=3em{
\P_S(\cW) \ar[r]^-{\varphi_W} \ar@{-->}[d]_{\mu} &
\P_S(f^*\cW') \ar@{-->}[d]^{\mu'} 
\\
\P_S(\bw2\cV_5) \ar[r]^-{\ \sbw2\varphi_V\ } &
\P_S(f^*\bw2\cV'_5)
}}
\quad\text{and}\quad 
\vcenter{\xymatrix@C=5em{
\P_S(\cV_6) \ar[r]^-{\varphi_V} \ar@{-->}[d]_{\bq} &
\P_S(f^*\cV'_6) \ar@{-->}[d]^{\bq'} 
\\
\P_S(\Sym^2\!\cW^\vee) &
\P_S(f^*\Sym^2\!\cW^{\prime\vee}) \ar[l]_-{\Sym^2\!\varphi_W^\vee} .
}}
\end{equation*}
\end{defi}
 
It is sometimes  convenient to express the commutativity of~\eqref{definition-gm-data} as the equality
\begin{equation}
\label{eq:mu-q}
\bq(v)(w_1,w_2) = v \wedge \mu(w_1) \wedge \mu(w_2)
\qquad 
\text{on } \cV_5 \otimes \Sym^2\!\cW.
\end{equation}
Families of normalized GM data of dimension $n$ form a category fibered in groupoids over the category  $\Sch/\k$ of schemes over $\k$; 
we denote it by $\ucMd_n$.


\begin{rema}
\label{remark:linear-lift}
 One could alternatively   define morphisms of families of GM data 
to be triples $(f,\widetilde\varphi_V,\widetilde\varphi_W)$, where $f$ is a morphism $S \to S'$ and
\begin{equation*}
\widetilde\varphi_V \colon \cV_6	\isomto f^*\cV'_6 \quad\text{and}\quad 
\widetilde\varphi_W \colon \cW 	\isomto f^*\cW'
\end{equation*}
are isomorphisms compatible with the subbundle $\cV_5 \subset \cV_6$ and the morphisms $\mu$ and $\bq$.\
 This   also defines a category fibered in groupoids over $\Sch/\k$, which we denote $\tucMd_n$ and call 
{\sf the category of linearized GM data}.

Denoting by $\widetilde\Aut(\ns,\cW,\cV_6,\cV_5,\mu,\bq)$ the automorphisms group  {scheme} in $\tucMd_n$,
 we obtain an embedding of group schemes
\begin{eqnarray}
\label{eq:gm-action-gm-data}
\Gm(S) &\lhra& \widetilde\Aut(\ns,\cW,\cV_6,\cV_5,\mu,\bq)\\
u &\longmapsto &(\widetilde\varphi_V = u,\ \widetilde\varphi_W=u^3).\nonumber
\end{eqnarray}
We have 
\begin{equation*}
\Aut(\ns,\cW,\cV_6,\cV_5,\mu,\bq) \cong \widetilde\Aut(\ns,\cW,\cV_6,\cV_5,\mu,\bq)/\Gm(S),
\end{equation*}
which essentially means that the fibered category of GM data is \emph{the rigidification} (\cite{acv,agv})
of the fibered category of linearized GM data with respect to the embeddings~\eqref{eq:gm-action-gm-data}.

This observation implies  that any morphism in $\ucMd_n$ over $f \colon S \to S'$
can be locally over $S'$ lifted to a morphism in $\tucMd_n$ (and such a lifting is unique up to   composition
with the action of $\Gm(S)$).\ We will frequently use these liftings.
 \end{rema}

\begin{lemm}
\label{lemma:gm-data-stack}
The fibered categories $\ucMd_n$ and $\tucMd_n$ are stacks over $\Sch/\k$.
\end{lemm}

\begin{proof}
For the fibered category $\tucMd_n$, this is a   consequence 
of the fact that quasicoherent sheaves form a stack in the fppf topology: 
  a family of linearized GM data is   a collection of quasicoherent sheaves and morphisms between them 
that satisfy some properties that are stable under base change.

For the fibered category $\ucMd_n$,   use~\cite[Theorem~5.1.5]{acv}.
\end{proof}

\subsection{Equivalence of stacks}
\label{subsection:gm=gm-data}

The main result of this section is a relation between the stack~$\ucM_n$ of smooth polarized GM varieties
and an open substack of the stack $\ucMd_n$ of normalized GM data.\
To define this substack, we use the notion of a GM intersection associated with a GM data set defined in~\cite[(2.8)]{DK}.

\begin{defi}
\label{def:gm-data-smooth}
A family $(\ns,\cW,\cV_6,\cV_5,\mu,\bq)$ of normalized GM data of dimension $n$ over a scheme $S$ is {\sf smooth} if 
for each  {geometric} point $s $ in $ S$, the GM intersection  
\begin{equation*}
 \bigcap_{v \in \cV_{6,s}}  \{ \bq(v) = 0  \} \subset \P(\cW_s) 
\end{equation*}
corresponding to the GM data set $(\cW_s,\cV_{6,s},\cV_{5,s},\mu_s,\bq_s)$ 
  {is} a smooth GM variety of dimension~$n$.
\end{defi}

By~\cite[Lemma~2.8]{DK}, a GM intersection corresponding to a GM data set of dimension~$n$ is a smooth GM variety of dimension~$n$
if and only if the GM intersection has dimension~$n$ and is smooth.\
As we will see in the proof of Lemma~\ref{lemma:gm-pointwise} below, $n$ is the expected dimension of the corresponding GM intersection, 
hence the condition for the GM intersection to be a smooth GM variety of dimension $n$ is open.\
Therefore, families of smooth normalized GM data of dimension~$n$ are classified by an open substack of $\ucMd_n$.

\begin{theo}
\label{theorem:gm-substack-gm-data}
For each $n\in\{1,\dots,6\}$, the stack of polarized GM varieties $\ucM_n$ is equivalent 
to the open substack of $\ucMd_n$ classifying families of smooth normalized GM data of dimension $n$.
 \end{theo}

\begin{proof}
Let $(\ns,\cW,\cV_6,\cV_5,\mu,\bq)$ be a smooth family of normalized GM data over a scheme~$S$ and let~$\pi_{\P_S(\cW)} \colon \P_S(\cW) \to S$ be the natural projection.\ We have
\begin{eqnarray*}
&& H^0\bigl( \P_S(\cW), \pi_{\P_S(\cW)}^*\bigl(\cV_6^\vee \otimes \det(\cV_5) \otimes (\cV_6/\cV_5)^{\otimes 2}\bigr) \otimes \cO_{\P_S(\cW)}(2) \bigr) 
\\ 
&\cong
& H^0\left(S, \cV_6^\vee \otimes \det(\cV_5) \otimes (\cV_6/\cV_5)^{\otimes 2} \otimes \pi_{\P_S(\cW)*}\cO_{\P_S(\cW)}(2) \right)
\\ 
&\cong
& H^0\left(S, \cV_6^\vee \otimes \det(\cV_5) \otimes (\cV_6/\cV_5)^{\otimes 2} \otimes \Sym^2\!\cW^\vee \right)
\\ 
&\cong  
& \Hom\left(\cV_6, \Sym^2\!\cW^\vee \otimes \det(\cV_5) \otimes (\cV_6/\cV_5)^{\otimes 2} \right).
\end{eqnarray*}
Thus, the morphism $\bq$ can be thought of as a global section
\begin{equation}
\label{eq:zero-locus}
\bq \in H^0\bigl( \P_S(\cW), \pi_{\P_S(\cW)}^*\bigl(\cV_6^\vee \otimes \det(\cV_5) \otimes (\cV_6/\cV_5)^{\otimes 2}\bigr) \otimes \cO_{\P_S(\cW)}(2) \bigr).
\end{equation}
Consider the subscheme $\cX \subset \P_S(\cW)$ defined as the zero locus of this global section.\ 
Define the morphism $\pi_\cX \colon \cX \to S$ as the restriction of the projection $\pi_{\P_S(\cW)} \colon \P_S(\cW) \to S$ 
and the polarization $\cH$ on $\cX$ as the restriction of the hyperplane class of $\P_S(\cW)$.\ 
Each  {geometric} fiber $(\cX_s,\cH\vert_{\cX_s})$   is a smooth polarized GM variety of dimension $n$.\ 
Moreover, the map~$\pi_\cX \colon \cX \to S$ is proper by definition and flat by Lemma~\ref{lemma:gm-pointwise} below.\
Since all fibers are smooth  {(by Definition~\ref{def:gm-data-smooth})}, the map $\pi_\cX$ is also smooth.\ 
Thus, $(\cX \to S,\cH)$ is a family of smooth polarized GM varieties.\ 
This construction together with a relative version of~\cite[Theorem~2.3]{DK} implies   
\begin{equation}
\label{eq:reconstruction-v-w}
\P_S(\cW) \cong \P_S((\pi_{\cX*}\cO_\cX(\cH))^\vee) \quad \textnormal{and}\quad
\P_S(\cV_6) \cong \P_S(\pi_{\P_S(\cW)*}\cI_{\cX/\P_S(\cW)}(2)).
\end{equation}

Similarly, given a morphism $(f,\varphi_W,\varphi_V)$ of families of GM data 
from $(\ns,\cW,\cV_6,\cV_5,\mu,\bq)$ to $(\ns',\cW',\cV'_6,\cV'_5,\mu',\bq')$, 
 we consider the isomorphism
 \begin{equation*}
\P_S(\cW) \xrightarrow{\ \varphi_W\ } \P_S(f^*\cW') \isomto \P_{S'}(\cW') \times_{S'} S.
\end{equation*}
Since $\varphi_V$ and $\varphi_W$ are compatible with $\bq$,
it induces a morphism $\varphi \colon \cX \to \cX'$ such that~\eqref{eq:morphism-gm} is a Cartesian square.\  
Moreover, by construction, we have   $\varphi^*\cH' = \cH$ in $\Pic_{\cX/S}(S)$.\  
Therefore, $(f,\varphi)$ is a morphism of  {families of} GM varieties.

This means that we have defined a morphism of stacks
\begin{equation*}
\zeta \colon \ucMd_{n,\mathrm{smooth}} \lra \ucM_n
\end{equation*}
from the open substack of $\ucMd_n$ classifying families of smooth normalized GM data of dimension $n$ 
to the stack $\ucM_n$ of smooth polarized GM varieties.\
It remains to prove that~$\zeta$ is an isomorphism of stacks.

Let us check that $\zeta$ is faithful: assume that $(f_1,\varphi_{1W},\varphi_{1V})$ and $(f_2,\varphi_{2W},\varphi_{2V})$ are   morphisms
between GM data $(\ns,\cW,\cV_6,\cV_5,\mu,\bq)$ and $(\ns',\cW',\cV'_6,\cV'_5,\mu',\bq')$ such that the corresponding 
morphisms $(f_1,\varphi_1)$ and $(f_2,\varphi_2)$ between  {the corresponding families of GM varieties~$\cX$ and~$\cX'$} are the same.\
Set $f := f_1 = f_2$ and~\mbox{$\varphi := \varphi_1 = \varphi_2$}.\
By construction of~$\varphi$, there is a commutative diagram
\begin{equation*}
\xymatrix@C=3em{
\cX \ar[rr]^-{\varphi} \ar[d] &&
 \cX' \ar[d]
\\
\P_S(\cW) \ar[r]^-{\varphi_{iW}} &
\P_{S}(f^*\cW') \ar[r] &
\P(\cW').
}
\end{equation*}
By~\eqref{eq:reconstruction-v-w}, the fiberwise linear span of $\cX$ in $\P_S(\cW)$ is $\P_S(\cW)$, 
hence $\varphi_{1W} = \varphi_{2W}$.\
Furthermore, there is a commutative diagram
\begin{equation*}
\xymatrix@C=5em{
\P_S(\cV_6) \ar[d]_\bq \ar[r]^{\varphi_{iV}} &
\P_{S'}(\cV'_6) \ar[d]^{\bq'}
\\
\P_S(\Sym^2\!\cW^\vee) &
\P_{S'}(\Sym^2\!\cW^{\prime\vee}) \ar[l]_-{\Sym^2\!\varphi_{iW}^\vee} 
}
\end{equation*}
in which the vertical arrows are embeddings (again by~\eqref{eq:reconstruction-v-w}) 
and the maps $\varphi_{iV}$ become isomorphisms after base change to $S$.\ Since we already have $\varphi_{1W} = \varphi_{2W}$, the equality~\mbox{$\varphi_{1V} = \varphi_{2V}$} follows.\ This proves faithfulness.

Next, we check that $\zeta$ is full.\
Assume $(\ns,\cW,\cV_6,\cV_5,\mu,\bq)$ and $(\ns',\cW',\cV'_6,\cV'_5,\mu',\bq')$ are families of smooth normalized GM data,
let $(\cX \to S,\cH)$ and $(\cX' \to S',\cH')$ be the corresponding families of GM varieties, and
let $(f,\varphi)$ be a morphism between them.\ We must show that it comes from a morphism of GM data.\ By the stack property and the faithfulness proved above, it is enough to prove this locally over $S'$.\ Moreover, applying base change along~$f$, we can assume that $S' = S$ and $f = \id_S$.\ Then $\varphi \colon \cX \to \cX'$ is an isomorphism, so we can identify $\cX$ and $\cX'$ via $\varphi$.

By   construction of the morphism $\zeta$ above, 
the line bundles $\cO_\cX(\cH) = \cO_{\P_S(\cW)}(1)\vert_\cX$ and $\cO_\cX(\cH') = \cO_{\P_S(\cW')}(1)\vert_{\cX}$ 
agree up to the pullback of a line bundle on $S$.\ Shrinking $S$ if necessary, we can assume that this line bundle is trivial, so we can choose an isomorphism
 \begin{equation*}
 \varphi_H \colon \cO_{\cX}(\cH) \isomto \cO_{\cX}(\cH').
\end{equation*}
 Using the formulas~\eqref{eq:reconstruction-v-w}, we see that $\varphi_H$ induces isomorphisms
$\varphi_W \colon \P_S(\cW) \isomto \P_S(\cW')$ and~$\varphi_V \colon \P_S(\cV_6) \isomto \P_S(\cV'_6)$.\
It is easy to see that these isomorphisms are compatible with the subbundle $\cV_5$
 and the morphisms $\mu$ and $\bq$,
 so that $(\id_S,\varphi_V,\varphi_W)$
is an isomorphism of GM data.\
Moreover, the isomorphism of GM varieties that it induces coincides with the one we started from.\
This proves fullness.
 
Finally, we check that $\zeta$ is essentially surjective.\ Given a family  $(\cX \to S,\cH)$ of smooth polarized GM varieties, 
we need to construct a family of smooth normalized GM data~$(\ns,\cW,\cV_6,\cV_5,\mu,\bq)$  
such that $(\cX \to S,\cH) = \zeta(\ns,\cW,\cV_6,\cV_5,\mu,\bq)$.\ 
Since we are dealing with stacks and since we already proved that $\zeta$ is fully faithful,  
it is enough to construct this locally over $S$.\ 
So we can assume that $\cH$ is the class of a line bundle on $\cX$.\ 
Denoting it by~$\cO_\cX(\cH)$ and following the proof of~\cite[Theorem~2.3]{DK},
we set 
 \begin{equation}
\label{eq:w-v-l}
\cWn := (\pi_{\cX*}\cO_\cX(\cH))^\vee,
\quad 
\cV_6 := \pi_{\P_S(\cWn)*}(\cI_\cX(2)),
\quad 
\cL := \pi_{\cX*} (\bw2\cU_\cX \otimes \cO_\cX(\cH)).
\end{equation}
These are vector bundles of respective ranks $n + 5$, $6$, and $1$ on $S$.

To be more precise, we first define the bundle $\cWn$ by the first equality in~\eqref{eq:w-v-l} and
note that the natural rational map $\cX \dra  \P_S(\cWn)$ is regular and   a closed embedding
(both statements can be verified fiberwise  and follow from~\cite[Theorem~2.3]{DK}).

Then, we define the bundle~$\cV_6$  {by} the second equality in~\eqref{eq:w-v-l} (here $\cI_\cX(2)$ is the twist of the ideal sheaf of~$\cX$ in $\P_S(\cWn)$ by the square of the Grothendieck bundle on $\P_S(\cWn)$).

Finally, we let $\cU_\cX$   be the excess conormal bundle for the embedding $\cX \to \P_S(\cWn)$
(see~\cite[Definition~A.1]{DK}) and define the line bundle $\cL$ by the third equality in~\eqref{eq:w-v-l}.\
By~\cite[Theorem~2.3]{DK} again, there is a natural  {fiberwise monomorphism $\cU_X \to \pi_\cX^*\cV_6$}
inducing a regular map $X \to \Gr_S(2,\cV_6)$
and   this map factors through $\Gr_S(2,\cV_5)$ for a unique subbundle $\cV_5 \subset \cV_6$ of rank~5.

Let us renormalize the bundle $\cWn$ by setting
\begin{equation*}
\cW := \cWn \otimes (\cV_6/\cV_5) \otimes \cL.
\end{equation*}
In order  to construct the morphisms $\mu$ and $\bq$, we  use again the construction of~\cite[Theorem~2.3]{DK} in a relative setting, which produces maps
\begin{equation*}
\mun \colon \cL \otimes \cWn \to \bw2\cV_5,
\qquad
\bqn \colon \cV_6 \to \Sym^2\!\cWn^\vee,
\qquad 
\varepsilon \colon \det(\cV_5) \isomto \cL^{\otimes 2}.
\end{equation*}
We then set 
\begin{equation*}
\mu = \mun \otimes \id_{\cV_6/\cV_5}
,\qquad
\bq = \varepsilon^{-1} \circ \bqn.
\end{equation*}
The relation~\eqref{eq:mu-q} is equivalent to the relation~\cite[(2.7)]{DK} (with $W$ replaced by $\cWn$),
which is proved in~\cite[Lemma~2.7]{DK}.\
Thus, we obtain a family of normalized GM data on~$S$.\
Finally,  by~\cite[Theorem~2.3]{DK} again, the family of GM varieties corresponding to this family of GM data 
is isomorphic to $(\cX \to S,\cH)$.
\end{proof}

The following lemma was used in the proof of Theorem~\ref{theorem:gm-substack-gm-data}.

\begin{lemm}
\label{lemma:gm-pointwise}
Let $(\ns,\cW,\cV_6,\cV_5,\mu,\bq)$ be a family of normalized GM data over a scheme $S$.\
Consider the subscheme~$\cX \subset \P_S(\cW)$ defined as the zero locus of the global section~\eqref{eq:zero-locus}
and assume that for every   {geometric} point $s $ in $ S$,
 the fiber $\cX_s$ of $\cX$ is a smooth GM variety of dimension~$n$.\ 
Then $\pi_\cX \colon \cX \to S$ is a \textup(flat\textup) family of smooth GM varieties.
\end{lemm}

\begin{proof}
We only have to check that the morphism $\pi_\cX \colon \cX \to S$ is flat.\  Locally over $S$, the rational linear projection $\mu \colon \P_S(\cW) \dashrightarrow \P_S(\bw2\cV_5)$ can be lifted to a linear closed embedding~$\P_S(\cW) \to \P_S(\cK \oplus \bw2\cV_5)$, where $\cK$ is a vector bundle over $S$.\ Consider the subscheme~$M_\cX \subset \P_S(\cW)$
defined as the zero locus of 
\begin{equation*}
\bq\vert_{\cV_5} \in H^0\bigl( \P_S(\cW), \pi_{\P_S(\cW)}^*\bigl(\cV_5^\vee \otimes \det(\cV_5) \otimes (\cV_6/\cV_5)^{\otimes 2}\bigr) \otimes \cO_{\P_S(\cW)}(2) \bigr).
\end{equation*}
By the commutativity of~\eqref{definition-gm-data}, this subscheme can be represented as
 \begin{equation*}
M_\cX = \P_S(\cW) \times_{\P_S(\sbw2\cV_5 \oplus \cK)} \Cone_{\P_S(\cK)}\Gr_S(2,\cV_5),
\end{equation*}
 where $\Cone_{\P_S(\cK)}\Gr_S(2,\cV_5) \subset \P_S(\cK \oplus \bw2\cV_5)$ is the cone over the relative Grassmannian~$\Gr_S(2,\cV_5)$ 
 with vertex $\P_S(\cK)$.\  Furthermore, on~$M_\cX$, the map $\bq$ defines a section 
\begin{equation*}
\bq_{\cV_6/\cV_5} \in H^0\bigl( M_\cX, \pi_{M_\cX}^*\bigl(\det(\cV_5) \otimes (\cV_6/\cV_5)\bigr) \otimes \cO_{\P_S(\cW)}(2)\vert_{M_\cX} \bigr) 
\end{equation*}
whose zero locus is   the subscheme $\cX \subset M_\cX \subset \P_S(\cW)$.

Since every fiber of  $\cX\to S$ has dimension $n$,  every fiber of $M_\cX$ has dimension at most~$n+1$.\ On the other hand, it is the intersection 
of a codimension-3 subvariety $\Cone_{\P_S(\cK)}\Gr_S(2,\cV_5) \subset \P_S(\bw2\cV_5 \oplus \cK)$ 
with the linear projective subbundle $\P_S(\cW)$ of dimension $n + 4$, hence each fiber has dimension at least $n + 1$.\ 
Combining these two observations, we see that each fiber
of $M_\cX$ has dimension $n + 1$,   hence the intersection (the fiber product) defining $M_\cX$ is dimensionally transverse.

Let us show that $M_\cX$ is flat over $S$.\  The cone $\Cone_{\P_S(\cK)}\Gr_S(2,\cV_5)$ is flat over~$S$  and $M_\cX$ is cut in it by  relative hyperplane sections.\ Since the intersection is dimensionally transverse, each of these hyperplanes decreases the dimension of fibers by 1, hence they form
a regular sequence at every fiber.\  This implies flatness of $M_\cX$ over $S$.

Finally, as observed above, $\cX$ is the zero locus of a global section of a line bundle on~$M_\cX$  and each fiber of $\cX$
has codimension 1 in the corresponding fiber of $M_\cX$.\ 
Therefore, this global section is not a zero divisor at every fiber, hence $\cX$ is also flat over $S$.
\end{proof}

We can restate Theorem~\ref{theorem:gm-substack-gm-data} as follows.

\begin{coro}
\label{corollary:stack-gm-data}
For each $n\in\{1,\dots,6\}$, the stack $\ucM_n$ of smooth polarized GM varieties of dimension~$n$ is equivalent 
to the stack $\ucMd_{n,\mathrm{smooth}}$ of smooth normalized GM data of dimension~$n$.
\end{coro}

From now on, we will  identify the stacks $\ucM_n$ and $\ucMd_{n,\mathrm{smooth}}$
by using the  equivalence above.\
In particular,  we will sometimes think of  {an $S$-point} of the stack $\ucM_n$ 
as   a family~$(\ns,\cW,\cV_6,\cV_5,\mu,\bq)$ of smooth normalized GM data over $S$.

\subsection{The ordinary and special substacks}
\label{subsection:ord-spe-substacks}

Let $(\cX \to S,\cH)$ be a family of smooth GM varieties.\ Consider the corresponding family $(\ns,\cW,\cV_6,\cV_5,\mu,\bq)$ of GM data.\ 
There is a commutative diagram
\begin{equation}
\label{definition-cokernel-sheaf}
\vcenter{\xymatrix@C=3em@R=0.5em{
\cV_5 \otimes  \det(\cV_5)^\vee \otimes ((\cV_6/\cV_5)^\vee)^{\otimes 2} \otimes \cW \ar[dr]^-\bq \ar[dd]_{\mu} & \\\
&\cW^\vee \\
\bw2\cV_5^\vee \otimes (\cV_6/\cV_5)^\vee \ar[ur]_-{\mu^T} ,
}}
\end{equation}
where the   vertical arrow is defined as the composition of the map $\mu$ with the canonical morphism
$\cV_5 \otimes \det(\cV_5)^\vee\otimes \bw2\cV_5 \to \bw2\cV_5^\vee$.\ The commutativity of~\eqref{definition-cokernel-sheaf} follows from the commutativity of~\eqref{definition-gm-data}.

\begin{lemm}\label{lemma15}
The cokernels of the horizontal maps in~\eqref{definition-cokernel-sheaf} are isomorphic.\ 
They are line bundles on a closed subscheme of $S$.
\end{lemm}

\begin{proof}
Let us show that the  arrow $\mu$ in~\eqref{definition-cokernel-sheaf} is surjective.\ 
By Lemma~\ref{lemma-epi-points}, this can be done pointwise, so it is enough to consider the case 
where $S $ is the spectrum of  {an algebraically closed} field~$\K$.\  
Then, $(\ns,\cW,\cV_6,\cV_5,\mu,\bq) $  is just a normalized GM data set $ (W,V_6,V_5,\mu,\bq)$ over~$\K$.

Assume that the map $\mu$ is not surjective.\ 
Trivializing  $\det (V_5)$ and $V_6/V_5$ for simplicity, we can rewrite the nonsurjectivity
condition as follows: there is an element $\xi \in \bw2V_5$ such that the subspace $W_0 = \Im(\mu) \subset \bw2V_5$ 
is orthogonal to the space $V_5 \wedge \xi \subset \bw3V_5$, that is, 
 \begin{equation*}
V_5 \wedge \xi \subset W_0^\perp.
\end{equation*}
The space $\P(V_5 \wedge \xi)$ contains quite a lot of decomposable vectors---if $\xi$
is decomposable, $\P(V_5 \wedge \xi) \cong \P^2$   consists of decomposable vectors only, 
while if $\xi$ has rank 4, $\P(V_5 \wedge \xi) \cong \P^4$   contains a $\P^3$ of decomposable vectors.\  But by~\cite[Proposition~3.13]{DK}, the space $W_0^\perp $ is equal to $A \cap \bw3V_5$, 
where $A \subset \bw3V_6$ is the Lagrangian subspace associated with $X$ by~\cite[Theorem~3.6]{DK}.\ By~\cite[Theorem~3.16]{DK}, for smooth GM varieties of dimension $n \ge 3$,
it contains no decomposable vectors, and for smooth GM curves and surfaces, it contains at most a  curve
of decomposable vectors (\cite[Remark~3.17]{DK}).\  The  arrow $\mu$ in~\eqref{definition-cokernel-sheaf} is therefore surjective.\ 

The isomorphism between the cokernels of the horizontal arrows then follows by abstract nonsense.\  Finally, the rank of the cokernel sheaves is at most 1   by~\cite[Proposition~2.28]{DK}.\  Therefore, it is a line bundle on a subscheme of $S$ by Lemma~\ref{lemma-support-cokernel}.
\end{proof}

\begin{defi}
\label{definition:gm-special-locus}
Given a GM data set over a scheme $S$, consider the cokernel sheaf 
\begin{equation}
\label{eq:cokernel-sheaf}
\cC := \Coker(\mu^T) \cong \Coker(\bq)
\end{equation} 
  discussed in Lemma~\ref{lemma15} and denote by 
$ S_{\gmspe} \subset S
 $
its (closed) scheme-theoretic support ({we call it} the {\sf GM-special locus}) and by $\cJ_{\gmspe} \subset \cO_S$
 its ideal.\ 
 
 By Lemma~\ref{lemma-support-cokernel}, the  scheme $S_\gmspe$ is the degeneration scheme for 
both morphisms $\mu^T$ and $\bq$ in~\eqref{definition-cokernel-sheaf}.\ We further define 
\begin{equation*}
S_{\gmord} := S \setminus S_{\gmspe}
\end{equation*}
to be the open complement of $S_{\gmspe}$ in $S$ ({we call it} the {\sf GM-ordinary locus}).
 \end{defi}

\begin{lemm}
\label{lemma:gm-ord-spe-substacks}
For each $n\in\{1,\dots,6\}$, there is an open substack $\ucM_{n,{\rm ord}} \subset \ucM_n$ 
and a closed substack $\ucM_{n,{\rm spe}} \subset \ucM_n$ such that 
\begin{equation*}
\begin{aligned}
\ucM_{n,{\rm ord}}(S) &= \{ (\cX \to S,\cH) \in \ucM_n(S) \mid S_{\rm spe} = \emptyset \},\\
\ucM_{n,{\rm spe}}(S) &= \{ (\cX \to S,\cH) \in \ucM_n(S) \mid S_{\rm spe} = S \}.
\end{aligned}
\end{equation*}
Moreover,  {$\ucM_{n,{\rm ord}}$ is the open complement of the closed substack $\ucM_{n,{\rm spe}} \subset \ucM_n$}.
\end{lemm}

\begin{proof}
It is enough to prove that the formation of the  {GM-ordinary} $S_{\gmord} \subset S$ and  {GM-special} $S_{\gmspe} \subset S$  {loci} is functorial in $S$,
that is, that it is compatible with base change.\ This follows from the fact that the formation of the cokernel sheaf commutes 
with base change (since the pullback functor is right exact).
\end{proof}

By~\cite[Section~2.5]{DK}, the open substack $\ucM_{n,{\rm ord}} \subset \ucM_n$ classifies 
families of smooth ordinary GM varieties of dimension $n$,
while the closed substack $\ucM_{n,{\rm spe}} \subset \ucM_n$ classifies families of smooth special GM varieties.

 {In the case $n = 2$, consider also the open substack 
\begin{equation}
\label{eq:ucm-ss}
 {\ucMtwo} \subset \ucM_{2,{\rm ord}}
\end{equation} 
that classifies strongly smooth ordinary GM surfaces (\cite[Definition~3.15]{DK}).}

\begin{lemm}
\label{eq:special-gerbe}
The stacks $\ucM_{n,{\rm ord}}$   for $n \in \{2,\dots,5\}$  and the stacks $\ucM_{n,{\rm spe}}$  for $n \in \{3,\dots,6\}$   
are smooth Deligne--Mumford stacks of finite type over $\Q$ with affine diagonals of finite type.\ We have
\begin{equation*}
\dim(\ucM_{n,{\rm ord}}) = 25 - (5-n)(6-n)/2 ,
\qquad
\dim(\ucM_{n,{\rm spe}}) = 25 - (6-n)(7-n)/2.
\end{equation*}
In particular,  for $n \in \{3,\dots,6\}$, the stack~$\ucM_{n,{\rm spe}}$ has codimension $6 - n$ in~$\ucM_n$.

 {For $n \ge 4$, the stack~$\ucM_{n,{\rm spe}}$ is a~$\bmu_2$-gerbe over the stack~$\ucM_{n-1,{\rm ord}}$, 
and the stack~$\ucM_{3,{\rm spe}}$ is a~$\bmu_2$-gerbe over the stack ${\ucMtwo}$.}
 \end{lemm}

\begin{proof}
Since $\ucM_{n,{\rm ord}}$ is an open substack in~$\ucM_n$, 
it inherits the properties of the latter established in Proposition~\ref{proposition:mgm-dm}; in particular, it has the same dimension.

We   show next that  {for $n \ge 4$}, the stack~$\ucM_{n,{\rm spe}}$ is a~$\bmu_2$-gerbe over the stack $\ucM_{n-1,{\rm ord}}$.\
Consider the $\Gm$-gerbes 
\begin{equation*}
\tucMd_{n,\mathrm{spe}} \to \ucM_{n,{\rm spe}}
,\qquad 
\tucMd_{n-1,\mathrm{ord}} \to \ucM_{n-1,{\rm ord}}
\end{equation*}
obtained by passing to linearized data  {(see Remark~\ref{remark:linear-lift})}, 
so that the arrows are rigidifications  {functors} with respect to the natural $\Gm$-actions.\
We will show that $\tucMd_{n,\mathrm{spe}}$ is a $\bmu_2$-gerbe over $\tucMd_{n-1,\mathrm{ord}}$ and that the corresponding
$\bmu_2$- and $\Gm$-actions on objects of $\tucMd_{n,\mathrm{spe}}$ commute.\
This will prove that, after passing to $\Gm$-rigidifications, 
there is a morphism of stacks~$\ucM_{n,{\rm spe}} \to \ucM_{n-1,{\rm ord}}$ which is a $\bmu_2$-gerbe.

To be more precise, we will  show that $\tucMd_{n,\mathrm{spe}}$ is the root stack over $\tucMd_{n-1,\mathrm{ord}}$
associated with its  line bundle $\det(\cV_6)$ in the sense of~\cite[Section~B.1]{agv} 
(in fact, this is the reason why we pass to stacks of linearized data,
since the line bundle~$\det(\cV_6)$ is just not defined on the stack $\ucMd_{n-1,\mathrm{ord}}$).\ 
For this, we check that the groupoid of linearized families~$(\ns,\cW,\cV_6,\cV_5,\mu,\bq)$ of special GM data of dimension~$n$ over a scheme $S$
is equivalent to the groupoid of linearized families $(\ns,\cW_0,\cV_6,\cV_5,\mu_0,\bq_0)$ of ordinary GM data of dimension $n-1$ 
equipped with a line bundle $\cW_1$ and an isomorphism 
\begin{equation*}
\bq_1 \colon \Sym^2\!\cW_1  \isomto \det(\cV_6).
\end{equation*}

Indeed, given a family of special linearized GM data over a scheme $S$, we set 
\begin{equation*}
\cW_1 := \Ker(\cW \xrightarrow{\ \mu\ } \bw2\cV_5 \otimes (\cV_6/\cV_5)).
\end{equation*}
This is a line bundle because, by definition of special GM data  {and Lemma~\ref{lemma15}}, 
the map $\mu$ has constant rank~$n+4$, while $\cW$ is a vector bundle of rank $n+5$.\ Furthermore, we consider the map
\begin{equation*}
\cV_6 \otimes \Sym^2\!\cW_1  \hookrightarrow 
\cV_6 \otimes \Sym^2\!\cW  \xrightarrow{\ \bq\ } 
\det(\cV_5) \otimes (\cV_6/\cV_5)^{\otimes2} \isomto 
\det(\cV_6) \otimes (\cV_6/\cV_5).
\end{equation*}
By~\eqref{eq:mu-q}, this map vanishes on the subbundle $\cV_5 \otimes \Sym^2\!\cW_1 \subset \cV_6 \otimes \Sym^2\!\cW_1$ 
hence factors through a morphism $(\cV_6/\cV_5) \otimes \Sym^2\!\cW_1 \to \det(\cV_6) \otimes (\cV_6/\cV_5)$.\  Twisting it by~$(\cV_6/\cV_5)^\vee$, we obtain a morphism 
\begin{equation*}
\Sym^2\!\cW_1  \lra \det(\cV_6)
\end{equation*}
which we denote by $\bq_1$.\  
It is surjective by  {Lemma~\ref{lemma-epi-points} and}~\cite[Lemma~2.33]{DK}, hence an isomorphism by Lemma~\ref{lemma-kernel-cokernel-lf}, 
since   its source and target are both line bundles.\ 
Finally,  by~\cite[Proposition~2.30]{DK}, there is a canonical direct sum decomposition $\cW \cong \cW_0 \oplus \cW_1$ 
(orthogonal with respect to all quadrics).\  
We denote by $\mu_0$ and $\bq_0$ the restrictions of~$\mu$ to~$\cW_0$ and of~$\bq$ to~$\cV_6 \otimes \Sym^2\!\cW_0 $.\  
By~\cite[Lemma~2.33]{DK}, $(\ns,\cW_0,\cV_6,\cV_5,\mu_0,\bq_0)$ is a linearized family of ordinary GM data of dimension $n-1$.\  
This defines a functor between the groupoids   (the action of the functor on morphisms is obvious).

Conversely, assume we are given a family $(\ns,\cW_0,\cV_6,\cV_5,\mu_0,\bq_0)$ of ordinary GM data of dimension $n-1$, 
a line bundle $\cW_1$, and an isomorphism $\bq_1\colon\Sym^2\!\cW_1  \isomto \det(\cV_6)$.\
We set 
\begin{equation*}
 \cW := \cW_0 \oplus \cW_1, 
 \qquad 
 \mu = (\mu_0,0),
 \qquad
 \bq = \bq_0 \oplus \bq_1,
\end{equation*}
where, by an abuse of notation, we denote by $\bq_1$ the map 
\begin{equation*}
\cV_6 \otimes \Sym^2\!\cW_1  \twoheadrightarrow 
(\cV_6/\cV_5) \otimes \Sym^2\!\cW_1 \xrightarrow{\ \bq_1\ } 
(\cV_6/\cV_5) \otimes \det(\cV_6) \isomto \det(\cV_5) \otimes (\cV_6/\cV_5)^{\otimes 2}.
\end{equation*}
By~\cite[Lemma~2.33]{DK}, this is a family of smooth special GM data of dimension $n$.

It is straightforward to see that the constructed functors are mutually inverse, so the groupoids are equivalent, 
hence $\tucMd_{n,\mathrm{spe}}$ is the root stack over $\tucMd_{n-1,\mathrm{ord}}$ 
and $\ucM_{n,{\rm spe}}$ is a~$\bmu_2$-gerbe over $\ucM_{n-1,{\rm ord}}$.

For $n = 3$, the argument is the same; the only difference is that the ordinary GM surface associated 
with a smooth special GM threefold is automatically strongly smooth {(see~\cite[Section~3.4]{DK})}.

This implies that $\ucM_{n,{\rm spe}}$ is a smooth Deligne--Mumford stack and gives its dimension.\
 {The other properties of $\ucM_{n,{\rm spe}}$ follow from Proposition~\ref{proposition:mgm-dm}, since it is a closed substack in~$\ucM_n$.}
The statement about the codimension is   a simple computation.
\end{proof}

\begin{rema}
\label{remark:special-rigidification}
The proof of Lemma~\ref{eq:special-gerbe} shows that the automorphism group scheme 
of each object of the stack $\ucM_{n,{\rm spe}}$ contains the constant group scheme $\bmu_2$
and that the morphism of stacks~$\ucM_{n,{\rm spe}} \to \ucM_{n-1,{\rm ord}}$ is the $\bmu_2$-rigidification.
\end{rema}

\subsection{Lagrangian data}
\label{subsection:lagrangian-data}

In~\cite[Section~3]{DK}, we explained the relation between GM and Lagrangian data sets.\ We now define   families of Lagrangian data and show that they form a stack.

\begin{defi}
\label{definition:lagrangian-data}
A {\sf family of Lagrangian data over a scheme $S$} is a   quadruple  $(\ns,\cV_6,\cV_5,\cA)$, 
where~$\cV_6$ is a vector bundle of rank~6 on $S$, 
$\cV_5 \subset \cV_6$ is a subbundle of corank~1, 
and~\mbox{$\cA \subset \bw3\cV_6$} is a Lagrangian subbundle.

{A morphism} between families of Lagrangian data  $(\ns,\cV_6,\cV_5,\cA)$ and $(\ns',\cV'_6,\cV'_5,\cA')$   
is a pair $(f,\varphi)$ fitting into a Cartesian square
\begin{equation*}
\xymatrix{
\P_S(\cV_6) \ar[r]^-{\varphi} \ar[d] & 
\P_{S'}(\cV'_6) \ar[d]
\\
S \ar[r]^-f &
S'
}
\end{equation*}
and such that $\varphi(\P_S(\cV_5)) = \P_{S'}(\cV'_5)$ and $(\bw3\varphi)(\P_S(\cA)) = \P_{S'}(\cA')$.
\end{defi}

Families of Lagrangian data form a category fibered in groupoids over the category $\Sch/\k$, which we denote $\cmlag$.

\begin{rema}
\label{remark:linear-lagrangian-lift}
As for GM data (see Remark~\ref{remark:linear-lift}), 
we can define a category~$\tcmlag$ of families of linearized Lagrangian data (fibered in groupoids over $\Sch/\k$)
with the same objects as in~$\cmlag$ but with morphisms defined as pairs $(f,\widetilde\varphi)$ 
formed by a morphism $f \colon S \to S'$ and an isomorphism $\widetilde\varphi \colon \cV_6 \isomto f^*\cV'_6$
such that $\widetilde\varphi(\cV_5) = f^*\cV'_5$ and~$(\bw3{\widetilde\varphi})(\cA) = f^*\cA'$.

Denoting by $\widetilde\Aut(\ns,\cV_6,\cV_5,\cA)$ the automorphism group scheme in $\tcmlag$,
we obtain an embedding of group schemes
\begin{equation}
\label{eq:gm-action-lagrangian-data}
\Gm(S) \lhra \widetilde\Aut(\ns,\cV_6,\cV_5,\cA) 
\end{equation}
that takes an invertible function $u$ to the automorphism given by $\widetilde\varphi = u$.\ We have 
\begin{equation*}
\Aut(\ns,\cV_6,\cV_5,\cA) \cong \widetilde\Aut(\ns,\cV_6,\cV_5,\cA)/\Gm(S) 
\end{equation*}
and the fibered category of Lagrangian data is the rigidification of the fibered category of linearized Lagrangian data
with respect to the embeddings~\eqref{eq:gm-action-lagrangian-data}.

This observation implies  that any morphism in $\cmlag$ over $f \colon S \to S'$
can be locally over $S'$ lifted to a morphism in $\tcmlag$ (and such a lifting is unique up to the composition with the action of $\Gm(S)$).
In what follows, we will frequently use such a lifting.
\end{rema}

 {The argument of the proof of Lemma~\ref{lemma:gm-data-stack} implies} the following.

\begin{lemm}
The fibered categories $\cmlag$ and $\tcmlag$ are stacks over $\Sch/\k$.
\end{lemm}

Given a family of Lagrangian data $(\ns,\cV_6,\cV_5,\cA)$, 
we consider the natural epimorphism 
\begin{equation*}
\lambda \colon \cV_6 \lra \cV_6/\cV_5.
\end{equation*}
For each $p\in\{1,\dots,6\}$, it extends by the Leibniz rule to an epimorphism 
\begin{equation*}
\lambda_p \colon \bw{p}\cV_6 \lra \bw{p-1}\cV_5 \otimes (\cV_6/\cV_5)
\end{equation*}
whose kernel is the subbundle $\bw{p}\cV_5 \subset \bw{p}\cV_6$.

\begin{defi}\label{def317}
We say that a family of Lagrangian data $(\ns,\cV_6,\cV_5,\cA)$ 
 {\sf has rank $k$} if  the composition
\begin{equation}
\label{eq:varphi}
\varphi\colon \cA \lhra \bw3\cV_6 \xrightarrow{\ \lambda_3\ } \bw2\cV_5 \otimes (\cV_6/\cV_5)
\end{equation}
has rank $k$ and 
$\bw{k-1}\varphi_s$ does not vanish for any  {geometric} point   $s$ in $ S$.\ 
We say that the Lagrangian data {\sf avoids decomposable vectors} if, for each  {geometric} point $s $ in $ S$, 
the Lagrangian subspace $\cA_s \subset \bw3\cV_{6,s}$ contains no decomposable vectors, that is, $\P(\cA_s) \cap \Gr(3,\cV_{6,s}) = \emptyset$.
\end{defi}

The above two conditions define a  {locally closed} substack in $\cmlag$ classifying families  {of} Lagrangian data of rank $k$ avoiding decomposable vectors.\
We denote it by $\cmlag_k$.\
{We will show later (Proposition~\ref{proposition:cmlag-quotient}) that this stack is a global quotient stack.}

{Finally, we define the special and ordinary loci for Lagrangian data.}

\begin{defi}
\label{definition:lagrangian-special-locus} 
Given a family $(\cV_6,\cV_5,\cA)$ of Lagrangian data on a scheme $S$, of rank $k$, we denote by
\begin{equation*}
S_\lagspe \subset S
 \end{equation*}
the degeneracy locus of the composition $\varphi$
 and   by $\cJ_\lagspe \subset \cO_S$ its ideal   (it is generated by the  $k\times k$-minors of $\varphi$).\
We call $S_\lagspe$ the {\sf Lagrangian special locus} of $S$.\
Its complement 
\begin{equation*}
S_\lagord = S \setminus S_\lagspe
\end{equation*}
is called the {\sf Lagrangian ordinary locus}.
\end{defi}

As in the proof  of Lemma~\ref{lemma:gm-ord-spe-substacks}, this gives rise to a closed substack $\cmlag_{k,{\rm spe}} \subset \cmlag_k$ 
of special Lagrangian data and an open substack $\cmlag_{k,{\rm ord}} \subset \cmlag_k$ of ordinary Lagrangian data, such that
 $\cmlag_{k,{\rm ord}}$ is the open complement of $\cmlag_{k,{\rm spe}} \subset \cmlag_k$.

\section{Relation between families of GM and Lagrangian data}\label{section-relation}

We consider below two constructions relating GM data to Lagrangian data.\
We pay special attention to the relation between their special loci.

\subsection{From families of GM data to families of Lagrangian data}
\label{subsection-from-gm-to-lag}

Let  $(\ns,\cW,\cV_6,\cV_5,\mu,\bq)$ be a family of smooth normalized GM data.\ 
We construct over the same scheme $S$ an associated family of Lagrangian data avoiding decomposable vectors (Definition~\ref{def317}).
 
Our construction  is a relative (and normalized) version of the construction of the proof of~\cite[Theorem~3.6]{DK} 
(with ``the odd part'' omitted).\
We consider the diagram
\begin{equation}\label{equation-ax-monad}
{\xymatrix@R=6mm
@C=3mm{
\cV_5 \otimes \cW \otimes (\cV_6/\cV_5)^\vee \ar[rr]^-{f_1} && 
\bw3{\cV_5} \oplus (\cV_6\otimes \cW \otimes (\cV_6/\cV_5)^\vee) \ar[rr]^-{f_2} \ar[d]^{f_3} && 
\cW^\vee \otimes \det(\cV_6) \\
&& \bw3 \cV_6 
}}
\end{equation}
with morphisms defined by
\begin{eqnarray*}
 f_1(v\otimes w) &=& (-v \wedge \mu(w), v\otimes w),\\
 f_2(\xi, v\otimes w)(w') &=& \xi \wedge \mu(w') + \bq(v)(w,w'),\\
 f_3(\xi,v\otimes w) & =& \xi + v\wedge \mu(w)
\end{eqnarray*}
(we omit factors corresponding to line bundles, {which do not matter   here}).\
We have $f_2 \circ f_1 = 0$  {by~\eqref{eq:mu-q}}
and $f_3 \circ f_1 = 0$.\ If we set
\begin{equation}
\label{definition-ca}
\cA := \Ker(f_2)/\Im(f_1), 
\end{equation}
 the  morphism  $f_3$ induces a morphism $\cA \to \bw3\cV_6$.\ 

\begin{prop}
\label{proposition-gm-to-lag}
Let $(\ns,\cW,\cV_6,\cV_5,\mu,\bq)$ be a family of smooth normalized GM data of dimension $n \ge 3$.\ 
Define  $\cA$  by~\eqref{definition-ca}.\ 
Then $(\ns,\cV_6,\cV_5,\cA)$ is a family of Lagrangian data of rank $n+5$   avoiding decomposable vectors.\ 
This defines a morphism of stacks
\begin{eqnarray*}
\fa \colon \ucM_n &\lra& \cmlag_{n+5}\\
(\ns,\cW,\cV_6,\cV_5,\mu,\bq) &\longmapsto& (\ns,\cV_6,\cV_5,\cA).
\end{eqnarray*}
Moreover, the Lagrangian and GM special loci in~$S$ coincide set-theoretically 
but not {necessarily} scheme-theoretically: we have 
\begin{equation*}
\cJ_\lagspe = \cJ_\gmspe^2,
\end{equation*}
that is, the ideal of the Lagrangian special locus is the square of the ideal of the GM special locus.
\end{prop}

\begin{proof}
Checking that $\cA$ is a vector bundle of rank~10 
(it is enough for that to check that~$f_2$ is an epimorphism and that $f_1$ is a fiberwise monomorphism)
and that the map $\cA \to \bw3\cV_6$ induced by $f_3$ is a fiberwise monomorphism can be done pointwise and thus follows from the proof of~\cite[Theorem~3.6]{DK}.

We now show that $\cA$ has the Lagrangian property, that is,  that the composition
\begin{equation}
\label{compo}
\cA \otimes \cA \lra \bw3\cV_6 \otimes \bw3\cV_6 \xrightarrow{\ \wedge\ } \det(\cV_6)
\end{equation}
vanishes identically.\
It is not enough to prove this property  pointwise, since the scheme~$S$ might be nonreduced, but it is enough to check it locally.\
 It will be convenient to compose~{\eqref{compo}} with the isomorphism $\lambda_6 \colon \det(\cV_6) \isomto \det(\cV_5) \otimes (\cV_6/\cV_5)$.\
We will also use the definition~\eqref{definition-ca} of~$\cA$ and the fact that the map $\cA \to \bw3\cV_6$ is induced by $f_3$.\
The resulting composition 
\begin{multline*}
\qquad
(\bw3{\cV_5}  \oplus (\cV_6\otimes \cW \otimes (\cV_6/\cV_5)^\vee)) \otimes (\bw3{\cV_5} \oplus (\cV_6\otimes \cW \otimes (\cV_6/\cV_5)^\vee)) 
\\ \xrightarrow{\ f_3 \wedge f_3\ } 
\det(\cV_6) \xrightarrow{\ \lambda_6\ }
\det(\cV_5) \otimes (\cV_6/\cV_5)
\qquad
\end{multline*}
is given by
\begin{equation}
\label{form}
(\xi_1,v_1\otimes w_1) \otimes (\xi_2,v_2 \otimes w_2) \longmapsto
\lambda_6\bigl((\xi_1 + v_1\wedge \mu(w_1)) \wedge (\xi_2 + v_2 \wedge \mu(w_2))\bigr).
\end{equation}
It is thus enough to check that~\eqref{form} vanishes on $\Ker(f_2) \otimes \Ker(f_2)$.

Since $\xi_1$ and $\xi_2$ are sections of $\bw3\cV_5$, we have $\xi_1\wedge \xi_2 = 0$.\ 
Choosing locally a direct sum decomposition $\cV_6 = \cV_5 \oplus (\cV_6/\cV_5)$  {and a generator $v_0$ for the second summand}, 
we can write 
\begin{equation*}
v_i = v'_i + \lambda(v_i) {v_0},
\qquad\textnormal{with }
v'_i \in \cV_5,
\end{equation*}
 {where we think of $\lambda(v_i)$ as of a scalar}.\
We can rewrite the right side of \eqref{form} as
\begin{equation*}
\begin{aligned}
\lambda_6(\xi_1\wedge & v_2 \wedge \mu(w_2)) + \lambda_6(v_1 \wedge \mu(w_1) \wedge \xi_2) + \lambda_6(v_1 \wedge \mu(w_1) \wedge v_2 \wedge \mu(w_2)) \\
 = & -\lambda(v_2) \xi_1 \wedge \mu(w_2) + \lambda(v_1) \xi_2 \wedge \mu(w_1) 
\\
& \hspace{5em}{} + \lambda(v_1) v'_2 \wedge \mu(w_1) \wedge \mu(w_2) - \lambda(v_2) v'_1 \wedge \mu(w_1) \wedge \mu(w_2) \\
= & -\lambda(v_2) (\xi_1 \wedge \mu(w_2) + \bq(v'_1)(w_1,w_2)) + \lambda(v_1) (\xi_2 \wedge \mu(w_1) + \bq(v'_2)(w_1,w_2)) \\
= &   -\lambda(v_2) (\xi_1 \wedge \mu(w_2) + \bq(v_1)(w_1,w_2)) + \lambda(v_2)\lambda(v_1)\bq(v_0)(w_1,w_2)
\\
&  \hspace{5em} {}+ \lambda(v_1) (\xi_2 \wedge \mu(w_1) + \bq(v_2)(w_1,w_2)) - \lambda(v_1)\lambda(v_2)\bq(v_0)(w_1,w_2)  \\
= & - \lambda(v_2) f_2(\xi_1,v_1 \otimes w_1)(w_2) + \lambda(v_1) f_2(\xi_2,v_2 \otimes w_2)(w_1)
\end{aligned}
\end{equation*}
(in the first equality, we use the Leibniz rule for $\lambda_6$ and the fact that $\lambda$ vanishes on~$\xi_i$ and on~$\mu(w_i)$,
as well as the relation $\lambda_2(v_1 \wedge v_2)  = \lambda(v_1)v_2 - \lambda(v_2)v_1= \lambda(v_1)v'_2 - \lambda(v_2)v'_1$; 
in the second equality, we use~\eqref{eq:mu-q}; in the third equality,  we use the definition of $v'_i$;
and~in the last equality, we use the definition of $f_2$ and cancel out two summands equal to $\pm\lambda(v_1)\lambda(v_2)\bq(v_0)(w_1,w_2)$).\ 
It follows that the map \eqref{form} vanishes identically on  {the subbundle}~$\Ker(f_2) \otimes \Ker(f_2)$, 
hence the induced map vanishes identically on $\cA$.

Consider now the map $\varphi$ defined by~\eqref{eq:varphi}.\
It is induced by the composition of the maps in the top row and the right column of the diagram
\begin{equation}\label{diagram-factorization-1}
\vcenter{\xymatrix@M=6pt@C=15pt{
\Ker(f_2) \ar@{^{(}->}[r] \ar@{->>}[d] &
\bw3{\cV_5}  \oplus (\cV_6\otimes \cW \otimes (\cV_6/\cV_5)^\vee) \ar[r]^-{f_3} \ar[d]^{(0,\lambda)} &
\bw3\cV_6 \ar[d]^{\lambda_3} \\
\cA \ar [r]^\nu &
\cW \ar[r]^-\mu &
\bw2\cV_5 \otimes (\cV_6/\cV_5).
}}
\end{equation}
The right square of the diagram is commutative because $\lambda_3$ vanishes on $\bw3\cV_5$ and $\lambda_2$ vanishes on~$\bw2\cV_5$.\
The   arrow $(0,\lambda)$   vanishes on $\Im(f_1) \subset \Ker(f_2)$
 hence factors   through $\cA$, thus defining the   arrow $\nu$.\ Therefore, we obtain a commutative diagram
\begin{equation}
\label{diagram-a-w-v}
\vcenter{\xymatrix@C=3em@M=6pt{
\cA \ar@{^{(}->}[r] \ar[d]_\nu\ar[dr]^\varphi & \bw3\cV_6 \ar[d]^{\lambda_3} \\
\cW \ar[r]^-{\mu} & \bw2\cV_5 \otimes (\cV_6/\cV_5).
}}
\end{equation} 
The rank of $\cW$ is $n+5$, hence the rank of  $\varphi$ is at most $n+5$.\
The fact that it is at least~$n + 4$ at each  {geometric} point of $S$ can be verified pointwise and follows from~\cite[(3.9)]{DK}.\
Also,~\cite[(3.9) and Theorem~3.16]{DK} proves that~$\cA$ has no decomposable vectors
(this is the only place where we use the condition $n \ge 3$).\ {Thus, $(\ns,\cV_6,\cV_5,\cA)$ is a family of Lagrangian data of rank~\mbox{$n + 5$} avoiding decomposable vectors.}

Let us show that the association $\fa \colon \ucM_n \to \cmlag_{n+5}$ 
that takes a family of smooth normalized GM data $(\ns,\cW,\cV_6,\cV_5,\mu,\bq)$
to the family of Lagrangian data $(\ns,\cV_6,\cV_5,\cA)$, where~$\cA$ is defined by~\eqref{definition-ca}, is a morphism of stacks,
 {meaning that it is defined on morphisms.}

Assume for simplicity that $f = \id_S$ (the general case reduces to this by base change).\
A morphism of families of GM data from~$(\ns,\cW,\cV_6,\cV_5,\mu,\bq)$ to $(\ns,\cW',\cV'_6,\cV'_5,\mu',\bq')$ is then given by a pair of isomorphisms
$\varphi_V \colon \P_S(\cV_6) \isomto \P_S(\cV'_6)$ and $\varphi_W \colon \P_S(\cW) \isomto \P_S(\cW')$ over~$S$.\
The first  can be lifted to an isomorphism
\begin{equation*}
\widetilde\varphi_V \colon \cV_6 \isomto \cV'_6 \otimes \delta
\end{equation*}
for an appropriate line bundle $\delta$ on $S$.\
Using compatibility with the morphism $\mu$, we conclude that $\varphi_W$ lifts to an isomorphism
\begin{equation*}
\widetilde\varphi_W \colon \cW \isomto \cW' \otimes \delta^{\otimes 3}.
\end{equation*}
It is straightforward to see that the pair $(\widetilde\varphi_V,\widetilde\varphi_W)$ defines a morphism 
from the diagram~\eqref{equation-ax-monad} to the analogous diagram for the family of GM data $(\ns,\cW',\cV'_6,\cV'_5,\mu',\bq')$ twisted by $\delta^{\otimes 3}$.
It follows that the morphism $\bw3{\widetilde\varphi}_V \colon \bw3\cV_6 \to \bw3\cV'_6 \otimes \delta^{\otimes 3}$ takes $\cA$ to $\cA' \otimes \delta^{\otimes 3}$.\
Therefore, we have $\bw3\varphi_V(\P_S(\cA)) = \P_S(\cA')$, hence $\varphi_V$ is a morphism between the associated families of Lagrangian data.\
This operation is compatible with compositions of morphisms and takes the identity to the identity, hence $\fa$ is a morphism of stacks.

Finally, consider the special locus of the family of Lagrangian data constructed above.\
Its ideal $\cJ_\lagspe$ is generated by the $(n+5) \times (n+5)$-minors of the map $\varphi$  {defined by~\eqref{eq:varphi}}.\ Because of the factorization in~\eqref{diagram-a-w-v} (and since the rank of $\cW$ is $n + 5$),
every such minor is the product of a minor of $\nu$ and a minor of $\mu$ of the same size.\ Consequently, the ideal~$\cJ_\lagspe$ is the product of two ideals, one generated by the minors of $\nu$
and the other generated by the minors of $\mu$.\  The latter ideal is by definition  equal to the ideal $\cJ_\gmspe$ defining the special GM locus.\  To finish the proof, we must show that the minors of $\nu$ generate the same ideal.

Since the left vertical arrow in~\eqref{diagram-factorization-1} is surjective, this ideal coincides with the 
ideal generated by the minors of the map $\Ker(f_2) \to \cW$, that is, by Lemma~\ref{lemma-support-cokernel},
with the annihilator of the cokernel of this map.\ Since $f_2$ is surjective, 
the cokernel of the map~$\Ker(f_2) \xrightarrow{ (0,\lambda)  } \cW$
is isomorphic to the cokernel of the map 
\begin{equation*}
\bw3\cV_5 \oplus (\cV_6 \otimes \cW \otimes (\cV_6/\cV_5)^\vee) \xrightarrow{\ (0,\lambda) + f_2\ } \cW \oplus (\cW^\vee \otimes \det(\cV_6)).
\end{equation*}
Since $(0,\lambda)$ is surjective, this sheaf is isomorphic 
to the cokernel of the map $\Ker(0,\lambda) \xrightarrow{\ f_2\ } \cW^\vee \otimes \det(\cV_6)$.\ 
Altogether, this means  
\begin{equation*}
\Coker(\nu) \cong 
\Coker\bigl(
\bw3\cV_5 \oplus (\cV_5 \otimes \cW \otimes (\cV_6/\cV_5)^\vee) \xrightarrow{\ (\mu,\bq)\ }\cW^\vee \otimes \det(\cV_6)
\bigr).
\end{equation*}
By Lemma~\ref{lemma15}, the images of   the two components of this map coincide,
hence the cokernel of their sum equals to the cokernel of each of them, that is, 
\begin{equation}
\label{eq:cokernel-nu}
\Coker(\nu) \cong \cC,
\end{equation}
where $\cC$ is the cokernel sheaf of the family of GM data as defined in~\eqref{eq:cokernel-sheaf}.\ 
The annihilator of $\Coker(\nu)$ is thus again the ideal $\cJ_\gmspe$ of the GM special locus.
\end{proof}

\begin{rema}
\label{remark:family-ordinary-surfaces}
The argument of Proposition~\ref{proposition-gm-to-lag} also applies to families of smooth normalized GM data 
of dimension $n = 2$ such that the corresponding GM varieties are strongly smooth (\cite[Definition~3.15]{DK}) 
ordinary GM surfaces since, by~\cite[(3.9) and Theorem~3.16]{DK}, the corresponding Lagrangian subspaces have no decomposable vectors.\ It defines a morphism of stacks $\ucMtwo \to \cmlag_7$.
\end{rema}

We will need some properties of the construction presented above.\
Consider the family 
\begin{eqnarray}
\label{equation-qa}
\bq_\cA\colon \cV_6 \otimes \Sym^2\!\cA & \lra &\det(\cV_5) \otimes (\cV_6/\cV_5)^{ \otimes 2}
\\
v \otimes a_1\otimes a_2 & \longmapsto & -\lambda_4(v\wedge a_1) \wedge \lambda_3(a_2)\nonumber
\end{eqnarray}
of quadratic forms on $\cA$ (the formula is symmetric in $\xi_1$ and $ \xi_2$ by the Lagrangian property of $\cA$; 
see the proof of~\cite[Theorem~3.6]{DK}  
for details).

\begin{lemm}\label{lemma-quadratic-compatible}
The quadratic form on $\cA$ defined by~\eqref{equation-qa} is equal to the form induced by $\bq$ via the map $\nu\colon\cA \to \cW$.
\end{lemm}

\begin{proof}
Let $v$ be a local section of $\cV_6$ and let $a_1$ and $a_2$ be local sections of $\cA$.\ Choose a lift of $a_i$ to a local section  $(\xi_i, v_i\otimes w_i)$
of $\Ker(f_2) \subset \bw3\cV_5 \oplus (\cV_6 \otimes \cW \otimes (\cV_6/\cV_5)^\vee)$.\
We have
\begin{equation*}
 \lambda_4(v \wedge f_3(\xi_1, v_1 \otimes w_1)) = 
 \lambda_4(v \wedge (\xi_1 + v_1 \wedge \mu(w_1))) = 
 \lambda(v)\xi_1 + \lambda_2(v \wedge v_1)\wedge \mu(w_1) 
\end{equation*}
and
\begin{equation*}
 \lambda_3(f_3(\xi_2, v_2 \otimes w_2)) = 
 \lambda_3(\xi_2 + v_2 \wedge \mu(w_2)) = \lambda(v_2)\mu(w_2).
\end{equation*}
Therefore,
\begin{eqnarray*}
 \bq_\cA(v)(a_1, a_2) &=& -\lambda_4(v \wedge f_3(\xi_1, v_1 \otimes w_1)) \wedge \lambda_3(f_3(\xi_2, v_2 \otimes w_2)) \\
 &=& -(\lambda(v)\xi_1 + \lambda_2(v \wedge v_1)\wedge \mu(w_1)) \wedge \lambda(v_2)\mu(w_2) \\
 &=& -\lambda(v_2)\lambda(v)\xi_1 \wedge \mu(w_2) - \lambda(v_2)\lambda_2(v \wedge v_1) \wedge \mu(w_1) \wedge \mu(w_2).
\end{eqnarray*}
On the other hand, since $(\xi_1, v_1\otimes w_1)$ is in the kernel of $f_2$, we have $\xi_1 \wedge \mu(w_2) = -\bq(v_1)(w_1,w_2)$.
Using this and~\eqref{eq:mu-q}, the above equals
\begin{multline*}
\qquad
 \lambda(v_2)\lambda(v)\bq(v_1)(w_1,w_2) - \lambda(v_2)\lambda_2(v \wedge v_1) \wedge \mu(w_1) \wedge \mu(w_2) = \\ 
 \bq(\lambda(v_2)\lambda(v)v_1 - \lambda(v_2)\lambda_2(v \wedge v_1))(w_1,w_2).
 \qquad
\end{multline*}
It remains to observe that $\lambda(v_2)\lambda(v)v_1 - \lambda(v_2)\lambda_2(v \wedge v_1) = \lambda(v_1)\lambda(v_2)v$,
so that finally 
\begin{equation*}
 \bq_\cA(v)(a_1, a_2) = 
 \lambda(v_1)\lambda(v_2)\bq(v)(w_1,w_2) = 
 \bq(v)(\lambda(v_1)w_1,\lambda(v_2)w_2) = 
 \bq(v)(\nu(a_1),\nu(a_2)).
\end{equation*}
This is precisely the compatibility we were claiming.
\end{proof}

\begin{lemm}\label{lemma-factorization2}
Assume that the GM special locus $S_{\gmspe}$ of a smooth family of normalized GM data is a Cartier divisor.\ The map~\eqref{eq:varphi} of the corresponding family of Lagrangian data then factors as
\begin{equation*}
\xymatrix{
\cA \ar@{->>}[dr] \ar[rr]^-\nu && \cW \ar[dr] \ar[rr]^-\mu && \hbox to 1em{$\bw2\cV_5 \otimes (\cV_6/\cV_5)$\hss} \\
& \cW' \ar[ur] && \cW'', \ar@{^{(}->}[ur]
}
\end{equation*}
where $\cW'$ and $\cW''$ are vector bundles of rank $n+5$, the left diagonal arrow is an epimorphism,
the right diagonal arrow is a fiberwise monomorphism, the inner diagonal arrows are monomorphisms, and their 
cokernels are line bundles on the subscheme $S_{\gmspe}$.
\end{lemm}

\begin{proof}
Let $\cW'$ be the image of $\nu$ (so that the map $\cA \to \cW'$ is surjective).\ 
We have an exact sequence 
\begin{equation*}
0 \to \cW' \to \cW \to \Coker(\nu) \to 0.
\end{equation*}
By~\eqref{eq:cokernel-nu} and Lemma~\ref{lemma15}, 
 $\Coker(\nu)$ is a line bundle on $S_{\gmspe}$.\ 
Since $S_{\gmspe}$ is a Cartier divisor, we conclude that $\cW'$ is a vector bundle (Lemma~\ref{lemma-kernel-cartier}).\ 
Analogously, considering the dual of the map $\mu$, we construct the vector bundle $\cW''$ and the other factorization.
\end{proof}

\subsection{From families of Lagrangian data to families of GM data}
\label{subsection-from-lag-to-gm}

As in the previous section, we assume $n \le 5$.\
Let $(\ns,\cV_6,\cV_5,\cA)$ be a family of Lagrangian data of rank $n+5$   avoiding decomposable vectors.\  
Let $S_{\lagspe} \subset S$ be its special locus.\ 
Assume additionally that $S_\lagspe$ is a double Cartier divisor (Definition~\ref{defi:double}), that is,
\begin{equation}\label{equation-s1-2e}
S_{\lagspe} = 2E
\end{equation} 
(equivalently, $\cJ_\lagspe = \cI_E^2$),
where $E$ is an effective Cartier divisor.\ 
We will construct from~$(\ns,\cV_6,\cV_5,\cA)$ a family of smooth normalized GM data on $S$.\

Consider the map~\eqref{eq:varphi}.\ By definition, its rank is $n+5$, it is at least $n+4$ at every  {geometric} point, and $S_{\lagspe}$ is its degeneration scheme.\ 
Since $S_{\lagspe}$ is a Cartier divisor,  Proposition~\ref{proposition-factorization} applies and implies that the map can be written as a composition
\begin{equation*}
\cA \twoheadrightarrow \cW' \to \cW'' \hookrightarrow \bw2\cV_5 \otimes (\cV_6/\cV_5),
\end{equation*}
where $\cW'$ and $\cW''$ are vector bundles of rank $n+5$, the left arrow is an epimorphism, the right arrow is 
a fiberwise monomorphism, and the middle map is a monomorphism whose cokernel $\cW''/\cW'$ is a line bundle on $S_{\lagspe} = 2E$.\ Tensoring over $\cO_{2E}$   the  exact sequence $0 \to \cO_E(-E) \to \cO_{2E} \to \cO_E \to 0$ with $\cW''/\cW'$, we obtain an exact sequence
\begin{equation*}
0 \to (\cW''/\cW') \otimes_{\cO_{2E}} \cO_E(-E) \to \cW''/\cW' \to (\cW''/\cW') \otimes_{\cO_{2E}} \cO_E \to 0, 
\end{equation*}
where both the first and the last terms  are line bundles on $E$.\ Moreover, this is the unique representation of $\cW''/\cW'$ as an extension of two line bundles on $E$.

We define $\cW$ as the kernel of the map $\cW'' \twoheadrightarrow \cW''/\cW' \twoheadrightarrow (\cW''/\cW') \otimes_{\cO_{2E}} \cO_E$,
so that   we have a factorization
\begin{equation}\label{equation-factorization-2}
\phi\colon \cA \twoheadrightarrow \cW' \to \cW \to \cW'' \hookrightarrow \bw2\cV_5 \otimes (\cV_6/\cV_5),
\end{equation}
where $\cW'$, $\cW$, and $\cW''$  are vector bundles of rank $n+5$, the two middle maps  
are monomorphisms, and $\cW/\cW'$ and $\cW''/\cW$ are line bundles on $E$.\ This is the unique factorization of  $\phi$ 
with these properties.

We define the map 
\begin{equation*}
\mu\colon\cW \lra \bw2\cV_5 \otimes (\cV_6/\cV_5)
\end{equation*}
as the composition of the third and the fourth arrows in~\eqref{equation-factorization-2} and the map 
\begin{equation*}
\nu\colon\cA \lra \cW
\end{equation*}
as the composition of the first and the second arrows.\ With these definitions, we  have again a commutative square~\eqref{diagram-a-w-v}.

\begin{lemm}
\label{lemma:bq-gm}
The family of quadratic forms $\bq_\cA \colon \cV_6 \otimes \Sym^2\!\cA  \to \det(\cV_5) \otimes (\cV_6/\cV_5)^{ \otimes 2}$ defined by~\eqref{equation-qa}
induces a family of quadratic forms $\bq \colon \cV_6 \otimes \Sym^2\!\cW  \to \det(\cV_5) \otimes (\cV_6/\cV_5)^{ \otimes 2}$.
\end{lemm}

\begin{proof}
We will proceed in two steps.\ First, we show that $\bq_\cA$ induces a family $\bq'$ of {quadratic forms} on $\cW'$.\
Since $\cA \twoheadrightarrow \cW'$ is surjective,  it is enough to show
that its kernel bundle is contained in the kernel of $\bq_\cA$.\ This is obvious, since the kernel of that map 
is contained in the kernel of $\lambda_3$ (by definition of $\cW'$) which in   turn is contained in the kernel of $\bq_\cA$ by~\eqref{equation-qa}.

We then show that the  family of quadratic forms $\bq'$ on $\cW'$ induces a family of quadratic forms on~$\cW$ given by the Hecke transform of $\bq'$ as defined in Lemma~\ref{lemma-bq-factors}.\
We set 
\begin{equation*}
\cE := \cW'
\qquad
\cK := \cW''/\cW'(-2E),
\qquad 
\cV := \cV_6 \otimes \det(\cV_5)^\vee \otimes (\cV_6/\cV_5)^\vee.
\end{equation*}
 Restricting the sequence $0 \to \cW' \to \cW'' \to \cW''/\cW' \to 0$ to $2E$, we obtain
\begin{equation*}
0 \to \cK \to \cW'_{2E} \to \cW''_{2E} \to \cW''/\cW' \to 0,
\end{equation*}
thus the line bundle $\cK$ on $2E$ is the kernel of the map $\cW'_{2E} \to \cW''_{2E}$.\ 
In particular, it is a line subbundle in $\cE_{2E}$.\
Moreover, $\cK$ is contained in the kernel of the map $\cW' \to \bw2\cV_5 \otimes (\cV_6/\cV_5)$ restricted to $2E$, 
hence, by definition of $\bq_\cA$ in~\eqref{equation-qa}, it is contained in the kernel of $\bq'$.

We are therefore  in the setup of Lemma~\ref{lemma-bq-factors}.\
By definition of $\cW$, there is an exact sequence 
\begin{equation*}
0 \to \cW' \to \cW \to \cK_E(E) \to 0.
\end{equation*}
Comparing with~\eqref{equation-tcw}, we see that the bundle $\cW \cong \tcE$ can be
identified with the Hecke transform of $\cW'$ with respect to $\cK$,
and   Lemma~\ref{lemma-bq-factors} provides it with a family of quadratic forms.
 \end{proof}

We can now prove the main result of this section.

\begin{prop}\label{proposition-gm-from-lag-naive}
The collection $(\ns,\cW,\cV_6,\cV_5,\mu,\bq)$ constructed above is a family of smooth normalized GM data of dimension $n$.\ 
Its special locus coincides scheme-theoretically with the Cartier divisor~$E$.
\end{prop}

\begin{proof}
To show that $(\ns,\cW,\cV_6,\cV_5,\mu,\bq)$ is a family of GM data,
 we only have to verify~\eqref{eq:mu-q}.\ Since $\bq$ is induced by~$\bq_\cA$, it is enough to check
that $\bq_\cA(v)(a_1,a_2) = v \wedge \mu(\nu(a_1)) \wedge \mu(\nu(a_2))$ for $v \in \cV_5$.\ But this follows  
from the equality $\lambda_4(v \wedge \xi) = - v\wedge \lambda_3(\xi)$ for $v \in \cV_5$ and $\xi \in \bw3\cV_6$,
and the commutativity of~\eqref{diagram-a-w-v}.

The statement about the special locus is also clear, since by construction, the map~$\mu$ is a composition $\cW \to \cW'' \hookrightarrow \bw2\cV_5 \otimes (\cV_6/\cV_5)$, where the second map
is a fiberwise monomorphism and the degeneration scheme of the first map is equal to $E$.

It remains to show that the family of GM data is smooth, that is, that for each  {geometric} point $s $ of $ S$, 
the GM intersection corresponding to the GM data at the point $s$ is smooth.

If $s \notin E$, then $\cW_s = \cW'_s$ is the image of the map 
$\cA_s \xrightarrow{\ (\lambda_{3 })_s\ } \bw2 \cV_{5,s} \otimes (\cV_{6}/\cV_{5})_s$ 
and the quadratic form $\bq$ on it is induced by the form $\bq_\cA$ on $\cA_s$.\ 
Therefore, by~\cite[Theorem~3.6]{DK}, the corresponding GM intersection is just the ordinary GM variety associated
with the Lagrangian subspace $\cA_s \subset \bw3 \cV_{6,s}$, which has no decomposable vectors, 
and the hyperplane~$\cV_{5,s} \subset \cV_{6,s}$.
It is smooth by~\cite[Theorem~3.16]{DK}.

Now assume   $s \in E$.\ For brevity, we   write $V_6,V_5,A,W',W,W''$ and so on for the fibers of the corresponding vector bundles at the geometric point $s$.\
We   also choose a trivialization for $V_6/V_5$ to get rid of it in the formulas.\
Consider the restriction 
\begin{equation}\label{equation-factorization-3}
A \twoheadrightarrow W' \to W \to W'' \hookrightarrow \bw2V_5
\end{equation}
of the sequence~\eqref{equation-factorization-2} to $s$.\
Denote by $K_A$, $K'$, and $K$ the respective kernels of the first three maps in \eqref{equation-factorization-3}.\ We have $\dim (K_A) = 5 - n$ and $\dim (K') = \dim (K) = 1$.\ Since the rank of the composition of the maps in~\eqref{equation-factorization-3} is $n+4$ (because $s$ is a point of the special locus), it follows that $K'$ is equal to the kernel of the composition $W' \to W \to \bw2V_5$.\ Therefore the map~$\mu\colon W \to \bw2V_5$ is injective on  $W_0 = \nu(A)$.\ In particular, we have a canonical direct sum decomposition
\begin{equation}
\label{eq:w-k-w0}
W = K \oplus W_0.
\end{equation}
 {Note that $(W_0,V_6,V_5,\mu\vert_{W_0},\bq\vert_{W_0})$ is a GM data set corresponding to a smooth GM variety~$X_0$ of dimension $n-1$: this follows from~\cite[Theorem~3.6 and Theorem~3.16]{DK}.

The direct sum decomposition~\eqref{eq:w-k-w0}} coincides 
with the direct sum decomposition of Proposition~\ref{proposition-hecke-transform}(a) (its construction is  the same).\  
Therefore, by  Proposition~\ref{proposition-hecke-transform}(b), the decomposition is orthogonal 
with respect to the quadrics $\bq(v)$ for all $v \in V_6$.\ Furthermore, the subspace $K$ is contained in the kernel of the quadric $\bq(v)$ for all~\mbox{$v \in V_5$}, 
 {since~\eqref{eq:mu-q} holds}
 and $K$ is the kernel of $\mu$.\ {It follows that 
\begin{equation*}
M_X = \bigcap_{v \in V_5}  \{ \bq(v) = 0  \} \subset \P(W)
\end{equation*}
is the cone with vertex $\P(K) \subset \P(W)$ over the Grassmannian hull $M_{X_0} \subset \P(W_0)$ of~$X_0$  {(see~\cite[Section~2.4]{DK})}
and, if $v_0 \in V_6 \setminus V_5$, the last quadric $\bq(v_0)$ can be written as the sum $\bq(v_0) = \bq_1 \oplus \bq_0$,
where~$\bq_1 \in \Sym^2\!K^\vee$ and $\bq_0 \in \Sym^2\!W_0^\vee$ is the equation of $X_0$ in $M_{X_0}$.

If $\bq_1 \ne 0$, then $X$ is the double covering of $M_{X_0}$ branched over $X_0$, 
that is, the special GM variety associated with $X_0$ (see~\cite[Lemma~2.33]{DK}).\ In particular, $X$ is a smooth GM variety.\ If $\bq_1 = 0$, then $X$ is the cone over $M_{X_0}$.\ It remains to show that $\bq_1 \ne 0$.\ For this we recall from Proposition~\ref{proposition-hecke-transform}(b) that $\bq_1$ is the residual quadric of $\bq'(v_0)$.\ We describe it below.}

For a general vector $v_0 \in V_6 \setminus V_5$, we have $A \cap (v_0 \wedge \bw2V_6) = 0$ 
(see Remark~\ref{remark:a-v0-intersection}).\
Consider the family of quadratic forms $\bq(v_0)$ in a small neighborhood $S^0$ of~$s$ in $ S$.\
Upon shrinking $S^0$, we may assume that the vector bundle $\cV_6$ is trivial with fiber $V_6$, that
\begin{equation*}
\bw3V_6 \otimes \cO = \bw3\cV_5 \oplus (v_0 \wedge \bw2V_6) \otimes \cO
\end{equation*}
is a Lagrangian direct sum decomposition, and that $\cA \cap (v_0 \wedge \bw2V_6) \otimes \cO = 0$ at all points of $S^0$.\ 
By~\cite[Lemma~C.5]{DK}, we obtain
\begin{equation*}
\Coker(\bq_\cA(v_0)\colon\cA \to \cA^\vee) \cong \Coker(\varphi^\vee\colon \bw2\cV_5^\vee \to \cA^\vee)
\end{equation*}
(where we assume that the line bundles $\det(\cV_5)$ and $\cV_6/\cV_5$ are trivial on $S^0$).\
Since the kernel of the epimorphism $\cA \twoheadrightarrow \cW'$ is contained in the kernels of both $\bq_\cA(v_0)$ and $ {\varphi}$, 
we can cancel it out and obtain
\begin{equation*}
\Coker(\bq'(v_0)\colon\cW' \to {\cW'}^\vee) \cong \Coker(\bw2\cV_5^\vee \to {\cW'}^\vee).
\end{equation*}
The rightmost map  factors as an epimorphism $\bw2\cV_5^\vee \to {\cW''}^\vee$ followed by the dual of the map $\cW' \to \cW''$.\ 
Dualizing the sequence $0 \to \cW' \to \cW'' \to \cW''/\cW' \to 0$ and taking into account 
that  $\cW''/\cW'  $  is a line bundle on the Cartier divisor $2E$ (by construction of the bundles $\cW'$ and $\cW''$ 
at the beginning of Section~\ref{subsection-from-lag-to-gm}),
 we deduce that~$\Coker(\bw2\cV_5^\vee \to {\cW'}^\vee)$ is a line bundle on $2E$, 
hence so is $\Coker(\bq'(v_0)\colon \cW' \to {\cW'}^\vee)$.\
By Lemma~\ref{lemma-support-cokernel}, the discriminant $\Dis(\bq'(v_0))$ is equal to $2E$
 and by Lemma~\ref{lemma-residual-quadric}, the residual quadric $\bq_1$ is nonzero.
 
As explained above, this means that the GM intersection corresponding to the point $s$ is a smooth GM variety of dimension $n$.
 \end{proof}

\subsection{Compositions of the constructions}
\label{subsection:gm-and-lagrangian}

We show that the constructions introduced in Sections~\ref{subsection-from-gm-to-lag} and~\ref{subsection-from-lag-to-gm} 
are mutually inverse.

Let $S$ be a scheme and let $E \subset S$ be an effective Cartier divisor.\ Denote by
\begin{equation*}
\begin{aligned}
 \ucM_n(S,E) 		& = && \{ (\ns,\cW,\cV_6,\cV_5,\mu,\bq) 	&&\in \ucM_n(S) 	&&\mid S_\gmspe = E \},\\
\cmlag_{n+5}(S,E) 	& = && \{ (\ns,\cV_6,\cV_5,\cA) 		&&\in \cmlag_{n+5}(S) 	&&\mid S_\lagspe = E \}
 \end{aligned}
\end{equation*}
the subgroupoids of $\ucM_n(S)$ and $\cmlag_{n+5}(S)$ (defined in Sections~\ref{subsection:stack-gm} and~\ref{subsection:lagrangian-data})  formed by all families of smooth normalized GM data of dimension $n$ 
(resp.\ by all families of Lagrangian data of rank $n+5$ avoiding decomposable vectors) over~$S$ whose special loci is $E$.

\begin{prop}
\label{prop:equivalence-groupoids}
Assume $n\in\{3,4,5\}$.\ For every effective Cartier divisor $E \subset S$, the morphism of stacks $\fa$ defined in Proposition~\textup{\ref{proposition-gm-to-lag}} 
induces an equivalence of groupoids 
\begin{equation*}
\ucM_n(S,E) \isomlra \cmlag_{n+5}(S,2E).
\end{equation*}
\end{prop}

\begin{proof}
 By Proposition~\ref{proposition-gm-to-lag}, the morphism of stacks $\fa$ doubles the special locus, 
hence induces a functor between the groupoids $\ucM_n(S,E) \to \cmlag_{n+5}(S,2E)$.\
Let us show that the construction of Section~\ref{subsection-from-lag-to-gm} defines its quasi-inverse functor.\
This construction is clearly functorial, so it remains to consider its compositions with $\fa$.

Let us start  with a family of smooth normalized GM data $(\ns,\cW,\cV_6,\cV_5 ,\mu,\bq)$ with special locus~$E$ 
and let $(\ns,\cV_6,\cV_5,\cA)$ be the family of Lagrangian data obtained by applying the morphism~$\fa$.\
 Its special locus is $2E$ by Proposition~\ref{proposition-gm-to-lag}.\
Applying the construction of Section~\ref{subsection-from-lag-to-gm} to $(\ns,\cV_6,\cV_5,\cA)$, we must show 
that the family of GM data that we get is isomorphic to the original one.

The first step of the construction of Section~\ref{subsection-from-lag-to-gm} is the factorization~\eqref{equation-factorization-2} 
of the morphism~$\varphi$ defined by~\eqref{eq:varphi}.\ Comparing it with the factorization of Lemma~\ref{lemma-factorization2} 
and using the uniqueness of such a factorization (Proposition~\ref{proposition-factorization}), 
we deduce that the bundles $\cW'$, $\cW$, $\cW''$, and the maps~$\nu$ and $\mu$ defined 
by this factorization agree with those in the lemma.\ It remains to show that the quadratic forms $\bq$ agree.\ This follows from the compatibility of Lemma~\ref{lemma-quadratic-compatible}
and the uniqueness of the induced quadratic form.

Conversely, let us start with a family of Lagrangian data $(\ns,\cV_6,\cV_5,\cA)$.\ We produce a family 
$(\ns,\cW,\cV_6,\cV_5 ,\mu,\bq)$ of smooth normalized GM data by the constructions of Section~\ref{subsection-from-lag-to-gm} 
and  apply the functor $\fa$.\
In other words, we consider the diagram~\eqref{equation-ax-monad} and our goal is to show that the cohomology bundle
$\Ker(f_2)/\Im(f_1)$ of its upper row is isomorphic to the Lagrangian subbundle~$\cA \subset \bw3\cV_6$ we started with.

For this, we consider the map 
\begin{eqnarray*}
f_4\colon\cV_6 \otimes \cA \otimes (\cV_6/\cV_5)^\vee& \lra & \bw3\cV_5 \oplus (\cV_6 \otimes \cW \otimes (\cV_6/\cV_5)^\vee)\\
 v \otimes a&\longmapsto&  (\lambda_4(v \wedge a), v \otimes \nu(a)).
\end{eqnarray*}
We have
\begin{equation*}
f_2\circ f_4(v \otimes a)(w') = \lambda_4(v \wedge a) \wedge \mu(w') + \bq(v)(\nu(a),w').
\end{equation*}
For $w' = \nu(a')$, we have $\mu(w') = \mu(\nu(a')) = \lambda_3(a')$, hence the right   side equals
\begin{equation*}
- \bq_\cA(v)(a,a') + \bq(v)(\nu(a),\nu(a')) = 0, 
\end{equation*}
since $\bq$ is induced by $\bq_\cA$ via the map $\nu$.
This means that the composition 
\begin{equation*}
\cV_6 \otimes \cA \otimes \cA \otimes (\cV_6/\cV_5)^\vee \xrightarrow{\ \nu\ } 
\cV_6 \otimes \cA \otimes \cW \otimes (\cV_6/\cV_5)^\vee \xrightarrow{\ f_2 \circ f_4\ }
\det(\cV_6)
\end{equation*}
vanishes.\
Since $\nu$ is surjective on the complement of a Cartier divisor, 
it follows from Lemma~\ref{lemma-factorization} that $f_2 \circ f_4 = 0$.\ Therefore, the map $f_4$ factors through the kernel of $f_2$.

Furthermore, the restriction of $f_4$ to $\cV_5 \otimes \cA \otimes (\cV_6/\cV_5)^\vee$ can be rewritten as
\begin{equation*}
f_4(v \otimes a) = (\lambda_4(v \wedge a), v \otimes \nu(a)) = ( - v \wedge \lambda_3(a), v \otimes \nu(a)) = f_1(v \wedge \nu(a)),
\end{equation*}
hence the composition $\cV_6 \otimes \cA \otimes (\cV_6/\cV_5)^\vee \xrightarrow{\ f_4\ } \Ker(f_2) \to \Ker(f_2)/\Im(f_1)$
factors as 
\begin{equation*}
\cV_6 \otimes \cA \otimes (\cV_6/\cV_5)^\vee \xrightarrow{\ \lambda\ } \cA \lra \Ker(f_2)/\Im(f_1).
\end{equation*}
Finally,   by the commutativity of~\eqref{diagram-a-w-v}, we have
\begin{equation*}
f_3\circ f_4(v \otimes a) = f_3(\lambda_4(v \wedge a), v \otimes \nu(a)) = \lambda_4(v \wedge a) + v \wedge \mu(\nu(a)) = \lambda(v) a.
\end{equation*}
Together with the above observation, this means that the composition
\begin{equation*}
\cA \lra \Ker(f_2)/\Im(f_1) \xrightarrow{\ f_3\ } \bw3\cV_6
\end{equation*}
is the embedding of $\cA$.\ Since both $\cA$ and $\Ker(f_2)/\Im(f_1)$ are Lagrangian subbundles in~$\bw3\cV_6$, they are isomorphic.
\end{proof}

\begin{rema}
\label{remark:ordinary-surfaces-equivalence}
The same argument proves that there is an equivalence $ {\ucMtwo}(S) \isomlra \cmlag_{7}(S)$ of groupoids,
where ${\ucMtwo} \subset \ucM_{2,{\rm ord}}$ is the substack defined by~\eqref{eq:ucm-ss}.
\end{rema}

We finish this section by stating a combination of the above results  {(including Theorem~\ref{theorem:gm-substack-gm-data})}
which is a simplified version of Proposition~\ref{prop:equivalence-groupoids}.

\begin{coro}
\label{corollary:constructing-families-gm}
Let $(\ns,\cV_6,\cV_5,\cA)$ be a family of Lagrangian data   of rank~$n + 5$, avoiding decomposable vectors, 
and such that its special locus  is a double Cartier divisor~$S_\lagspe = 2E$.\
There is a unique family of smooth GM varieties $(S,\cX \to S,\cH)$  
such that $(\ns,\cV_6,\cV_5,\cA)$ is obtained from the corresponding family of GM data by the morphism~$\fa$.
\end{coro}

 {This corollary will be used later for constructing interesting families of GM varieties.}

\section{Descriptions as global quotients} 
\label{section:global}

We   describe the moduli stacks $\ucM_n$ and $\cmlag_n$ (defined in Sections~\ref{subsection:stack-gm} and~\ref{subsection:lagrangian-data})  as global quotients stacks
and derive a description of their coarse moduli spaces as the corresponding GIT quotients.

\subsection{EPW sextics}
\label{subsection:epw}

Let $V_6$ be a 6-dimensional  {$\k$}-vector space and let  $A \subset \bw{3}{V_6}$ be a Lagrangian subspace 
for the  $\det(V_6)$-valued  symplectic form defined by wedge product.

\begin{defi}\label{definition-epw}
For any integer $\ell$, we set
\begin{equation}\label{yabot}
Y_A^{\ge \ell}:=\bigl\{[v]\in\P(V_6) \mid \dim\bigl(A\cap (v \wedge\bw{2}{V_6} )\bigr)\ge \ell\bigr\}
\end{equation}
and endow it with a   scheme structure as in~\cite[Section 2]{og1}.\ 
The locally closed subsets
\begin{equation}\label{yaell}
Y_A^\ell := Y_A^{\ge \ell} \setminus Y_A^{\ge \ell + 1} 
\end{equation}
 of $\P(V_6)$ form the {\sf {Eisenbud--Popescu--Walter (EPW)} stratification} and the sequence of inclusions 
\begin{equation*}
\P(V_6) = Y_A^{\ge 0} \supset Y_A^{\ge 1} \supset Y_A^{\ge 2} \supset \cdots
\end{equation*}
is called   the {\sf EPW sequence}.\ When the scheme $Y_A := Y_A^{\ge 1}$ is not the whole space $\P(V_6)$, it is a sextic hypersurface (\cite[(1.8)]{og1}) called an {\sf EPW sextic}.\ 
The scheme $Y_A^{\ge 2}$ is nonempty and has everywhere dimension $\ge 2$ (\cite[(2.9)]{og1}).
\end{defi}


%

The following theorem gathers various results of O'Grady's (see~\cite[Theorem~B.2]{DK}; all these results were proved for $\k = \C$ but, by the Lefschetz principle, 
they  extend to any field $\k$ of characteristic zero).

\begin{theo}[O'Grady]
\label{theorem-ogrady-stratification}
Let $A \subset \bw{3}{V_6}$ be a Lagrangian subspace.\ If $A$ contains no decomposable vectors, that is $\P(A) \cap \Gr(3,V_6) = \varnothing$, then
\begin{itemize}
\item[\rm (a)] $Y_A $ is an integral normal sextic hypersurface in $\P(V_6)$;
\item[\rm (b)] $Y_A^{\ge 2} = \Sing(Y_A )$ is an integral normal  Cohen--Macaulay surface of degree $40$;
\item[\rm (c)] $Y_A^{\ge 3} = \Sing(Y_A^{\ge 2})$ is  finite and smooth,  and is empty for  $A$ general;
\item[\rm (d)] $Y_A^{\ge 4} $  is empty.
\end{itemize}
\end{theo}

\begin{rema}
\label{remark:a-v0-intersection}
It follows that if $A$ contains no decomposable vectors,  we have $A \cap (v_0 \wedge \bw2V_6) = 0$ for general $v_0 \in V_6$.\
We used this   observation in the proof of Proposition~\ref{proposition-gm-from-lag-naive}.
\end{rema}

If  $A \subset \bw{3}{V_6}$ is a Lagrangian subspace, its orthogonal $A^ \perp \subset \bw{3}{V_6^\vee}$ is also a Lagrangian subspace.\ 
In the dual projective space $\P(V_6^\vee) = \Gr(5,V_6)$, the EPW sequence for $A^\bot$ can be described in terms of $A$ as
\begin{equation}
\label{dualYA}
Y_{A^\perp}^{\ge \ell} = \big\{ V_5 \in \Gr(5,V_6) \mid \dim (A \cap \bw{3}{V_5}) \ge \ell \big\}.
\end{equation}
The canonical identification $\Gr(3,V_6) \isom \Gr(3,V_6^\vee) $ induces an isomorphism 
between the intersections $\P(A) \cap \Gr(3,V_6)$ and $\P(A^\perp) \cap \Gr(3,V_6^\vee)$ (\cite[(2.82)]{og3}).\
In particular, $A$ contains no decomposable vectors if and only if the same holds for $A^\perp$.

We will not need this fact, but if $A$ contains no decomposable vectors, the hypersurfaces~$Y_A$ and~$Y_{A^\perp}$ 
are projective dual (\cite[Corollary 3.6]{og2} or~\cite[Proposition~B.3]{DK}).

If $\k = \C$ and $A$ contains no decomposable vectors, 
O'Grady defined in~\cite[Section~1.2]{og4} a canonical double cover $\widetilde{Y}_A \to Y_A$
 {(called the {\sf double EPW sextic})}.\
This construction was generalized in~\cite[Theorem~5.2]{DK:coverings} to other EPW strata; 
it works over an arbitrary field $\k$ of characteristic different from $2$ and provides canonical double coverings
\begin{equation*}
\widetilde{Y}^{\ge 0}_A \to Y_A^{\ge 0},
\qquad
\widetilde{Y}^{\ge 1}_A \to Y_A^{\ge 1},
\qquad 
\widetilde{Y}^{\ge 2}_A \to Y_A^{\ge 2}
\end{equation*}
branched over $Y_A^{\ge 1}$, $Y_A^{\ge 2}$, and $Y_A^{\ge 3}$, respectively 
(the first of these is the usual double covering of $\P(V_6)$ branched over the EPW sextic hypersurface).\
We denote the quotient stacks of these coverings by their natural involutions by 
\begin{equation}
\label{eq:hya}
\widehat{Y}_A^{\ge \ell} := \widetilde{Y}_A^{\ge \ell}/\bmu_2.
\end{equation} 
They come with  natural maps
\begin{equation*}
\rho_A^{\ell} \colon \widehat{Y}_A^{\ge \ell} \lra {Y}_A^{\ge \ell}.
\end{equation*}
For $\ell = 0$, we obtain the root stack of $\P(V_6)$ with respect to the EPW sextic hypersurface.

Consider the natural action of the group $\PGL(V_6)$ on the Lagrangian Grassmannian~$\LGr(\bw3V_6)$
and its natural linearization in the line bundle $\cO(2)$ (note that the line bundle $\cO(1)$ does not admit a linearization).\
O'Grady showed (\cite{og5}) that the GIT quotient
\begin{equation*}
\bEPW  :=\LGr(\bw3V_6)\gquot\PGL(V_6)
\end{equation*}
is a coarse GIT moduli space for  {double} EPW sextics.\
The following lemma will be crucial for us.

\begin{lemm}[O'Grady]\label{ssig}
The hypersurface 
\begin{equation*}
\Sigma := \{ A \in \LGr(\bw3V_6)\mid \text{$A$ has decomposable vectors} \}
\end{equation*}
is $\PGL(V_6)$-invariant and its complement  
\begin{equation*}
\LGradv(\bw3V_6) = \LGr(\bw3V_6) \setminus \Sigma
\end{equation*}
is affine and consists of stable points.\ {In particular, the stabilizer $\PGL(V_6)_A \subset \PGL(V_6)$ of any $A \in \LGradv(\bw3V_6)$ is finite.}
\end{lemm}

\begin{proof}
The stability statement is proved in~\cite[Corollary~2.5.1]{og5} 
(over $\k = \C$, but stability is  defined over the algebraic closure (\cite[Definition~1.7]{git})
and by the Lefschetz principle, stability over $\C$ implies stability over any algebraically closed field of characteristic zero) 
and the rest is easy.\ 
\end{proof}

The  hypersurface $\Sigma \subset \LGr(\bw3V_6)$ has degree~42, the degree of $\Gr(3,V_6)$.\
We denote by~$\SEPW$  the image of $\Sigma$ in $\bEPW$.\  
Its complement     
\begin{equation}
\label{eq:m-epw}
\EPW  := \bEPW \setminus \SEPW = \LGradv(\bw3V_6) \gquot \PGL(V_6)
\end{equation}
is  affine; it is a coarse moduli space for EPW sextics $Y_A$ such that $A$ has  no decomposable vectors.

\subsection{The moduli stack of Lagrangian data}
\label{subsection:global-quotient-lagrangian}

We deal here with the easier case of the moduli stack of Lagrangian data.

For each integer~$n\in\{2,\dots,5\}$, {we consider the following relative versions of some EPW strata
(we changed the number $\ell$ in~\eqref{dualYA} to $n = 5 - \ell$):}
\begin{align*} 
 \rS_n &:= \{ (A,V_5) \in \LGradv(\bw3V_6) \times \P(V_6^\vee) \mid \dim(A \cap \bw3V_5) = 5 - n \}\\
 \intertext{(one could also define $\rS_n$ for $n \le 1$, but it is  empty) and}
 \brS_n &:= \{ (A,V_5) \in \LGradv(\bw3V_6) \times \P(V_6^\vee) \mid \dim(A \cap \bw3V_5)  \in \{ 5 - n, 6 - n \} \}.
\end{align*}
The subscheme $\rS_{n-1}$ of $\brS_n$ is closed and $\rS_n$ is its open complement.\ 
The  scheme $\brS_n$ is locally closed   in $\LGradv(\bw3V_6) \times \P(V_6^\vee)$ and in $\LGr(\bw3V_6) \times \P(V_6^\vee)$.\
In particular, it is a quasiprojective scheme.\
We will need the following result.

\begin{lemm}
\label{lemma:sn-bsn}
For $n \ge 2$, the scheme $\rS_n$ is smooth of dimension $60 - \frac12(5 - n)(6 - n)$.\ 

 For~$n \in \{2,5\}$, the scheme $\brS_n$ is smooth; for~$n\in\{3,4\}$ it is normal and \mbox{$\Sing(\brS_n) = \rS_{n-1}$}.
\end{lemm}

\begin{proof}
The fiber of the projection $\brS_n \to \LGradv(\bw3V_6)$ over a Lagrangian subspace $A$ with no decomposable vectors 
is the union $Y_{A^\perp}^{5-n} \cup Y_{A^\perp}^{6-n}$ of strata of the dual EPW stratification 
associated with $A$, hence Theorem~\ref{theorem-ogrady-stratification} applies.
 \end{proof}

 Consider now the action of $\PGL(V_6)$ on the product $\LGr(\bw3V_6) \times \P(V_6^\vee)$.\ 
As we noted above, the line bundle $\cO(2,0)$ has a natural linearization.\ It is   clear that $\cO(0,6)$ also admits a linearization.\ 
Consequently, for any $m \in \Z$, the line bundle $\cO(2m,6)$ admits a $\PGL(V_6)$-linearization.\ 

\begin{coro}
\label{corollary:stability}
Take  $n \in \{2,3,4,5\}$.\
For $m \gg 0$, the subschemes $\brS_n \subset \LGr(\bw3V_6) \times \P(V_6^\vee)$ and $\rS_n \subset \brS_n$
consist of $\PGL(V_6)$-stable points for the $\cO(2m,6)$-linearization.\ 
\end{coro}

\begin{proof}
This follows from Lemma~\ref{ssig} and~\cite[Proposition~2.18]{git} applied to morphisms $\brS_n \to \LGradv(\bw3V_6)$ and~$\rS_n \to \LGradv(\bw3V_6)$.
\end{proof}

The action of $\PGL(V_6)$ on $\brS_n$ also induces an action of $\GL(V_6)$.
The canonical morphism 
\begin{equation*}
\brS_n / \GL(V_6) \lra \brS_n /\PGL(V_6)
\end{equation*}
of global quotient stacks is a $\Gm$-gerbe, because the center $\Gm \subset \GL(V_6)$ acts trivially on~$\brS_n$.\
In fact, this morphism is the rigidification for the natural embedding of~$\Gm$ 
into the automorphism groups of objects of the stack $\brS_n/\GL(V_6)$.\ 
Recall also that the stack $\cmlag_{n+5}$ is the rigidification of the stack $\tcmlag_{n+5}$ (Remark~\ref{remark:linear-lagrangian-lift}).

\begin{prop}
\label{proposition:cmlag-quotient}
For each $n \in \{2,3,4,5\}$, the moduli stack $\cmlag_{n+5}$ of families of Lagrangian data of rank $n+5$ avoiding decomposable vectors 
is the global quotient stack 
\begin{equation*}
\cmlag_{n+5} \simeq \brS_n/\PGL(V_6).
\end{equation*}
In particular, it is a separated Deligne--Mumford stack of finite presentation over $\Q$.\
Its special locus  is also a global quotient stack
\begin{equation*}
\cmlag_{n + 5,{\rm spe}} \simeq \rS_{n-1}/\PGL(V_6).
\end{equation*}
The stack $\cmlag_{n+5}$
is smooth for $n \in \{2,5\}$; for $n \in \{3,4\}$, it is singular along $\cmlag_{n + 5,{\rm spe}}$.
\end{prop}

\begin{proof}
We first prove that the stack of families of linearized Lagrangian data $\tcmlag_n$ is isomorphic to the quotient stack $\brS_n/\GL(V_6)$  by constructing morphisms in both directions between these stacks.

The scheme $\brS_n$ comes with  {the tautological} family of Lagrangian data on the trivial bundle~$\cV_6 = V_6 \otimes \cO_{\brS_n}$ 
(the Lagrangian subbundle is pulled back from $\LGr(\bw3V_6)$ and the subbundle $\cV_5$ is pulled back from $\P(V_6^\vee)   \simeq \Gr(5,V_6) $).\ 
The definition of $\brS_n$ ensures that this family of Lagrangian data has rank $n+5$ and avoids decomposable vectors.\ 
Hence, it induces  a morphism $\brS_n \to \tcmlag_{n+5}$.\ 
The morphism is $\GL(V_6)$-equivariant, hence factors through a map from the quotient stack $\brS_n/\GL(V_6)  $ to $ \tcmlag_{n+5}$.

%
%

Let us construct the inverse.\ 
Let $S$ be a scheme and let $(\ns,\cV_6,\cV_5,\cA)$ be a family of linearized Lagrangian data of rank $n+5$, avoiding decomposable vectors.\  
Consider the~$\GL(V_6)$-torsor  {$f \colon \widetilde{S} \to S$} associated with the vector bundle~$\cV_6$, so that
the pullback of the bundle~$\cV_6$ to $\widetilde{S}$ comes with a canonical trivialization $ {f^*}\cV_6 \cong V_6 \otimes \cO_{\widetilde{S}}$.\   
The pullbacks of the bundles~$\cA$ and $\cV_5$ can be considered respectively 
as a Lagrangian subbundle \mbox{$ {f^*}\cA \hookrightarrow \bw3V_6 \otimes \cO_{\widetilde{S}}$} and 
as a corank-1 subbundle $ {f^*}\cV_5 \hookrightarrow V_6 \otimes \cO_{\widetilde{S}}$.\  
Moreover, these subbundles are $\GL(V_6)$-equivariant.\ 
Together they provide a $\GL(V_6)$-equivariant map 
\begin{equation*}
\widetilde{S} \lra \LGr(\bw3V_6)  \times \P(V_6^\vee).
\end{equation*}
As the family $(\ns,\cV_6,\cV_5,\cA)$ has rank $n+5$ and avoids decomposable vectors, the map factors through the subscheme $\brS_n$.\ This map is $\GL(V_6)$-equivariant, hence gives a map 
\begin{equation*}
S = \widetilde{S} /\GL(V_6) \to \brS_n / \GL(V_6).
\end{equation*}
This construction defines a morphism of stacks $\tcmlag_{n+5} \to \brS_n/\GL(V_6)$.

It is easy to see that the  morphisms we constructed above are mutually inverse, hence define an isomorphism of stacks.\ 
Moreover, this isomorphism is compatible with the embeddings of $\Gm(S)$ into automorphisms groups of objects of the stacks $\tcmlag_{n+5}$ and $\brS_n/\GL(V_6)$.\ 
Therefore, the rigidifications of these stacks, $\cmlag_{n+5}$ and $\brS_n/\PGL(V_6)$, are also isomorphic.

Since $\brS_n$ is quasiprojective and, by Corollary~\ref{corollary:stability}, consists of $\PGL(V_6)$-stable points, 
the stack $\brS_n/\PGL(V_6)$ is a separated Deligne--Mumford stack of finite presentation over~$\Q$.\ 
It is clear that under the isomorphism $\cmlag_{n+5} \simeq \brS_n/\PGL(V_6)$, the special locus of~$\cmlag_{n+5}$ corresponds to the substack $\rS_n/\PGL(V_6)$.\
The description of the singular locus of~$\brS_n/\PGL(V_6)$ follows immediately from Lemma~\ref{lemma:sn-bsn}.
\end{proof}

\subsection{The moduli stack of GM varieties}
\label{subsection:global-quotient-gm}

We now describe  the moduli stack  of smooth GM varieties 
(which we identify with the moduli stack  of smooth normalized GM data).

 {As before}, consider the product $\LGr(\bw3V_6) \times \P(V_6^\vee)$.\ Set $\cV_6 := V_6 \otimes \cO$ and let $\cV_5 \subset \cV_6$ and~$\cA \subset \bw3\cV_6$
be the pullbacks of the tautological subbundles from $\P(V_6^\vee)$ and $\LGr(\bw3V_6)$ respectively.\ 
Note that $\bw3\cV_6 / \bw3\cV_5 $ is isomorphic to $ \bw2\cV_5 \otimes (\cV_6/\cV_5)$ via the map $\lambda_3$, 
so that  the subscheme $\rS_n \subset \LGradv(\bw3V_6) \times \P(V_6^\vee)$ is just the rank-$(n + 5)$ degeneracy locus 
of the morphism~$\varphi$ defined by~\eqref{eq:varphi}.\ 
In particular, $\brS_n$ is a Lagrangian intersection locus 
(as defined in~\cite[Section~4]{DK:coverings}) for the Lagrangian subbundles
\begin{equation*}
\cA \hookrightarrow \bw3\cV_6
\qquad\text{and}\qquad 
\bw3\cV_5 \hookrightarrow \bw3\cV_6.
\end{equation*}
Following~\cite[Section~4.1]{DK:coverings}, we denote by $\cC_n$ the cokernel sheaf of $\varphi$ on $\brS_n$.\ 
By definition of~$\brS_n$, the rank of $\cC_n$ is $5 - n$.\  Consider the reflexive hull of its top exterior power
\begin{equation}
\label{eq:cr-n}
\cR_n := (\bw{5-n}\cC_n)^{\vee\vee}.
\end{equation}
For $n \le 4$, this is a rank-1 reflexive sheaf on $\brS_n$, and for $n = 5$, we have $\cR_5 \cong \cO_{\brS_5}$.\ 
Furthermore, it was shown in the proof of~\cite[Theorem~4.2]{DK:coverings} that if 
$\cL_n$ is  the line bundle
\begin{equation}
\label{eq:cl-n}
\cL_n := \left.\bigl( \det(V_6)^{15-n} \otimes \det(\cA^\vee) \otimes \det(\cV_5^\vee)^{\otimes 6} \bigr)\right\vert_{\brS_n},
\end{equation}
there is a natural morphism 
\begin{equation*}
\bm_n \colon \cR_n \otimes \cR_n \lra \cL_n 
\end{equation*}
which, for $n \le 4$,  identifies $\cL_n$ with the reflexive hull $(\cR_n \otimes \cR_n)^{\vee\vee}$, 
and for~$n = 5$,  is just a global section of $\cL_5$ (in fact, $\bm_5 = \det(\varphi)$).

\begin{lemm}
\label{lemma:ideal-mn}
The subscheme\/ $ {B(\bm_n) \subset {}} \brS_n$, defined by the ideal image of the map $\bm_n$ twisted by~$\cL_n^{-1}$, is equal to $\rS_{n-1} \subset \brS_n$.
\end{lemm}

\begin{proof}
By ~\cite[Theorem~4.2]{DK:coverings}, the subscheme $B(\bm_n)$ coincides (locally over $\brS_n$) with the branch locus of the double covering of $\brS_n$ 
associated with the reflexive sheaf $\cR_n$ and the morphism $\bm_n$.\ Since, by Lemma~\ref{lemma:sn-bsn}, the schemes $\rS_n$ are smooth of expected dimensions, \cite[Corollary~4.7]{DK:coverings}
identifies the branch locus with the scheme $\rS_{n-1}$.
\end{proof}

We are therefore in the setup of Appendix~\ref{section:generalized-root-stack}.\ Accordingly, we consider the root stack 
\begin{equation*}
\widehat{\rS}_n \to \brS_n
\end{equation*}
of $(\cR_n,\bm_n)$.\ 
For~$n = 5$, the stack $\hrS_5 $ is isomorphic to  $\sqrt{(\cL_5,\det( {\varphi}))/\,\brS_5}$, the root stack 
with respect to the hypersurface $S_4 \subset \brS_5$ in the sense of~\cite[Section~B.2]{agv}.\
For $n \le 4$, the stacky locus~$\rS_{n-1}$ of $\hrS_n$ has codimension $6 - n \ge 2$; 
in this sense $\hrS_n$ is a \emph{generalized} root stack.

We have the following property.

\begin{lemm}
\label{lemma:hrsn-smooth}
The stack\/ $\hrS_n$ is a smooth   separated  Deligne--Mumford stack.\
The action of the group $\PGL(V_6)$ on $\brS_n$ lifts to an action on $\widehat{\rS}_n$ such that the morphism
$\hrS_n \to \brS_n$ is~$\PGL(V_6)$-equivariant.
\end{lemm}

\begin{proof}
To show that $\hrS_n$ is a smooth  {and separated} Deligne--Mumford stack, it is enough, in view of Proposition~\ref{proposition:ts-hs},
to check that the {\'etale} double cover  $\widetilde{S}$ of $\brS_n$ associated with the morphism $\bm_n$
and a square root $\cM$ of $\cL_n$ (which exists locally over $\brS_n$) is smooth.\
This  follows from~\cite[Corollary~3.7]{DK:coverings}, since $\LGradv(\bw3V_6) \times \P(V_6^\vee)$ is smooth and
$\rS_n$ is smooth of expected codimension  {(Lemma~\ref{lemma:sn-bsn})}.\

To show that the $\PGL(V_6)$ action on $\brS_n$ lifts to $\hrS_n$, 
recall from~\eqref{eq:generalized-root-stack} that the stack $\widehat{\rS}_n$  {can be} defined as the quotient stack
\begin{equation*}
 \widehat{\rS}_n =\wwrS_n/\Gm,\qquad\textnormal{where\ \ }
\wwrS_n := \Spec_{\, \brS_n } \bigl(\cO_{\brS_n}[\cL_n^{\pm1},\cR_n]\bigr),
\end{equation*}
  the sheaf of algebras $\cO_{\brS_n}[\cL_n^{\pm1},\cR_n]$ is defined in~\eqref{eq:os-l-r},
and the $\Gm$-action corresponds to its grading.\ The sheaves $\cR_n$ and $\cL_n$ and the morphism $\bm_n$ are $\GL(V_6)$-equivariant and the center $\Gm \subset \GL(V_6)$ 
acts on them with respective weights $3(5 - n)$ and $6(5 - n)$ {by~\eqref{eq:cr-n} and~\eqref{eq:cl-n}}.\ Therefore, the group 
\begin{equation}
\label{eq:rg-n}
\rG_n := \GL(V_6)/\bmu_{3(5-n)}
\end{equation}
acts on the sheaf of algebras $\cO_{\brS_n}[\cL_n^{\pm1},\cR_n]$ defined in~\eqref{eq:os-l-r},
hence also on its relative spectrum~$\hhrS_n$,
in such a way that the action of its center $\Gm/\bmu_{3(5-n)} \cong \Gm$ corresponds to the grading of the algebra.\
Therefore, the stack $\hrS_n$ carries an action of the quotient group 
\begin{equation*}
\bigl( \GL(V_6)/\bmu_{3(5-n)} \bigr) \big/ \bigl( \Gm/\bmu_{3(5-n)} \bigr) \cong \GL(V_6) / \Gm = \PGL(V_6)
\end{equation*}
and the map $\widehat{\rS}_n \to \brS_n$ is $\PGL(V_6)$-equivariant.
\end{proof}

The argument of the proof of the lemma also has the following useful consequence.

\begin{coro}
There is an isomorphism of stacks $\hrS_n / \PGL(V_6) \cong \hhrS_n /\rG_n$.
\end{coro}

We are now ready to prove the main result of this section.

\begin{theo}
\label{theorem:gm-global-quotient}
For $n \in \{3,4,5\}$, the stack of smooth polarized GM varieties $\ucM_n$ is isomorphic to the global quotient stack 
\begin{equation*}
\ucM_n \cong \widehat{\rS}_n / \PGL(V_6) \cong \hhrS_n /\rG_n.
\end{equation*}
 {In particular, it is a smooth separated Deligne--Mumford stack of finite presentation over~$\Q$.}
\end{theo}

\begin{proof}
The first step is the construction of a morphism of stacks 
\begin{equation*}
\hhrS_n /\rG_n \cong \widehat{\rS}_n/\PGL(V_6) \lra \ucM_n.
\end{equation*}
This is equivalent to the construction of a $\rG_n$-equivariant family of smooth GM varieties over~$\hhrS_n$, 
which is accomplished by a combination of  several constructions described earlier.

The natural morphism $\hhrS_n \to \brS_n$ can be factored as the composition 
\begin{equation*}
\hhrS_n \xrightarrow{\ \hhf\ } \bbrS_n :=
\Spec_{\,\brS_n}( \cO_{\brS_n}[\cL_n^{\pm1}] ) = 
\Spec_{\,\brS_n} \Bigl(\bigoplus_{i \in \Z} \cL_n^i \Bigr) \xrightarrow{\ \theta\ } {\brS_n},
\end{equation*}
where $\theta \colon  \bbrS_n \to \brS_n$ is the $\Gm$-torsor associated with the line bundle $\cL_n$ 
and $\hhf$ is the double cover associated by~\cite[Proposition~2.5]{DK:coverings} with the reflexive sheaf $ {\theta}^*\cR_n$ 
and the natural morphism
\begin{equation*}
\theta^*\cR_n \otimes \theta^*\cR_n \xrightarrow{\ \theta^*(\bm_n)\ } \theta^*\cL_n \cong \cO.
\end{equation*}
The sheaf $\theta^*\cR_n$ is the reflexive sheaf 
associated with the Lagrangian intersection of the subbundles $\theta^*\cA$ and $\theta^*(\bw3\cV_5)$ on the scheme $\bbrS_n$.\
Therefore, by~\cite[Corollary~3.6]{DK:coverings}, the double cover $\hhrS_n$ is smooth.

Recall from~\cite[Definition~2.8]{DK:coverings} the notions of branch and ramification loci, $B(f)$ and~$R(f)$, 
for the double cover $\hhf$.\
By~Lemma~\ref{lemma:ideal-mn} and~\cite[Corollary~4.7]{DK:coverings}, we have
\begin{equation*}
B(f) = \theta^{-1}(\rS_{n-1}),
\qquad
R(f) \cong \theta^{-1}(\rS_{n-1}) \subset \hhrS_n,
\end{equation*}
and the preimage $\hhf^{-1}(B(f))$   is the first order infinitesimal neighborhood 
of  $R(f)$.

Denote by 
\begin{equation*}
\bS := \Bl_{R(f)} \bigl(\, \hhrS_n \bigr) \xrightarrow{\ \beta\ } \hhrS_n
\end{equation*}
the blow up of the scheme $\hhrS_n$ along   $R(f)$ and let $E$ be its exceptional divisor.\ 
The preimage of the subscheme $\rS_{n-1}$ under the map $\beta \circ \hhf \circ \theta$ is the Cartier divisor~$2E$.\ The following diagram  collects the   stacks and morphisms that we constructed:
\begin{equation*}
\xymatrix@C=2em{
\bS \ar[rr]^-\beta &&
\hhrS_n \ar[rr]_-{\bmu_2}^-{f} \ar[d]^-{\Gm} &&
\bbrS_n \ar[d]_-{\Gm}^-{\theta}
\\
&&
\hrS_n \ar[rr]^-{\mathrm{root}} &&
\brS_n \ar[r] &
\LGradv(\bw3V_6) \times \P(V_6^\vee).
}
\end{equation*}
 {The labels $\bmu_2$, $\Gm$, and ``$\mathrm{root}$'' in the diagram mean that the corresponding arrows 
are a~$\bmu_2$-torsor, $\Gm$-torsors, and a root stack, respectively.} 

On $\brS_n$, we have the tautological family $(\ns,\cV_6,\cV_5,\cA)$ of Lagrangian data described in Proposition~\ref{proposition:cmlag-quotient}.\ 
Its pullback to the blow up $\bS$ is a family of Lagrangian data on $\bS$ of rank~\mbox{$n + 5$} avoiding decomposable vectors.\ 
Its special locus is the preimage of $\rS_{n-1}$, that is, the Cartier divisor $2E$.\ 
Since this divisor is a double, the construction of Section~\ref{subsection-from-lag-to-gm} applies: 
by Proposition~\ref{prop:equivalence-groupoids}, there exists a family of smooth normalized GM data $(\bS,\cW,\cV_6,\cV_5,\mu,\bq)$ 
with   special locus  $E$.\ We claim that this family is the pullback with respect to the blow up morphism $\beta \colon \bS \to \hhrS_n$.\ Since this is the blow up of a smooth scheme along a smooth center, it is enough to check
that all the bundles $\cW$, $\cV_6$, and $\cV_5$ restrict trivially to the fibers of the exceptional divisor $E$.

For the bundles $\cV_6$ and $\cV_5$, this is obvious, since they do not change
in the construction of Proposition~\ref{prop:equivalence-groupoids}.\ For $\cW$, it is a bit more complicated.\ This bundle is constructed in Lemma~\ref{lemma:bq-gm} and, according to Proposition~\ref{proposition-hecke-transform}, the restriction of $\cW$ to $E$ is a direct sum 
\begin{equation*}
\cW\vert_E \cong (\cW'_E / \cK_E) \oplus \cK_E(E).
\end{equation*}
The first summand $\cW'_E/\cK_E$ is isomorphic to the image 
of the restriction to $E$ of the pullback to $\bS$ of the map  
$\varphi \colon \cA \to \bw2\cV_5 \otimes (\cV_6/\cV_5)$.\ In particular, it is trivial on the fibers of $E$.\ The second summand comes by Proposition~\ref{proposition-hecke-transform}  with a natural isomorphism
\begin{equation*}
\bq_1 \colon \Sym^2(\cK_E(E)) \isomto \cV_6/\cV_5.
\end{equation*}
Its target is a line bundle trivial on the fibers of $E$, hence so is its source.\ Finally, the fibers of $E$ are projective spaces, hence a line bundle on a fiber, whose square is trivial, is trivial itself.\ Thus,~$\cK_E(E)$ is trivial on the fibers of $E$ and so is $\cW$.

We conclude that there is a family of GM data on $\hhrS_n$ whose pullback to $\bS$ is the family of GM data obtained from the family of Lagrangian data $(\ns,\cV_6,\cV_5,\cA)$ by the construction of Section~\ref{subsection-from-lag-to-gm}.\ Let us check that it is $\rG_n$-equivariant.\ 
The family $(\ns,\cV_6,\cV_5,\cA)$ is $\rG_n$-equivariant as a family of Lagrangian data 
(note, however, that it is not equivariant as a family of {linearized Lagrangian data; see Remark~\ref{remark:linear-lagrangian-lift}})
because, in Definition~\ref{definition:lagrangian-data} of a morphism of Lagrangian data, 
we ask for isomorphisms between \emph{projectivizations} of the appropriate bundles.\ The construction of Section~\ref{subsection-from-lag-to-gm} is natural, hence
the resulting family of GM data on the blow up~$\bS$ of $\hhrS_n$ is also $\rG_n$-equivariant.\ 

Finally, the pullback functor for the blow up~$\beta \colon \bS \to \hhrS_n$ is fully faithful, hence the resulting GM data on $\hhrS_n$ 
is $\rG_n$-equivariant (again, it is not equivariant as a family {of linearized GM data; see Remark~\ref{remark:linear-lift}}).\
Consequently, we obtain a family of smooth normalized GM data on the quotient stack $\hhrS_n/\rG_n \cong \hrS_n/\PGL(V_6)$, 
which gives the desired map to $\hrS_n/\PGL(V_6) \to \ucM_n$.

Let us construct the morphism in the opposite direction.\ 
Recall that $\ucM_n$ is a smooth Deligne--Mumford stack by Proposition~\ref{proposition:mgm-dm}.\ Let 
\begin{equation*}
S \lra \ucM_n
\end{equation*}
be an \'etale covering by a smooth scheme $S$, let $(\ns,\cW,\cV_6,\cV_5,\mu,\bq)$ be the corresponding family of smooth normalized GM data, and let $\widetilde{S}$ be the $\PGL(V_6)$-torsor associated with the rank-6 bundle $\cV_6$.\
The pullback of $\cV_6$ to $\widetilde{S}$ is trivial up to a twist.\
Replacing the bundles~$\cV_6$, $\cV_5$, and $\cW$ by appropriate twists, we can assume that 
$\cV_6 $ is the trivial bundle $V_6 \otimes \cO_{\widetilde{S}}$.

Let~$(\widetilde{S},\cV_6,\cV_5,\cA)$ be the family of Lagrangian data obtained from $({\widetilde{S}},\cW,\cV_6,\cV_5,\mu,\bq)$ 
by the construction of Proposition~\ref{proposition-gm-to-lag}.\
Since $\cV_6$ is trivial, we obtain a morphism 
\begin{equation*}
\widetilde{S} \lra \LGr(\bw3V_6) \times \P(V_6^\vee)
\end{equation*}
such that $\cA$ and $\cV_5$ are the pullbacks of the tautological bundles.\ 
Since $(\widetilde{S},\cV_6,\cV_5,\cA)$ has rank~$n+5$ and avoids decomposable vectors, 
this morphism factors through $\brS_n$.\ 
Let us show that it also factors through the stack $\hrS_n \to \brS_n$.

If $n = 5$, the stack $\hrS_n$ is the root stack of the section $\det(\varphi)$ of the line bundle~$\cL_5$,
so, by~\cite[Section~B.1]{agv}, it is enough to check that the pullback of the ideal generated by~$\det( {\varphi})$ is a square.\ 
Since this ideal defines the Lagrangian-special locus for the family~$(\widetilde{S},\cV_6,\cV_5,\cA)$,
it is, by Proposition~\ref{proposition-gm-to-lag},  the square of the ideal defining the GM-special locus on~$\widetilde{S}$.\
The universal property of the root stack gives the required factorization.

If $n < 5$, we apply Proposition~\ref{prop:factorization-through-hats}.\
Its assumptions are satisfied because $\widetilde{S}$ is smooth and {the locus $B(\bm_n)$
associated with} the map $\bm_n$ is equal to $\rS_{n-1}$ by Lemma~\ref{lemma:ideal-mn},
so, by Proposition~\ref{proposition-gm-to-lag}, 
its preimage in~$\widetilde{S}$ is set-theoretically equal to the GM special locus in $\widetilde{S}$, 
which by Lemma~\ref{eq:special-gerbe} has codimension at least 2 since $n < 5$.

Therefore, we obtain a morphism $\widetilde{S} \to \hrS_n$.\
Passing to   quotients by $\PGL(V_6)$, we obtain a morphism
\begin{equation*}
S = \widetilde{S}/\PGL(V_6) \lra \hrS_n/\PGL(V_6).
\end{equation*}
If we replace the \'etale covering $S \to \ucM_n$ by another \'etale covering $S' \to S \to \ucM_n$,
it is easy to see that the morphisms $S \to \hrS_n/\PGL(V_6)$ and $S' \to \hrS_n/\PGL(V_6)$ are compatible.\
Therefore, we obtain a morphism
\begin{equation*}
\ucM_n \lra \hrS_n/\PGL(V_6)
\end{equation*}
which is  inverse to the one constructed before.

 In view of Proposition~\ref{proposition:mgm-dm}, the only thing that remains to be proved is
 the separatedness of the stack $\ucM_n$.\
This follows from the fact that the scheme $\brS_n$ provides a covering of the stack $\brS_n/\PGL(V_6)$ in the smooth topology.\
The morphism $\hrS_n \to \brS_n$ induced by the morphism 
\begin{equation}
\label{eq:hsn-bsn-stacks}
\hrS_n/\PGL(V_6) \lra \brS_n/\PGL(V_6)
\end{equation} 
on this covering is proper by Corollary~\ref{corollary:hs-s-proper}
hence,  {by~\cite[\href{https://stacks.math.columbia.edu/tag/06TZ}{Lemma~06TZ}]{SP}} so is the morphism~\eqref{eq:hsn-bsn-stacks}.\
Consequently, the separatedness of $\brS_n/\PGL(V_6)$ (proved in Proposition~\ref{proposition:cmlag-quotient}) 
implies the separatedness of~$\ucM_n \cong \hrS_n/\PGL(V_6)$.
\end{proof}

One immediate consequence of Theorem~\ref{theorem:gm-global-quotient} is the following.

\begin{coro}
\label{corollary:gm-ord}
For $n \in \{3,4,5\}$, the stack $\ucM_{n, {\rm ord}}$ of ordinary smooth GM varieties of dimension~$n$ 
is isomorphic to the quotient stack $\rS_n/\PGL(V_6)$.\ 
Similarly, the stack ${\ucMtwo}$ of ordinary strongly smooth GM surfaces is isomorphic to the quotient stack $\rS_2/\PGL(V_6)$.\
 {These stacks are smooth separated Deligne--Mumford stacks of finite presentation over $\Q$.}
\end{coro}

\begin{proof}
The first part follows  from Theorem~\ref{theorem:gm-global-quotient}.\ The second part follows from  Remark~\ref{remark:ordinary-surfaces-equivalence}
and Proposition~\ref{proposition:cmlag-quotient}, since $\rS_1 = \varnothing$, hence $\brS_2 = \rS_2$.
\end{proof}

Theorem~\ref{theorem:gm-global-quotient} does not describe the stack $\ucM_6$ of smooth GM varieties of dimension~6.\
However, since every such variety is special, we have the following result.

\begin{coro}
\label{corollary:gm-6}
We have an equality of stacks $\ucM_6 = \ucM_{6, {\rm spe}}$.\  Therefore,~$\ucM_6$ is a $\bmu_2$-gerbe over $\ucM_{5, {\rm ord}} \cong \rS_5/\PGL(V_6)$.
\end{coro}

\begin{proof}
{The first assertion is obvious and the second follows from Lemma~\ref{eq:special-gerbe} and Corollary~\ref{corollary:gm-ord}}.
\end{proof}

\subsection{Coarse moduli spaces}
\label{subsection:coarse-moduli}

We use  the global quotient descriptions from previous sections to describe the coarse moduli spaces of GM varieties.\

 {Recall that for any $m \in \Z$, the line bundle $\cO(2m,6)$ on $\LGr(\bw3V_6) \times \P(V_6^\vee)$} admits a~$\PGL(V_6)$-linearization.\ 
This line bundle also admits a $\GL(V_6)$-linearization and, since for~$n \in \{3,4,5\}$,
the subgroup $\bmu_{3(5-n)} \subset \GL(V_6)$ acts trivially on it, 
this linearization induces a $\rG_n$-linearization (where $\rG_n = \GL(V_6)/\bmu_{3(5-n)}$ was defined in \eqref{eq:rg-n}).

\begin{coro}
\label{corollary:stability-gn}
 For each $n \in \{2,3,4,5\}$, the scheme $\hhrS_n$ consists of $\rG_n$-stable points for the~$\cO(2m,6)$-lineari\-zation.
\end{coro}

\begin{proof}
 {As in Corollary~\ref{corollary:stability}, this follows from Lemma~\ref{ssig} and~\cite[Proposition~2.18]{git} applied to 
 the morphism $\hhrS_n \to \LGradv(\bw3V_6)$.}
\end{proof}

 {In the next theorem, we prove that the stacks $\ucM_n$, $\ucM_{n, {\rm ord}}$, $\ucM_{n, {\rm spe}}$, and $\cmlag_{n+5}$ all admit coarse moduli spaces, 
which we denote by $\bcM_n$, $\bcM_{n, {\rm ord}}$, $\bcM_{n, {\rm spe}}$, and $\bcmlag_{n+5}$, respectively, and we describe them as GIT quotients.}

\begin{theo}
\label{theorem:gm-coarse}
{\rm(a)}  
For  $n\in\{3,4,5,6\}$, the respective coarse moduli spaces $\bcM_n$ and $\bcmlag_{n+5}$ of the stacks $\ucM_n$ and~$\cmlag_{n+5}$ are 
both isomorphic to the GIT quotient 
\begin{equation*}
\bcM_n \cong 
\bcmlag_{n+5} \cong
\brS_n\gquot\PGL(V_6)
\end{equation*}
taken with respect to the natural linearization of the line bundle $\cO(2m,6)$ for sufficiently large $m$.

{\rm(b)}  For  $n\in\{3,4,5 \}$, the respective coarse moduli spaces $\bcM_{n, {\rm ord}}$  and ${\bcMtwo}$ 
of the stacks $\ucM_{n, {\rm ord}}$ and ${\ucMtwo}$ of ordinary GM varieties 
are isomorphic to the GIT quotients 
\begin{equation*}
\bcM_{n, {\rm ord}} \cong \rS_n\gquot\PGL(V_6),
\qquad 
{\bcMtwo} \cong \rS_2\gquot\PGL(V_6).
\end{equation*}

{\rm(c)}  For $n \in \{3,4,5,6\}$, the coarse moduli space $\bcM_{n, {\rm spe}}$ of the stack $\ucM_{n, {\rm spe}}$ of special GM varieties 
is isomorphic to the GIT quotient 
\begin{equation*}
\bcM_{n, {\rm spe}} \cong \rS_{n-1}\gquot\PGL(V_6).
\end{equation*}
\end{theo}

\begin{proof}
We first prove part (a) for $n \in \{3,4,5\}$.\ Since, by Corollary~\ref{corollary:stability}, the scheme $\hhrS_n$ consists of stable points for the $\rG_n$-linearization of the bundle $\cO(2m,6)$, the morphism 
\begin{equation*}
\hhrS_n  {/ \rG_n} \lra \hhrS_n \gquot \rG_n
\end{equation*}
to the corresponding GIT quotient 
is a  {tame moduli space (\cite[Theorem~13.6]{Al}),
hence is a coarse moduli space (\cite[Remark~7.3]{Al}).\
It remains to recall that $\ucM_n \cong \hhrS_n / \rG_n$ by Theorem~\ref{theorem:gm-global-quotient}}.\ 
Similarly, the GIT quotient $\brS_n \gquot \PGL(V_6)$ is the coarse moduli space for the stack $\brS_n/\PGL(V_6)$,
which by Proposition~\ref{proposition:cmlag-quotient} is isomorphic to $\cmlag_{n+5}$.
So, it remains to identify the GIT quotients $\hhrS_n \gquot \rG_n$ and $\brS_n \gquot \PGL(V_6)$.

For this note that, since the group $\Gm/\bmu_{3(5-n)}$ acts on the algebra $\cO_{\brS_n}[\cL_n^{\pm1},\cR_n]$ via its grading, we have
\begin{equation*}
\hhrS_n \gquot (\Gm/\bmu_{3(5-n)}) = 
\Spec_{\brS_n}(\cO_{\brS_n}[\cL_n^{\pm1},\cR_n]) \gquot (\Gm/\bmu_{3(5-n)}) \cong 
\Spec_{\brS_n}(\cO_{\brS_n}) \cong 
\brS_n.
\end{equation*}
Moreover, we have $\PGL(V_6) \cong \rG_n / (\Gm/\bmu_{3(5-n)})$.\ Therefore,
\begin{equation*}
\hhrS_n \gquot \rG_n \cong
\Big(\, \hhrS_n \gquot (\Gm/\bmu_{3(5-n)}) \Big) \gquot \Big( \rG_n / (\Gm/\bmu_{3(5-n)}) \Big) \cong 
\brS_n \gquot \PGL(V_6).
\end{equation*}
This proves  {part (a)} for $n \in \{3,4,5\}$.

 {The proof of part~(b) is completely analogous,
using Corollary~\ref{corollary:gm-ord}  instead of Theorem~\ref{theorem:gm-global-quotient}.}


Let us prove part (c).\ 
By Lemma~\ref{eq:special-gerbe} and Remark~\ref{remark:special-rigidification}, 
the automorphism group scheme of each object of the stack $\ucM_{n,{\rm spe}}$ contains the constant group scheme $\bmu_2$ 
and the morphisms $\ucM_{n, {\rm spe}} \to \ucM_{n-1, {\rm ord}}$ for $n \ge 4$,
and~$\ucM_{3, {\rm spe}} \to \ucMtwo$ for $n=3$, are the $\bmu_2$-rigidifications.\
Therefore, by~\cite[Theorem~C.1.1(4)]{agv}, they have the same coarse moduli space and we conclude by part~(b).

Finally,  part (a) for $n = 6$ follows from Corollary~\ref{corollary:gm-6} and part (c).
\end{proof}

The coarse moduli {space} for smooth GM sixfolds (and for smooth ordinary GM fivefolds), 
which according to the above results is the GIT quotient $\rS_5/\PGL(V_6)$,
can also be constructed directly by following Mumford's proof for hypersurfaces in the projective space.\ Moreover, this approach gives the additional information that this moduli space is affine.

\begin{prop}\label{lemma:gm5-affine}
The coarse moduli space for smooth ordinary GM fivefolds and for smooth special GM sixfolds is affine.
\end{prop}

\begin{proof}
The argument is classical   (\cite[Proposition 4.2]{git}).\  A smooth ordinary GM fivefold  is by definition a smooth hypersurface of degree 2 in $G := \Gr(2,V_5)$.\ Inside  the projective space $\P(H^0(G,\cO_G(2)))$,
the subset of points  corresponding to sections whose zero locus in $G$ is singular is a hypersurface.\ This hypersurface is ample, hence its complement~$\P(H^0(G,\cO_G(2)))^0$ is affine and $\SL(V_5)$-invariant.\  The action of the reductive group~$\SL(V_5)$ on this affine set is linearizable and 
since the automorphism group of any smooth ordinary GM fivefold is finite (\cite[Proposition~3.21(c)]{DK}),  
the stabilizers are finite at points of $\P(H^0(G,\cO_G(2)))^0$, which is therefore contained in the stable locus.
  
The  {coarse} moduli space for smooth ordinary GM fivefolds is therefore a dense affine open subset 
of the projective irreducible 25-dimensional GIT quotient 
\begin{equation*}
\P(H^0(G,\cO_G(2)))\gquot\SL(V_5).
\end{equation*}
This proves the proposition.
\end{proof}

 The affineness properties can be also deduced from Theorem~\ref{theorem:gm-coarse}.\ Indeed, we have
\begin{equation*}
\bcM_{5, {\rm ord}} \cong \bcM_{6, {\rm {spe}}} \cong \rS_{5}\gquot\PGL(V_6)
\end{equation*}
and the {scheme} $\rS_5 = \bigl(\LGr(\bw3V_6) \times \P(V_6^\vee)\bigr) \setminus \bigl(\Sigma \cup \{ \det(\varphi) = 0 \}\bigr)$
is affine since the divisor~\mbox{$\Sigma \cup \{ \det(\varphi) = 0 \}$} in $\LGr(\bw3V_6) \times \P(V_6^\vee)$ is   ample.

\section{Applications}

In this section, we work over $\k = \C$.

\subsection{The period map}
\label{subsection:period-map}

The coarse moduli space~$\EPW$ for double EPW sextics was constructed in~\eqref{eq:m-epw}.\ 
It is an affine integral scheme of dimension $20$.\
The composition
\begin{equation*}
\pi_n \colon \ucM_n \to \bcM_n \isomlra \brS_n \gquot \PGL(V_6) \to \LGradv(\bw3V_6) \gquot \PGL(V_6) = \EPW
\end{equation*}
defines a morphism from the stack $\ucM_n$ of smooth GM varieties, or from its coarse moduli space $\bcM_n$, 
to the coarse moduli space~$\EPW$.

The period map
\begin{equation*}
\wpepw \colon \EPW \lhra \cD
\end{equation*}
for double EPW-sextics was constructed by O'Grady,  with values in the appropriate period  domain $\cD$; 
it is an open embedding  by   Verbitsky's Torelli Theorem (\cite[Theorem~1.3]{og6}).
 
\begin{prop}
\label{pm}
Assume $n \in\{4,6\}$.\ The map 
\begin{equation}
\label{eq:period-gm}
\wpgm := \wpepw \circ \pi_n \colon \ucM_n \lra \cD 
\end{equation}
 is the period map for GM varieties of dimension~$n$.
\end{prop}

\begin{proof}
This follows from~\cite[Proposition~5.27]{DK:periods}.
\end{proof}

\begin{rema}
GM varieties of dimension~$n \in\{3,5\}$ have   intermediate Jacobians that are 10-dimensional \ppavs\ (\cite[Proposition~3.1]{DK:periods}).\
We expect   their period maps to factor as
\begin{equation*}
\ucM_n \xrightarrow{\ \pi_n\ } \EPW  \lhook\joinrel\xrightarrow{\ \wpepw\ }  \cD \longrightarrow \cD/r_\cD \dra \cAb_{10},
\end{equation*}
where $r_\cD$ is the involution of the domain $\cD$ defined by O'Grady in \cite{og2} 
(geometrically, it corresponds to passing from an EPW sextic to its dual EPW sextic), 
$\cAb_{10}$ is the  coarse  moduli space for 10-dimensional \ppavs, and the   broken arrow is expected to be generically injective.\
To prove this factorization, however, one would need  an analogue of~\cite[Proposition~5.27]{DK:periods} for periods of odd-dimensional GM varieties.
\end{rema}

We can use Proposition \ref{pm} to describe  the fibers of $\wpgm$   for $n \in \{4,6\}$: they are the same as the fibers of $\pi_n$.\
The stacks $\widehat{Y}_{A^\perp}^{\ge \ell}$ were defined in~\eqref{eq:hya}.

\begin{coro}
\label{corollary:period-fiber}
 If $A \subset \bw3V_6$ is a Lagrangian subspace with no decomposable vectors, there is an isomorphism of stacks
\begin{equation}\label{yhat}
 \pi_4^{-1}([A]) \cong \rho_{A^\perp}^{-1}(Y_{A^\perp} \setminus Y_{A^\perp}^3) / \PGL(V_6)_A \subset \widehat{Y}^{\ge 1}_{A^\perp}/\PGL(V_6)_A,
 \end{equation}
where $\PGL(V_6)_A$ is the stabilizer of $A$ in $\PGL(V_6)$.\
Furthermore, the stack $\pi_6^{-1}([A])$ is a~$\bmu_2$-gerbe over $Y_{A^\perp}^0 / \PGL(V_6)_A$.\
In particular, {there are isomorphisms of coarse moduli spaces}
\begin{equation*}
 \pi_4^{-1}([A])_{\mathrm{coarse}} \cong (Y_{A^\perp} \setminus Y_{A^\perp}^3)\gquot\PGL(V_6)_A
\qquad\text{and}\qquad 
\pi_6^{-1}([A])_{\mathrm{coarse}} \cong Y_{A^\perp}^0\gquot\PGL(V_6)_A.
\end{equation*}
\end{coro}

\begin{proof}
By definition of {the scheme} $\brS_4$, the fiber of the map $\brS_4 \to \LGradv(\bw3V_6)$ over the point $[A]$
is the union 
$Y_{A^\perp} \setminus Y^{3}_{A^\perp} = Y^{1}_{A^\perp} \sqcup  Y^{2}_{A^\perp}$
of two EPW strata,
hence the fiber of {the composition} $\hrS_4 \to \brS_4 \to \LGradv(\bw3V_6)$ 
is isomorphic to $\rho_{A^\perp}^{-1}(Y_{A^\perp} \setminus Y_{A^\perp}^3) \subset \widehat{Y}^{\ge 1}_{A^\perp}$.\
Thus, the stack $\pi_4^{-1}([A])$ is isomorphic to the quotient stack~$\rho_{A^\perp}^{-1}(Y_{A^\perp} \setminus Y_{A^\perp}^3) / \PGL(V_6)_A$ 
and its coarse moduli space is $(Y_{A^\perp} \setminus Y_{A^\perp}^3)\gquot\PGL(V_6)_A$.

Similarly, Corollary~\ref{corollary:gm-6} identifies $\pi_6^{-1}([A])$ with a $\bmu_2$-gerbe over $Y_{A^\perp}^0 / \PGL(V_6)_A$
and its coarse moduli space with $Y_{A^\perp}^0\gquot\PGL(V_6)_A$.
\end{proof}

\subsection{Complete families of smooth GM varieties}
\label{subsection:complete-families}

Complete nonisotrivial families of smooth projective varieties  are hard to find in general  {(expecially those parameterized by rational curves)} and are interesting for this reason.\
Using our results, one can construct such families of GM varieties, some parameterized by the projective line.\

 {We start with a simple observation.}

\begin{lemm}
\label{lemma:a-constant}
Let $(\cX \to S,\cH)$ be a family of smooth GM varieties of dimension $n$ over a proper {reduced} connected scheme $S$.\ 
The map $\pi_n \colon S \to \EPW$ is constant.
\end{lemm}

\begin{proof}
This follows from the fact that $\EPW$ is affine.
\end{proof}

By Lemma~\ref{lemma:a-constant}, any    family of smooth GM varieties of dimension $n$ parameterized by a proper connected scheme $S$ corresponds 
to a fixed Lagrangian subspace $A \subset \bw3V_6$ and varying Pl\"ucker hyperplanes $V_5 \subset V_6$.\
In other words, repeating the argument of Corollary~\ref{corollary:period-fiber}, we see that such a family corresponds to a morphism 
\begin{equation*}
S \lra \pi_n^{-1}([A]) = (\rho^{5-n}_{A^\perp})^{-1} (Y_{A^\perp}^{\ge 5 - n} \setminus Y_{A^\perp}^{\ge 5 - n}) /\PGL(V_6)_A \subset  \widehat{Y}^{\ge 5-n}_{A^\perp}/\PGL(V_6)_A.
\end{equation*}

{The following result can be used to construct such a map.}

\begin{prop}
\label{proposition:explicit-family}
Let $n\in\{3,4,5,6\}$ and let $S$ be a connected {reduced} scheme.\ 
Assume that $f \colon S \to Y^{\ge 5-n}_{A^\perp}$ is a nonconstant morphism such that 
$f^{-1}(Y^{\ge 7-n}_{A^\perp}) = \vide$ and that
\begin{equation}
\label{eq:gm-family-conditions}
f^{-1}(Y^{\ge 6-n}_{A^\perp}) \hbox{ is equal to $2E$ for some Cartier divisor $E$ on $S$.}
\end{equation} 
Then there is a nonisotrivial family of smooth GM varieties $\cX \to S$ of dimension $n$.
\end{prop}

\begin{proof}
The family $\cX \to S$ exists by Corollary~\ref{corollary:constructing-families-gm}: consider the family of Lagrangian data
on $S$ given by the trivial bundles $\cV_6 = \bw3V_6 \otimes \cO_S$ and $\cA = A \otimes \cO_S$, and take for  $\cV_5$  the pullback
of the tautological rank-5 bundle on $\P(V_6^\vee)$ via the map $S \xrightarrow{\ f\ } Y^{\ge 5 - n}_{A^\perp} \hookrightarrow \P(V_6^\vee)$.\

This family is not isotrivial, because the corresponding map from~$S$ to the coarse moduli space $\bcM_n$ is nonconstant: 
this map is  the composition of $f$ with the quotient morphism~{$Y^{\ge 5-n}_{A^\perp} \to Y^{\ge 5-n}_{A^\perp}\gquot\PGL(V_6)_A$}; 
since the group $\PGL(V_6)_A$ is  finite and~\mbox{$\dim(f(S)) > 0$}, this is clear.
\end{proof}

It is not easy to find a map satisfying the condition~\eqref{eq:gm-family-conditions}.\ 
Sometimes, a double covering trick helps.

\begin{exam}
Let $L \subset \P(V_6^\vee)$ be a line such that $L \not\subset Y_{A^\perp}^{\ge 1}$ and $L \cap Y_{A^\perp}^{\ge 2} = \vide$.\
Then $L \cap Y_{A^\perp}^{\ge 1}$ is a divisor of degree 6 (because $Y_{A^\perp}^{\ge 1}$ is a sextic hypersurface).\
Let $\widetilde{L} \to L$ be the normalization of the double cover of $L$ branched over $L \cap Y_{A^\perp}^{\ge 1}$.\
Then,
\begin{itemize}
\item 
if the intersection $L \cap Y_{A^\perp}^{\ge 1}$ is transverse,  $\widetilde{L}$ is an integral curve of genus~2;
\item 
if the intersection $L \cap Y_{A^\perp}^{\ge 1}$ has exactly one nonreduced point, and its multiplicity is~2 or~3,  
$\widetilde{L}$ is an integral curve of genus~1;
\item 
in all other cases, each component of~$\widetilde{L} $ is isomorphic to $ \P^1$.
\end{itemize}
A general line falls into the first case.\  
A general line tangent to $Y_{A^\perp}^{\ge 1}$ at a general point falls into the second case.\    
Bitangent lines to $Y_{A^\perp}^{\ge 1}$ (of which there is a 6-dimensional family) fall into the third case.

Applying Proposition~\ref{proposition:explicit-family} to any of these families, 
we obtain a family of smooth GM varieties of dimension~5 over $\widetilde{L}$.\
It is not isotrivial  by Proposition~\ref{proposition:explicit-family}.
 \end{exam}

To construct families of GM varieties, one can also  apply  directly {Corollary~\ref{corollary:period-fiber}}.

\begin{exam}
Assume $Y_{A^\perp}^{\ge 3} = \emptyset$.\
There is a family of smooth GM fourfolds of maximal variation
parameterized by the double EPW sextic $\widetilde{Y}_{A^\perp}$.\
Indeed, by Corollary~\ref{corollary:period-fiber}, there is a   map
 \begin{equation*}
\widetilde{Y}_{A^\perp} \lra \widetilde{Y}_{A^\perp}  / \bmu_2 = \widehat{Y}_{A^\perp} \lra \widehat{Y}_{A^\perp}/\PGL(V_6)_A = \pi_4^{-1}([A]) \lhra \ucM_4.
 \end{equation*} 
Since any smooth double EPW sextic contains a uniruled divisor
(the Gromov--Witten invariants computed in  \cite{ob} include the degree of the divisor spanned by deformations of a rational curve of minimal degree on any smooth double EPW sextic, 
and this degree is nonzero), 
hence many rational curves, one obtains smooth nonisotrivial families of GM fourfolds  parameterized by~$\P^1$.
\end{exam}

\begin{exam}
Assume $Y_{A^\perp}^{\ge 3} = \emptyset$.\
As in the previous example, we can pull back the universal family of GM threefolds by the composition (see \eqref{yhat} for the notation)
\begin{equation*}
{Y}_{A^\perp}^{\ge 2} \isomlra 
\widehat{Y}_{A^\perp}^{\ge 2} \lra 
\widehat{Y}_{A^\perp}^{\ge 2}/\PGL(V_6)_A \lhra
\hrS_3/\PGL(V_6) \cong \ucM_3
\end{equation*}
and obtain a  family of smooth ordinary GM threefolds with maximal variation 
parameterized by the projective surface  ${Y}_{A^\perp}^{\ge 2}$.

When $A$ is general, the cotangent bundle of ${Y}_{A^\perp}^{\ge 2}$ is  globally generated  (\cite[Corollary~7.3]{dim})  
hence this surface contains no rational curves.\ 
Any Lagrangian $A$ with no decomposable vectors such that ${Y}_{A^\perp}^{\ge 2}$ contains a rational curve and $Y_{A^\perp}^{\ge 3} = \emptyset$ 
would give rise to a smooth nonisotrivial families of GM threefolds parameterized by~$\P^1$, but we do not know   any such Lagrangian.
\end{exam}

\appendix

\section{The generalized root construction}
\label{section:generalized-root-stack}

We discuss a  generalization of the root stack construction of~\cite{agv} which is also   
a particular case of the \emph{canonical stack} construction, as defined (under another name) 
in~\cite[Note~2.9 and proof of Proposition~2.8]{Vistoli} and developed in~\cite{GS}.

Let $S$ be a normal irreducible scheme.\
Let $\cR$ be a reflexive sheaf  of rank~1 on  $S$  such that  {the reflexive hull }
\begin{equation*}
\cL := (\cR \otimes \cR)^{\vee\vee}
\end{equation*}
 is a line bundle.\
If 
$\cL \cong \cM^{\otimes 2}$ for some line bundle~$\cM$ on $S$,
there is a scheme
\begin{equation}
\label{eq:wts}
\widetilde{S} := \Spec(\cO_S \oplus (\cM^{{-1}} \otimes \cR))
\end{equation}
equipped with a map $\tilde\rho \colon \widetilde{S} \to S$ 
which is finite of degree 2 and \'etale over the locally free locus of $\cR$ (\cite[Proposition~2.5]{DK:coverings}), 
and an involution $\tau$ of $ \widetilde{S}$ over $S$.\
Let 
\begin{equation*}
\widehat{S} := \widetilde{S} / \bmu_2(\tau)
\end{equation*}
be the quotient stack with respect to the $\bmu_2$-action on $\widetilde{S}$ generated by~$\tau$.\
There is a natural map $\hat\rho \colon \widehat{S} \to S$ which is an isomorphism over the locally free locus of $\cR$; over~$\Sing(\cR)$,  {it is a nilpotent thickening of a $\bmu_2$-gerbe over $\Sing(\cR)$}.

We want to show that the construction that produces the stack $\widehat{S}$ from~$S$
is more natural in a sense than the construction of the double covering.\ 
In particular, it does not require the existence (hence nor the choice)  of a square root $\cM$ of $\cL$.

The construction is very simple.\ 
 {Slightly generalizing the above setup, we assume that}~$\cR$ is a reflexive sheaf of rank~1 on  $S$ 
such that the sheaf $(\cR \otimes \cR)^{\vee\vee}$ is locally free
and let 
\begin{equation*}
\bm \colon \cR \otimes \cR \lra \cL
\end{equation*}
be a nonzero morphism into a line bundle $\cL$.\ Consider the  {quasicoherent} sheaf 
 \begin{equation}
\label{eq:os-l-r}
\begin{split} 
\cO_S[\cL^{\pm1},\cR] &:=  \bigoplus_{i \in \Z} (\cL^i \oplus (\cL^i \otimes \cR))  \\
&\ \cong 
 {}\cdots \oplus \cL^{-1} \oplus (\cL^{-1} \otimes \cR) \oplus 
\cO_S \oplus \cR \oplus \cL \oplus (\cL \otimes \cR) \oplus \cdots{}
\end{split}
\end{equation}
with the $\Z$-grading defined by
\begin{equation*}
\deg(\cL^i) = 2i
\quad \textnormal{and}\quad
\deg(\cL^i \otimes \cR) = 2i + 1.
\end{equation*}

\begin{lemm}
The morphism $\bm$ induces on $\cO_S[\cL^{\pm 1},\cR]$ 
a  commutative associative $\cO_S$-algebra structure.
\end{lemm}

\begin{proof}
There is a natural associative algebra structure on the sheaf
\begin{equation*}
\cO_S[\cR] :=
\cO_S \oplus \cR \oplus (\cR \otimes \cR)^{\vee\vee} \oplus 
(\cR \otimes \cR \otimes \cR)^{\vee\vee} \oplus 
(\cR \otimes \cR \otimes \cR \otimes \cR)^{\vee\vee} \oplus \cdots
\end{equation*}
(the associativity follows from the functoriality of the reflexive hull).\
It is also commutative since the automorphism of $(\cR \otimes \cR)^{\vee\vee}$ induced by the transposition of $\cR \otimes \cR$
is an automorphism of a reflexive sheaf which is the identity on the locally free locus of $\cR$, hence is itself the identity.\
Finally, the morphism $\bm$ induces a morphism 
$\cO_S[(\cR \otimes \cR)^{\vee\vee}] \to \cO_S[\cL] \hookrightarrow \cO_S[\cL^{\pm 1}]$ of commutative associative algebras  and  {we have}
\begin{equation*}
\cO_S[\cL^{\pm1},\cR] = \cO_S[\cR] \otimes_{\cO_S[(\cR \otimes \cR)^{\vee\vee}]} \cO_S[\cL^{\pm1}],
\end{equation*}
because $(\cR^{\otimes 2m})^{\vee\vee} \cong ((\cR \otimes \cR)^{\vee\vee})^{\otimes m}$
and $(\cR^{\otimes (2m+1)})^{\vee\vee} \cong ((\cR \otimes \cR)^{\vee\vee})^{\otimes m} \otimes \cR$.
\end{proof}

Consider the quotient stack
\begin{equation}
\label{eq:generalized-root-stack}
\widehat{S} := \left(\Spec_S(\cO_S[\cL^{\pm1},\cR])\right)/\Gm
\end{equation} 
 for the $\Gm$-action corresponding to  the grading defined above.\  We will call this stack {\sf the root stack of $(\cR,\bm)$}.

\begin{prop}
\label{proposition:ts-hs}
Let $\cM$ be a line bundle  on $S$ with an isomorphism $\cL \cong \cM^{\otimes 2}$.\
Consider the double covering~\eqref{eq:wts} of $S$ corresponding to the morphism
\begin{equation*}
(\cM^{-1} \otimes \cR) \otimes (\cM^{-1} \otimes \cR) \cong 
\cM^{-2} \otimes (\cR \otimes \cR) \xrightarrow{\ \bm\ }
\cM^{-2} \otimes \cL \cong
\cO_S.
\end{equation*}
There is a natural isomorphism  {of stacks} $\widetilde{S}/\bmu_2 \cong \widehat{S}$.
\end{prop}

\begin{proof}
Consider the sheaf of {commutative} algebras 
\begin{equation}
\label{eq:cos-cm-cr}
\cO_S[\cM^{\pm1},\cR] := 
\cO_S[\cL^{\pm1},\cR] \oplus (\cM \otimes \cO_S[\cL^{\pm1},\cR])
\end{equation} 
with multiplication induced by the multiplication in the algebra $\cO_S[\cL^{\pm1},\cR]$ 
and the isomorphism $\cM \otimes \cM \isomto \cL  \hookrightarrow \cO_S[\cL^{\pm 1},\cR]$.\
This algebra carries a natural $((\Z/2) \oplus \Z)$-grading induced by the $\Z$-grading of $\cO_S[\cL^{\pm1},\cR]$ and
\begin{equation*}
\deg(\cM) = (1,1).
\end{equation*}
This grading corresponds to a $(\bmu_2 \times \Gm)$-action on $\Spec_S(\cO_S[\cM^{\pm1},\cR])$.

By the definition~\eqref{eq:cos-cm-cr} of the algebra $\cO_S[\cM^{\pm1},\cR]$,
the $\bmu_2$-action on $\Spec_S(\cO_S[\cM^{\pm1},\cR])$ is free   and the invariant part is equal to $\cO_S[\cL^{\pm1},\cR]$.\
Therefore,
we get an \'etale double covering
\begin{equation*}
\Spec_S(\cO_S[\cM^{\pm1},\cR]) \lra \Spec_S(\cO_S[\cL^{\pm1},\cR]).
\end{equation*}
{On the other hand}, forgetting the $\Z/2$-grading and keeping the $\Z$-grading, we see that 
the $i$-th component of the algebra is isomorphic to $\cM^i \otimes (\cO_S \oplus (\cM^{-1} \otimes \cR))$.\
Therefore, the corresponding $\Gm$-action on $\Spec_S(\cO_S[\cM^{\pm1},\cR])$ is also free
 and the invariant part is equal to $\cO_S \oplus (\cM^{-1} \otimes \cR)$.\ Therefore,
we get a $\Gm$-torsor
\begin{equation*}
\Spec_S(\cO_S[\cM^{\pm1},\cR]) \lra \Spec_S(\cO \oplus (\cM^{-1} \otimes \cR)) = \widetilde{S}.
\end{equation*}
 Combining these maps, we obtain a diagram
\begin{equation*}
\xymatrix{
\Spec_S(\cO_S[\cM^{\pm1},\cR]) \ar[r]_-{\bmu_2} \ar[d]^{\Gm} &
\Spec_S(\cO_S[\cL^{\pm1},\cR]) 
\\
\widetilde{S} ,
}
\end{equation*}
where the horizontal arrow is  {a $\bmu_2$-torsor} and the vertical arrow is a $\Gm$-torsor.\
It induces  {a~$\bmu_2$-torsor} 
\begin{equation*}
\widetilde{S} \lra \Spec_S(\cO_S[\cL^{\pm1},\cR]) / \Gm = \widehat{S}.
\end{equation*}
It follows that $\widehat{S} \cong \widetilde{S}/\bmu_2$ and since $\deg(\cM^{-1} \otimes \cR) = (1,0) \in (\Z/2) \oplus \Z$,
the action of~$\bmu_2$ on~$\widetilde{S}$ is induced by the involution of the double covering $\widetilde{S} \to S$.
\end{proof}

\begin{rema}
\label{remark:usual-root-stack}
In the case where $\cR = \cO_S$ and the morphism $\bm \colon \cR \otimes \cR \to \cL$ is given by a global section $s$ of  {the} line bundle $\cL$,
the stack $\widehat{S}$ coincides with the usual root stack $\sqrt{(\cL,s)/S}$ defined in~\cite[Section~B.2]{agv}: this follows from Proposition~\ref{proposition:ts-hs} applied (\'etale locally) 
to the double covering $\widetilde{S}$ of $S$ branched over the zero locus of $s$.
\end{rema}

 We now discuss some properties of the root stack $\widehat{S}$.

\begin{coro}
\label{corollary:hs-s-proper}
The natural morphism $\hat\rho\colon \widehat{S} \to S$ is proper.
\end{coro}
\begin{proof}
The question is local over $S$, so we may assume we are in the setup of Proposition~\ref{proposition:ts-hs} and $\widehat{S} = \widetilde{S}/\bmu_2$.\
Then $\widetilde{S}$ is proper over $S$ by~\eqref{eq:wts}, hence so is $\widehat{S}$ 
 {by~\cite[\href{https://stacks.math.columbia.edu/tag/0CQK}{Lemma~0CQK}]{SP}.}
\end{proof}

Consider the subscheme $B(\bm) \subset S$ defined by the ideal image of the map
\begin{equation}
\label{eq:branch}
\cL^{-1} \otimes \cR \otimes \cR \xrightarrow{\ \bm\ } \cL^{-1} \otimes \cL = \cO_S.
\end{equation}
 Proposition~\ref{proposition:ts-hs} implies the main properties of the root stack $\widehat{S}$.

\begin{coro}
The natural morphism $\hat\rho\colon \widehat{S} \to S$ is an isomorphism over the complement of~$B(\bm) \subset S$ 
 and is a  {nilpotent thickening of a} $\bmu_2$-gerbe over~$B(\bm)$.
\end{coro}

\begin{proof}
Set $Z := B(\bm)$.\
Over $S \setminus Z$, we have an isomorphism $\cR \otimes \cR \cong \cL$, hence $\cR$ is invertible.\ 
The double covering $\widetilde{S} \to S$ (which exists locally over $S$) is therefore  \'etale over~$S \setminus Z$, 
hence its quotient stack $\widehat{S} \to S$ is an isomorphism over $S \setminus Z$.

On the other hand, over $Z$, the multiplication in the algebra defining $\widetilde{S}$ is zero, 
hence there is a natural embedding $Z \to \widetilde{S}$ over $Z \subset S$, 
 {and the schematic preimage of $Z \subset S$ in~$\widetilde{S}$ is a nilpotent thickening of $Z \subset \widetilde{S}$}.\
The $\bmu_2$-action on $Z \subset \widetilde{S}$ is trivial, hence gives a $\bmu_2$-gerbe~$Z/\bmu_2 \hookrightarrow \widetilde{S}/\bmu_2$ over $Z$
 {and the preimage of $Z$ in $\widehat{S}$ is its nilpotent thickening}.
\end{proof}

The following property of the stack $\widehat{S}$ is quite useful.\
 It is similar to the universal property of canonical smooth Deligne--Mumford stacks proved in~\cite[Theorem~4.6]{FMN} (see also~\cite[Lemma~2.4.1]{AV}).

\begin{prop}
\label{prop:factorization-through-hats}
Let $\widehat{S}  \xrightarrow{ \ \hat\rho\  } S$ be the root stack defined by~\eqref{eq:generalized-root-stack}
and let $B(\bm) \subset S$ be the subscheme defined by~\eqref{eq:branch}.\
Let $T$ be a smooth scheme and let $f \colon T \to S$ be a morphism 
such that $\codim_T(f^{-1}(B(\bm))) \ge 2$.\
There is a unique factorization 
\begin{equation*}
f \colon T\xrightarrow{ \  \ }  \widehat{S}  \xrightarrow{ \ \hat\rho\  } S.
\end{equation*}
\end{prop}

\begin{proof}
Consider the scheme 
\begin{equation*}
T \times_S \Spec_S(\cO_S[\cL^{\pm 1},\cR]) \cong \Spec_T(\cO_T[f^*\cL^{\pm1},f^*\cR]).
\end{equation*}
The sheaf $\cM := (f^*\cR)^{\vee\vee}$
is a rank-1 reflexive sheaf on a smooth scheme $T$, hence is a line bundle.\
Therefore, there is a natural epimorphism
\begin{equation*}
f^*\cR \twoheadrightarrow \cM \otimes \cI,
\end{equation*}
where $\cI$ is an ideal sheaf such that the support of $\cO/\cI$ has codimension at least $2$.\ 
Furthermore, the morphism $\bm\colon \cR \otimes \cR \to \cL$ induces the morphism
$f^*\bm \colon f^*\cR \otimes f^*\cR \to f^*\cL$ which factors through  {the tensor product of the reflexive hulls}
\begin{equation*}
f^*\cR \otimes f^*\cR \to \cM \otimes \cM \to f^*\cL,
\end{equation*}
and is an isomorphism away from $f^{-1}(B(\bm))$ and   the support of $\cO/\cI$, that is, in codimension~1.\
Since $T$ is smooth, it follows that 
\begin{equation*}
f^*\cL \cong \cM^{\otimes2}.
\end{equation*}
Therefore, we have a natural morphism of graded $\cO_T$-algebras
\begin{equation*}
 \cO_T[f^*\cL^{\pm1},f^*\cR] \twoheadrightarrow \cO_T[\cM^{\pm 2},\cM \otimes \cI] \hookrightarrow \cO_T[\cM^{\pm 2},\cM] = \cO_T[\cM^{\pm1}].
\end{equation*}
 It induces a morphism
\begin{multline*}
\qquad\Spec_T(\cO_T[\cM^{\pm1}]) \to \Spec_T(\cO_T[f^*\cL^{\pm1},f^*\cR]) \\ = 
T \times_S \Spec_S(\cO_S[\cL^{\pm 1},\cR]) \to
\Spec_S(\cO_S[\cL^{\pm 1},\cR])\qquad
\end{multline*}
compatible with the $\Gm$-actions corresponding to the gradings of the algebras.\
Since the source is a $\Gm$-torsor over $T$, passing to the quotients by $\Gm$, we obtain a morphism $T \to \widehat{S}$.\
By construction, the composition $T \to \widehat{S} \to S$ is equal to $f$ and the constructed morphism is unique with this property.
\end{proof}

\end{document}